\documentclass[preprint,10pt]{imsart}

\usepackage{ amsmath, amsfonts, amssymb, amsthm}

\usepackage[T1]{fontenc}
\usepackage[english]{babel}
\usepackage{color}
\usepackage{geometry}
\usepackage[latin1]{inputenc}

\fontsize{11pt}{14.5pt}\selectfont

\def\Dy#1{\Frac{\partial #1}{\partial y}}

\def\Dy_1y_1#1{\Frac{\partial^2 #1}{\partial y_1^2}}

\def\reff#1{{\rm(\ref{#1})}}

\newtheorem{Theorem}{Theorem}[part]
\newtheorem{Definition}{Definition}[part]
\newtheorem{Proposition}{Proposition}[part]
\newtheorem{Assumption}{Assumption}[part]
\newtheorem{Lemma}{Lemma}[part]

\newtheorem{Remark}{Remark}[part]

\makeatletter \@addtoreset{equation}{section}

\@addtoreset{Definition}{section}

\@addtoreset{Theorem}{section}

\@addtoreset{Proposition}{section}

\@addtoreset{Assumption}{section}

\@addtoreset{Corollary}{section}

\@addtoreset{Lemma}{section}

\@addtoreset{Remark}{section}

\@addtoreset{Example}{section}

\makeatother \makeatletter

\def \Int{\displaystyle\int}
\def \Frac{\displaystyle\frac}
\def \Inf{\displaystyle\inf}
\def \Sup{\displaystyle\sup}
\def \Lim{\displaystyle\lim}

\def \D{\mathbb{D}}
\def \E{\mathbb{E}}
\def \F{\mathbb{F}}
\def \G{\mathbb{G}}
\def \H{\mathbb{H}}
\def \I{\mathbb{I}}
\def \L{\mathbb{L}}
\def \N{\mathbb{N}}
\def \P{\mathbb{P}}
\def \R{\mathbb{R}}
\def \Q{\mathbb{Q}}

\def \S{\mathbb{S}}
\def \W{\overleftarrow{W}}

\def \[{[\,\!\![}
\def \]{]\,\!\!]}

\def \1{{\bf 1}}

\def \proof{{\noindent \bf Proof. }}
\def \ep{\hbox{ }\hfill$\Box$}

\newcommand{\No}[1]{\left\|#1\right\|}     % Norm
\newcommand{\abs}[1]{\left|#1\right|}     % absolute value
\def\reff#1{{\rm(\ref{#1})}}

\def\Bc{{\cal B}}

\def\Ec{{\cal E}}
\def\Fc{{\cal F}}
\def\Gc{{\cal G}}

\def\Lc{{\cal L}}
\def\Mc{{\cal M}}

\def\Lc{{\cal L}}
\def\Pc{{\cal P}}

\def\Yc{{\cal Y}}
\def\Zc{{\cal Z}}

\def\indi{{\bf 1}}
\def\Qc{{\cal Q}}
\def\eps{\varepsilon}

\begin{document}
\begin{frontmatter}
\title{Probabilistic interpretation for solutions of fully nonlinear stochastic PDEs}

%\runtitle{}

\date{\today}

\author{\fnms{Anis}
 \snm{MATOUSSI}\corref{}\ead[label=e1]{anis.matoussi@univ-lemans.fr}\thanksref{t3}}
 \thankstext{t3}{Research partly supported by the Chair {\it Financial Risks} of the {\it Risk Foundation} sponsored by Soci\'et\'e G\'en\'erale, the Chair {\it Derivatives of the Future} sponsored by the {F\'ed\'eration Bancaire Fran\c{c}aise}, and the Chair {\it Finance and Sustainable Development} sponsored by EDF and Calyon }
\address{
 Universit\'e du Maine \\
 Institut du Risque et de l'Assurance\\
Laboratoire Manceau de Math\'ematiques\\\printead{e1}
 }
 
\author{\fnms{Dylan}
 \snm{POSSAMA\"I}\corref{}\ead[label=e2]{possamai@ceremade.dauphine.fr}}
\address{CEREMADE Universit\'e Paris--Dauphine\\PSL Research University, CNRS\\ 
75016 Paris, France\\
\printead{e2}
}

\author{\fnms{Wissal}
 \snm{SABBAGH}\corref{}\ead[label=e3]{wissal.sabbagh@univ-lemans.fr}}
\address{Universit\'e du Maine\\
 Institut du Risque et de l'Assurance\\
 Laboratoire Manceau de Math\'ematiques\\
\printead{e3}
}

\runauthor{A. Matoussi, D. Possama\"i, W. Sabbagh}

\vspace{3mm}

%***************************
%********Abstract**************
%***************************
\begin{abstract}
In this article, we  propose a wellposedness  theory for a  class of second order backward doubly stochastic differential equation ($2$BDSDE). We prove existence and uniqueness of the solution under a Lipschitz type assumption on the generator, and we investigate the links between the  2BDSDEs and a class of  parabolic fully nonlinear Stochastic PDEs. Precisely, we  show that  the Markovian solution of  2BDSDEs provide a probabilistic  interpretation of the classical and stochastic viscosity  solution of  Fully nonlinear SPDEs.  
 
 \end{abstract}
\end{frontmatter}

\section{Introduction}
The starting point of this work is the following parabolic fully-nonlinear stochastic partial differential equation (SPDE for short)
\begin{equation}
\begin{split}
\label{SPDE1}
 du_t(x) + H(t,x,u_t(x),Du_t(,x),D^2u_t(x)) \, dt  + g(t,x,u_t(x), Du_t(x) )\circ d\W_t = 0, t\in[0,T],\ u_T=\Phi,
\end{split}
\end{equation}
where $g$ and
$H$ are given nonlinear functions. The differential term integrated with respect to $d\W_t$
refers to the backward stochastic integral against a finite-dimensional Brownian motion on some probability space $\big(\Omega, \mathcal{F},\mathbb{P}, (W_t)_{t\geq 0} \big)$. We use the backward notation because our approach is fundamentally based on the doubly stochastic framework introduced in the seminal paper by Pardoux and Peng \cite{pp1994}.

\vspace{0.3em}
\noindent The class of stochastic PDEs as in \eqref{SPDE1} and their extensions is an important one, since it arises in a number of applications, ranging from asymptotic limits of partial differential equations (PDEs for short) with rapid (mixing) oscillations in time, phase transitions and front propagation in random media with random normal velocities, filtering and stochastic control with partial observations, path--wise stochastic control theory, mathematical finance...  The main difficulties with equations like \eqref{SPDE1} are threefold.
\begin{itemize}
\item[$(i)$] Even in the deterministic case, there are no global smooth solutions in general.

\item[$(ii)$] Their fully nonlinear character seems to make them inaccessible to the classical martingale theory employed for the linear case

\item[$(iii)$] Even if smooth solutions were to exist, the equations cannot be described in a point--wise sense, because of the everywhere lack of differentiability of the Brownian paths.
\end{itemize}
\vspace{0.3em}
\noindent The starting point of the theory of SPDEs was with classical solutions in a linear context, wellposedness results having been obtained notably by Pardoux \cite{pardouxt1980stochastic}, Dawson \cite{dawson1972stochastic}, Ichikawa \cite{ichikawa1978linear} or Krylov and Rozovski{\u\i} \cite{krylov1977cauchy}. Extensions have been obtained later, notably by Pardoux and Peng \cite{pp1994} (see also Krylov and Rozovski{\u\i} \cite{krylov1981stochastic} or Bally and Matoussi \cite{BM01}) by introducing backward doubly stochastic differential equation (BDSDE for short), which allowed them to give a nonlinear Feynman--Kac's formula for the following class of semi--linear SPDEs 
\begin{equation}
\begin{split}
\label{semilinearSPDE}
 du_t(x) + \left[ \Lc u_t (x) + f (t,x,u_t(x), Du_t(x))\right] dt  + g(t,x,u_t(x), Du_t(x) )\circ d\W_t = 0 ,
\end{split}
\end{equation}
 where  $ \Lc$ is a linear second order diffusion operator and $f$ is a given nonlinear function.  The theory of BDSDE has then been extended in several directions, notably by Matoussi and Scheutzow \cite{MS02} who considered a class of BDSDE where the nonlinear noise term is given by the more general It\^o-Kunita's stochastic integral, thus allowing them to give a probabilistic interpretation of classical and Sobolev's solutions of semi--linear parabolic SPDEs driven by space-time white noise. However, the notion of viscosity solution to SPDE has remained a quite difficult and evasive subject throughout the years. Our aim in this paper, as will be detailed below, is to give a definition of a generalization of BDSDEs allowing for a probabilistic representation of solutions to fully nonlinear SPDEs, the appropriate notion of equation being that of second--order BDSDEs. Before going into details, let us review the associated literature.
 
\paragraph{Literature review for viscosity solution of SPDEs.}  Stochastic viscosity solutions for SPDEs were first introduced by Lions and Souganidis in their seminal papers \cite{lion:soug:98, lion:soug:00, lion:soug:01}. They used the so-called "stochastic characteristics" to remove the stochastic integrals from the SPDEs and thus transform them into PDEs with random coefficients. A few years later, Buckdahn and Ma \cite{buck:ma:10a,buck:ma:10b} considered a related but different class of semi--linear SPDEs, and studied them in light of the earlier results of Lions and Souganidis, giving in addition a probabilistic interpretation of such equation via BDSDE, but only in the case where the intensity of the noise $g$ \eqref{semilinearSPDE} did not depend on the gradient of the solution. They used the so--called Doss-Sussmann transformation and stochastic diffeomorphism flow technics to once more convert the semi--linear SPDEs into PDEs with random coefficients. This transformation was used again for non-standard optimal control problems in their paper \cite{buckdahn2007pathwise}, and also by Diehl and Friz \cite{diehl:friz:12} to solve semi--linear SPDEs and the associated BDSDEs driven by rough drivers. The case of fully nonlinear SPDEs was also considered by Buckdahn and Ma \cite{buckdahn2002pathwise}, still in the context of so-called stochastic viscosity solutions, and using a new kind of Taylor expansion for It\^o-type random fields. One had then to wait for ten years to see new progresses been made. Hence, Gubinelli, Tindel and Torrecilla \cite{Gubi:Tind:Torr:14} proposed a new definition of viscosity solutions to fully nonlinear PDEs driven by a rough path via appropriate notions of test functions and rough jets. These objects were defined as a controlled processes with respect to the driving rough path, and the authors showed that their notion of solution was compatible with the seminal results of Lions and Souganidis  \cite{lion:soug:98, lion:soug:00, lion:soug:01} and with the recent results of Caruana, Friz and Oberhauser \cite{Caru:Friz:Ober:11} on fully non--linear SPDEs driven with rough drivers. Independently, a series of papers involving Buckdahn, Bulla, Ma and Zhang \cite{buck2011pathwise,buckdahn2015pathwise,Buck:Ma:Zhan:15} proposed yet another alternative definition, based once more on pathwise Taylor expansions for random fields (like in \cite{buckdahn2002pathwise}), but now in the context of the functional It\^o calculus of Dupire, and by identifying the solution of the SPDE to the solution of a path-dependent PDE. Finally, the recent contribution of Friz, Gassiat, Lions and Souganidis \cite{friz2016eikonal} extends the notion of path--wise viscosity solutions for Eikonal equations having quadratic Hamiltonians.

\vspace{0.5em}
\paragraph{Literature review for $2$BSDES.}  Motivated by numerical methods for fully nonlinear PDEs (the case when $g$ is identically null in the equation \eqref{SPDE1}), second order BSDEs (2BSDEs for short) were introduced by Cheridito, Soner, Touzi and Victoir in \cite{CSTV07}. Then Soner, Touzi and Zhang \cite{STZ10} proposed a new formulation and obtained a complete theory of existence and uniqueness for such BSDEs. The main novelty in their approach is that they require that the solution verifies the equation $\P -a.s.$ for every probability measure $\P$ in a non-dominated class of mutually singular measures. This new point of view is inspired from the quasi-sure analysis of Denis and Martini \cite{deni:mart:06} who established the connection between the so-called hedging problem in uncertain volatility models and the Black--Scholes--Barrenblatt PDE (see also Avellaneda, Levy and Paras \cite{ALP95} and Lyons \cite{L95}). The latter equation is fully nonlinear and has a simple piecewise linear dependance on the second order term. Intuitively speaking (we refer the reader to \cite{STZ10} for more details), the solution to a 2BSDE with generator $F$ and terminal condition $\xi$ can be understood as a supremum in some sense of the classical BSDEs with the same generator and terminal condition, but written under the different probability measures considered. Following this intuition, a non-decreasing process $K$ is added to the solution and it somehow pushes (in a minimal way) the solution so that it stays above the solutions of the classical BSDEs. The theory being very recent, the literature remains rather limited. However, we refer the interested reader to Possama\"i \cite{poss13} and  Possama\"i and Zhou \cite{PZ13} who respectively extended these wellposedness results to generators with linear and quadratic growth, as well as the recent contribution of Possama\"i, Tan and Zhou \cite{PTZ14}, which lifts all the regularity assumptions assumed in the aforementioned works.

\vspace{0.5em}
\paragraph{Main contributions.} Our aim in this paper is to provide a complete theory of existence and uniqueness of second order BDSDEs (2BDSDEs for short) under Lipschitz-type hypotheses on the driver. In addition to the structural difficulties inherent with dealing with 2BSDEs through the quasi-sure analysis, the presence of two sources of randomness in the 2BDSDEs, which are mixed through the nonlinear coefficients of the equation, makes our study even more complex. In particular, we have to be extremely prudent when defining the probabilistic structure allowing for what is now commonly known as volatility uncertainty, in order to consider SPDEs with a nonlinearity with respect to the second-order space derivative. Note that, one of main difficulties with 2BDSDEs and BDSDEs is the extra backward integral term, which prevents us from obtaining path--wise estimates for the solutions, unlike what happens with BSDEs or 2BSDEs. This introduces additional non trivial difficulties. The same type of problems were already pointed out in \cite{BGM15}, where the authors analyze regression schemes for approximating BDSDEs as well as their convergence, and obtain non-asymptotic error estimates, conditionally to the external noise (that is W in our context). Similarly to the classical 2BSDEs, the solution of a 2BDSDE has to be represented as a supremum of solutions to standard BDSDEs. We therefore follow the original approach of Soner, Touzi and Zhang \cite{STZ10} by constructing the solution path--wise, using the so-called regular conditional probability distribution. We point out that since in our context the value process is a random field depending on two source of randomness, we get a dynamic programming principle without regularity on the terminal condition and the generator following the approach of Possama\"i, Tan and Zhou \cite{PTZ14}. This is different from the classical $2$BSDEs where regularity result for the value process, precisely the uniform continuity with respect to the trajectory of the fundamental noise, was crucial to prove their dynamic programming principle (Proposition 4.7 in \cite{sone:touz:zhan:13}).
Moreover, under regularity conditions on the coefficients, we show that classical solution of the fully nonlinear SPDE (\ref{SPDE1}) can be obtained via the associated Markovian $2$BDSDEs, thus extending the Feynman--Kac's formula to this context. Finally, we introduce the notion of stochastic viscosity solution for the fully non--linear SPDEs (\ref{SPDE1}) in the case where the intensity of the noise $g$ in the SPDE \eqref{SPDE1} does not depend on the gradient of the solution. This restriction is due to our approach based on the Doss--Sussmann transformation to convert fully nonlinear SPDEs to fully nonlinear PDEs with random coefficients. Let us conclude by insisting on one of the main implications of our results. It is our conviction that they open a new path for possible numerical simulations of solution to fully--nonlinear SPDEs. Indeed, numerical schemes for classical BDSDEs are by now well--known (see for instance \cite{BGM15}), and numerical procedures for solving second--order BSDEs have also been successfully implemented in the recent years, see for instance Possama\"i and Tan \cite{possamai2015weak} or Ren and Tan \cite{ren2015convergence}. Combining these two approaches should in principle allow to obtain efficient numerical schemes for computing solutions to 2BDSDEs, and therefore for fully non--linear SPDEs. As far as we know,  there are no literature on the subject, except the cases of  semilinear and quasilinear SPDEs (see  \cite{GGK15}, \cite{GK10}, \cite{GK11}, \cite{matouetal13}, \cite{BGM15}), and so our results could prove to be a non--negligible progress.

\vspace{0.5em}
\paragraph{Structure of the paper.} The paper is organized as follows. In Section 2, we recall briefly some notations, introduce the probabilistic structure on the considered product space allowing to choose the adequate set of measures, provide the precise definition of 2BDSDEs and show how they are connected to classical BDSDEs. Then, the aim of Section 3 is to prove the uniqueness of solution for 2BDSDEs, as a direct consequence of a representation theorem, which intuitively originates from the stochastic control interpretation of our problem. The proof of this representation is based on Lemma \ref{eq:mincond}. In this section, we prove also a priori estimates for 2BDSDEs. Section 4 is devoted to the existence result for solution of 2BDSDEs  by a path--wise construction on the shifted Wiener space. Since in our context, the value process is a random field depending on two source of randomness, we prove its regularity result in Lemma \ref{mesurabiliteV}. Once again, we cannot obtain the same regularity in the context of doubly stochastic 2BSDEs because we cannot have path--wise estimates for their solutions. In Section 5, we specialize our discussion to the Markovian context, and we give in Theorem \ref{representationsolutionclassique:theorem} the Feynman--Kac's formula for classical solution of such SPDEs. Then, we introduce the notion of stochastic viscosity solution and give in Theorem \ref{representationsolutionviscosité:theorem} the probabilistic representation for such solution, which is in our knowledge the first result of the kind for such class of fully nonlinear SPDEs. Finally, the Appendix collects several technical results needed for the existence of the solution of the 2BDDSEs and SPDEs. 

\vspace{0.5em}
\paragraph{Notations:} For any $n\in\mathbb N\backslash\{0\}$, we will denote by $x\cdot y$ the usual inner product of two elements $(x,y)$ of $\mathbb R^n$, and by $\|.\|$ the associated Euclidean norm when $n\geq 2$ and by $|\cdot|$ when $n=1$. Furthermore, for any $n\times n$ matrix with real entries $M$, $M^\top$ will denote its usual transpose. We abuse notations and also denote by $\No{\cdot}$ a norm on the space of square matrices with real entries. Moreover, $\S_n^{>0}$ will denote the space of all $n\times n$ positive definite matrices with real entries. For any topological space $E$, $\mathcal B(E)$ will denote the associated Borel $\sigma-$field.

\section{Preliminaries and assumptions}
\label{Preliminaries and Hypothesis }
Let us fix a positive real number $T>0$, which will be our finite time horizon, as well as some integer $d\geq 1$. We shall work on the product space $\Omega := \Omega^B\times\Omega^{W}$ where

\vspace{0.5em}
\noindent $\bullet$ $\Omega^B$ is the canonical space of continuous functions on $[0,T]$ vanishing at $0$, equipped with the uniform norm $\|\cdot\|_{\infty}$. $B$ will be the canonical process on $\Omega^B$, and $\P_B^{0}$ the Wiener measure on $(\Omega^B,\mathcal F^B)$, where $\mathcal F^B$ is the Borel $\sigma$-algebra. Generically, we will denote by $\omega^B$ an element of $\Omega^B$.

\vspace{0.5em}
\noindent
$\bullet$ $(\Omega^W, \Fc^{W}, \P^0_{W})$ is a copy of $(\Omega^B,\Fc^B)$ whose canonical process is denoted by $W$, and whose Wiener measure is denoted bu $\P_0^W$. Generically, we will denote by $\omega^W$ an element of $\Omega^W$, and the notation $\omega$ will be solely reserved for elements $\omega:=(\omega^B,\omega^W)$ of $\Omega$.
%$\Omega_W:=\{\tilde{\omega}\in C([0,T],\R^l), \tilde{\omega}_0=0\}$ equipped with the uniform norm $\|\tilde{\omega}\|_{\infty}:= \underset{0\leq t\leq T}{\Sup}|\tilde{\omega}_t|$, $W$ the canonical brownian motion $(W_t(\tilde{\omega})=\tilde{\omega}_t)$, $\P_0^{W}$ the Wiener measure.

\vspace{0.5em}
\noindent
We equip the product space $\Omega$ with the product $\sigma$-algebra $\Fc= \Fc^{B}\otimes\Fc^{W}$, and define $\P_0:= \P_B^0\otimes\P^0_{W}$. Let then $\F^{o,W}:=\{\Fc_{t,T}^{o,W}\}_{t\geq 0}$ and $\F^{W}:=\{\Fc_{t,T}^{W}\}_{t\geq 0}$ be respectively the natural and the augmented (under $\P^0_W$) retrograde filtration generated by $W$, defined by 
$$\Fc_{s,t}^{o,W}:=\sigma\{W_{r}-W_{s} , s\leq r\leq t\},\ \Fc_{s,t}^W:=\sigma\{W_{r}-W_{s} , s\leq r\leq t\}\vee \mathcal N^{\P^0_W}(\mathcal F^W),$$
where 
$$\mathcal N^{\P^0_W}(\mathcal F^W):=\left\{A\in\Omega^W,\ \exists\ \tilde A\in\Fc^W,\ A\subset \tilde A\ \text{and}\ \P^0_W(\tilde A)=0\right\}.$$
We also denote for simplicity $\Fc_{T}^{W}:=\Fc_{0,T}^{W}$. Similarly, we let $\F^{B}:=\{\Fc_{t}^{B}\}_{t\geq 0}$ be the forward (raw) filtration generated by $B$, that is $\Fc_t^{B}:=\sigma\{B_r, 0\leq r\leq t\}$ (we remind the reader that it is a classical result that in this case $\mathcal F^B=\mathcal F^B_T$). We also consider its right limit $\F^{B}_{+}:=\{\Fc_{t^+}^{B}\}_{t\geq 0}$. Finally, for each $t\in[0,T]$, we define
 $$\Fc_t:= \Fc_t^{B}\vee \Fc_{t,T}^{W}\quad \textrm{and}\quad \Gc_t:= \Fc_t^{B}\vee\Fc_{T}^{W}.$$
The collection $\F=(\Fc_t)_{ 0\leq t\leq T}$ is neither increasing nor decreasing and therefore does not constitute a filtration. However, $\G=(\Gc_t)_{ 0\leq t\leq T}$
is a filtration.

\vspace{0.5em}
\noindent For technical reasons related to the main result of Nutz \cite{N12}, we will work under the set theoretic model of ZFC (Zermelo--Fraenkel plus the axiom of choice) as well as any additional axiom ensuring the existence of medial limits in the sense of Mokobodzki (see \cite{F84}, statement $22$O(l) page $55$ for models ensuring this).

\subsection{A special family of measures on $(\Omega, \mathcal F)$}
In order to be able to consider SPDEs with a nonlinearity with respect to the second-order space derivative, we will need to consider a probabilistic structure allowing for what is now commonly known as volatility uncertainty. This basically means that we will allow the probability measure we consider on $(\Omega^B,\mathcal F^B)$ to change. Such an approach has been initiated with the name of quasi-sure stochastic analysis by Denis and Martini \cite{deni:mart:06} and has since proved very successful (see among others \cite{sone:touz:zhan:13, STZ10, Nutz12}). Following this approach, we say that a probability measure $\P_B$ on $(\Omega^B,\Fc^B)$ is a local martingale measure if the canonical process $B$ is a  local martingale under $\P_B$. We emphasize that by using integration by parts as well as the path--wise stochastic integration of Bichteler (see Theorem 7.14 in \cite{B81} or the more recent article of Karandikar \cite{K95}), we can give a path--wise definition of the quadratic
variation $\langle B\rangle_t$ and its density with respect to the Lebesgue measure $\widehat{a}_t$, by
$$\widehat{a}_t:=\underset{\epsilon \downarrow 0}{\overline{\Lim}}\
\Frac{1}{\epsilon} (\langle B\rangle_t-\langle B\rangle_{t-\epsilon}).$$
where the $\overline{\Lim}$ has to be understood in a component--wise sense.

\vspace{0.3em}
\noindent For practical purposes, we will restrict our attention to the set $\overline{\Pc}_S$ consisting of all probability measures 
$$\P(A):=\int_{\Omega^W}\int_{\Omega^B}{\bf 1}_{(\omega^B,\omega^W)\in A}d\P^{\alpha,\omega^W}(\omega^B)d\P^0_{W}(\omega^W),\ A\in\Fc,$$
such that for any $\omega^W\in\Omega^W$
\begin{align*} 
\P^{\alpha,\omega^W} :=
\P_{B}^0\circ(X^{\alpha,\omega^W})^{-1},\ \text{where} ~ X^{\alpha,\omega^W}_t :=
\Int_0^t \alpha_s^{1/2}(\cdot,\omega^W)dB_s\,, \,t\in[0,T]\,,\,
 \P_B^0-a.s.,
\end{align*} 
for some $\F$-adapted process $\alpha$ taking values in $\S_d^{>0}$ and satisfying for any $\omega^W\in\Omega^W$, $$\Int_0^T \|\alpha_t(\cdot,\omega^W)\|dt < \infty,\  \P_B^0 -a.s.$$
We emphasize here that by the classical results of Stricker and Yor \cite{SY78} on stochastic integration with a parameter, that we can always assume without loss of generality that the map
$$(t,\omega^B,\omega^W)\longmapsto \left( \int_0^t\alpha_s(\cdot,\omega^W)^{1/2}dB_s\right)(\omega^B),$$
is $\mathcal B([0,t])\otimes\mathcal F_t-$measurable. This implies in particular that the family $(\P^{\alpha,\omega^W},\ \omega^W\in\Omega^W)$ is a stochastic kernel (see for instance Definition 7.12 in \cite{BS78}).
The set $\overline{\Pc}_S$ has several nice properties, which can be deduced from similar results in \cite{sone:touz:zhan:11a}.

\begin{Lemma}
Every $\P\in\overline{\Pc}_S$ satisfies the martingale representation property, in the sense that for any square integrable $(\P, \G)-$martingale $M$, there exists a unique $\G-$predictable process $Z$ such that
$$M_t=M_0+\int_0^tZ_s\cdot dB_s,\ \P-a.s.$$
Moreover, they satisfy the Blumenthal $0-1$ law, and in particular, any $\mathcal G_{0^+}$ or $\mathcal F_{0^+}-$measurable random variable is deterministic with respect to $\Omega^B$, that is its randomness only comes from $\Omega^W$.
\end{Lemma}

\proof 
Let $M$ be a square integrable $(\P, \G)-$martingale. We have
$$M_s(\omega^B,\omega^W)=\E^\P\left[\left.M_t\right|\mathcal G_s\right](\omega^B,\omega^W),\ 0\leq s\leq t\leq T,\ \P-a.e.\ (\omega^B,\omega^W)\in\Omega.$$
However, using again the result of Stricker and Yor \cite{SY78}, this can be rewritten, for $\P_0^W-a.e.$ $\omega^W\in\Omega^W$, as
$$M_s(\cdot,\omega^W)=\E^{\P^{\alpha,\omega^W}}\left[\left.M(\cdot,\omega^W)\right|\mathcal F_t^B\right](\cdot),\ \P^{\alpha,\omega^W}-a.s.,$$
which therefore implies that for $\P_0^W-a.e.$ $\omega^W\in\Omega^W$, $M(\cdot,\omega^W)$ is a $(\P^{\alpha,\omega^W},\F^B)-$martingale, to which we can then apply the result of \cite{sone:touz:zhan:11a} to obtain the required martingale representation.

\vspace{0.5em}
\noindent The exact same reasoning gives us the second desired result, using again the fact that by \cite{sone:touz:zhan:11a}, the probability measures $\P^{\alpha,\omega^W}$ satisfy the Blumenthal $0-1$ law.
\ep
 
\vspace{0.5em}
%\textcolor{red}{Attention, maintenant on n'a plus que $B$ et $W$ ont leur crochet nul sous $\P$, mais si je ne me suis pas plant\'e, \c{c}a ne change rien \`a la formule d'It\^o qui est a priori le seul truc important pour nous. Par contre, \c{c}a peut potentiellement \^etre tr\`es g\^enant pour l'aggr\'egation du $Z$...}

\begin{Remark}
We recall from {\rm \cite{sone:touz:zhan:13}} that for a fixed $\P\in\overline{\Pc}_S$, we have from the Blumenthal zero-one law that $\E^\P[\xi|\Gc_t]=\E^\P[\xi|\Gc_{t^+}],\, \P- a.s.$ for any $t\in[0,T]$ and $\P-$integrable $\xi$. In particular, this implies immediately that any $\Gc_{t^+}-$measurable random variable has a $\Gc_t-$measurable $\P-$modification. Furthermore, if $\F^{W,o}$ denotes the raw backward filtration of the Brownian motion $W$, then any $\Gc_{t^+}-$measurable random variable also admits a $\F^B_t\vee\Fc^{W,o}_T-$measurable $\P-$modification. We will often implicitly work with such modifications $($which of course depend on the considered measure $\P)$.
\end{Remark}

\noindent We finish this section with the following definition.
\begin{Definition} For any subset $\mathcal Q$ of $\overline{\Pc}_S$, we say that a property holds $\mathcal Q-$quasi-surely $(\Qc-q.s.$ for short$)$ if it holds $\P-a.s.$ for all
$\P\in\Qc$.
\end{Definition}

\subsection{The non-linearity}
To introduce the non-linearity in the SPDEs we consider, we have to take a small detour, and start by introducing a map $H_t(w,y,z,\gamma):[0,T]\times\Omega\times\R\times\R^d\times D_H\longrightarrow \R$, where
$D_H\subset \R^{d \times d}$ is a given subset containing $0$. As is usual in any stochastic control problem, the following Fenchel conjugate of $H$ with respect to $\gamma$ will play an important role
$$F_t(w,y,z,a) := \underset{\gamma\in D_H}{\Sup} \left\{\Frac{1}{2} Tr(a\gamma) - H_t(w,y,z,\gamma)\right\}~\text{for}~ a\in \S_d^{>0}.$$
For ease of notations, we also define
$$\widehat{F}_t(y,z) := F_t(y,z,\widehat{a}_t)~\text{and}~ \widehat{F}_t^0 := \widehat{F}_t(0,0),$$
where as usual we abuse notations and suppress the dependence on $\omega\in\Omega$ when it is not important (which will not be the case in all the article). Since $F$ may not be always finite, we denote by $D_{F_t(y,z)} := \{a,~ F_t(w,y,z,a) < +\infty\}$ the domain of $F$ in $a$ for a fixed $(t,w,y,z)$.

\vspace{0.5em}
\noindent We will also consider a function $g_t(\omega,y,z):[0,T]\times\Omega\times\R\times\R^d\longrightarrow \R^d$ and denote $g_t^0:= g_t(0,0)$, with the same convention as above on the $\omega-$dependence. The maps $\widehat F$ and $g$ are intended to play the role of the generators of the doubly stochastic BSDEs we will consider later on. Since the theory of BDSDEs is an $L^2$-type theory, we need to restrict again the class $\overline{\Pc}_S$ to account for these integrability issues. 
\begin{Definition}\label{defP}
$\Pc$ is the collection of all $\P\in\overline{\Pc}_S$ such that
$$\underline{a}_{\P}\leq \widehat{a}\leq\overline{a}_{\P} ,~dt\times d\P-a.e.~ \text{for some}~ \underline{a}_{\P},\overline{a}_{\P}\in\S_d^{>0},$$
$$ \E^{\P}\left[\left(\Int_0^T |\widehat{F}_t^0|^{2} dt\right)\right] <+\infty~, ~ \E^{\P}\left[\left(\Int_0^T \|g_t^0\|^{2}
dt\right)\right] <+\infty.$$
\label{def}
\end{Definition}
\noindent The first condition in Definition \ref{defP} ensures that the process $B$ is actually a square-integrable martingale under any of the measures in $\Pc$, and not only a local-martingale, while the second condition is here to ensure wellposedness of the BDSDEs which will be defined below. Of course, these integrability assumptions will not be enough and need to be complemented with further assumptions on the functions $F$ and $g$ that we now list
\begin{Assumption} \label{ass}
\begin{itemize}
\item[\rm{(i)}] $\Pc$ is not empty, and the domain $D_{F_t(y,z)}=:D_{F_t}$ is actually independent of $(\omega,y,z)$.
\item[\rm{(ii)}] For fixed $(t,y,z,a)\in[0,T]\times\mathbb R\times\mathbb R^d\times D_{F_t}$, $F_t(\cdot,y,z,a)$ is $\Fc_t-$measurable, and $g$ is $\Fc_t-$measu-rable as well.
\item[\rm{(iii)}] There is $C > 0$ and $0\leq \alpha <1$ s.t. $ \forall (y,y',z,z',t,a,\omega ) \in \mathbb R \times \mathbb R \times \mathbb R^d  \times \mathbb R^d \times [0,T] \times D_{F_t}\times\Omega  $, 
$$ ~ |F_t(\omega ,y,z,a)-F_t(\omega,y',z',a)|\leq C\left(|y-y'|+\|a^{1/2}(z-z')\|\right),$$
$$~ \|g_t(\omega ,y,z)-g_t(\omega,y', z')\|^2\leq C |y-y' |^2+\alpha\|z-z'\|^2$$
\item[\rm{(iv)}] There exists a constant $\lambda\in[0,1[$ such that
$$(1-\lambda)\widehat{a}_t \geq \alpha I_d,\ dt\times \Pc-q.e. $$
\item[\rm{(v)}] $F$ and $g$ are uniformly continuous in $\omega$ for the $\|. \|_\infty$ norm on $\Omega$.
%\item[\rm{(vi)}] $g$ is uniformly continuous in $\omega$ for the $\|. \|_\infty$
%norm.
\end{itemize}
\end{Assumption}
\begin{Remark}
The assumptions $(i)$ and $(ii)$ are classic in the second order framework, see {\rm \cite{STZ10}}. The Lipschitz assumption $(iii)$ is standard in the BSDE theory since the paper {\rm\cite{PP90}}. The contraction condition satisfied by $g$ with respect to the variable $z$ $($i.e  $ 0 \leq  \alpha < 1)$ and assumption $(iv)$ are necessary for the wellposedness of  
our second order BDSDEs.  These type of conditions are well known $($see e.g. {\rm\cite{pardouxt1980stochastic}}$)$ for the semilinear stochastic PDE \eqref{semilinearSPDE} to be a well-posed stochastic parabolic equation. The last hypothesis $(v)$ is proper to the second order framework, it is linked to our intensive use of regular conditional probability distributions $($r.c.p.d.$)$ in our existence proof, and to the fact that we construct our solutions pathwise, thus avoiding complex issues related to negligible sets.
\end{Remark}

\subsection{Important spaces and norms}
For the formulation of the second order BDSDEs, we will use the same spaces and norms (albeit with some modifications, in particular concerning the measurability assumptions) as the one introduced for second order BSDEs in \cite{STZ10}.

\vspace{0.3em}
\noindent $\bullet$ For $p\geq 1$, $L^{p}$ denotes the space of all $\Fc_T-$measurable scalar r.v. $\xi$ with
$$\|\xi\|^p_{L^{p}}:= \underset{\P\in\Pc}{\Sup}\,
\E^{\P}[|\xi|^p] < +\infty .$$
$\bullet$ $\H^{p}$ denotes the space of all $\R^d-$valued processes $Z$ which are $\G-$predictable and s.t. $Z_t$ is $\underset{\P\in\Pc}{\bigcap}\overline{\Fc}^\P_t-$measurable for a.e. $t\in[0,T]$, with
$$\|Z\|^p_{\H^{p}}:= \underset{\P\in\Pc}{\Sup}\, \E^{\P}\left[\left(\Int_0^T \|\widehat{a}^{1/2}_tZ_t\|^2 dt\right)^{\frac{p}{2}}\right]< +\infty .$$
$\bullet$ $\D^{p}$ denotes the space of $\R-$valued processes $Y$, which are $\G-$progressively measurable, s.t. $Y_t$ is $\Fc_{t^+}^B\vee\Fc_{t,T}^W-$ measurable for every $t\in[0,T]$, with
$$\Pc-q.s.\text{ c\`adl\`ag paths, and } \ \|Y\|^p_{\D^{p}}:= \underset{\P\in\Pc}{\Sup}\,\E^{\P}\left[\underset{0\leq t\leq T}{\Sup}\,|Y_t|^p \right]< +\infty .$$
$\bullet$ $\I^{p}$ denotes the space of all $\R-$valued and $\G-$progressively measurable processes $K$ null at $0$, s.t. $K_t$ is $\underset{\P\in\Pc}{\bigcap}\overline{\Fc}^\P_t-$measurable for every $t\in[0,T]$, with
$$\Pc-q.s.\text{ c\`adl\`ag and nondecreasing paths, and  } ~ \|K\|^p_{\I^{p}}:= \underset{\P\in\Pc}{\Sup}\,\E^{\P}\Big[K_T^p \Big]< +\infty .$$
$\bullet$ For each $\xi\in L^{1},\P\in\Pc$ and $t\in[0,T]$, we denote by $\E_t^{\Pc,\P}[\xi]
:=\underset{\P^{'}\in\Pc(t^+,\P)}{\rm ess \, sup^{\P}}\,
\E_t^{\P^{'}}[\xi],$  
with 
$$\Pc(t^+,\P):=\{\P^{\prime}\in\Pc\,,\,\P^{\prime} = \P~\text{on}~ \Gc_{t^+} \},$$
where $ \E_t^{\P}[\xi]:= \E^{\P}[\xi|\Gc_t]=\E^{\P}[\xi|\Fc_t^{B}\vee \Fc_{T}^{W}] ,~ \P-a.s.$
Then we define for each $p\geq 2$,
$$\L^{p} := \{\xi\in L^{p} : \|\xi\|_{\L^{p}} < +\infty\},$$
where $$\|\xi\|^p_{\L^{p}}:=  \underset{\P\in\Pc}{\Sup}\,
\E^{\P}\Big[\underset{0\leq t\leq T}{\rm ess \, sup^{\P}}\,\big(\E_t^{\Pc,\P}[|\xi|^{2}]\big)^{\frac{p}{2}}\Big].$$
Finally, we denote by ${\rm UC}_b(\Omega)$ the collection of all bounded and uniformly continuous maps $\xi:\Omega \longrightarrow \R$ for the $\|.\|_\infty-$norm, and we let $\Lc^{p}$ be the closure of ${\rm UC}_b(\Omega)$ under the norm $\|.\|_{\L^{p}}$, for every $p\geq 2$.
\subsection{Definition of the 2BDSDE and connection with standard BDSDEs}
We shall consider the following second order backward doubly stochastic differential equation (2BDSDE for short) 
\begin{align}
Y_t = \xi+\Int_t^T\widehat{F}_s(Y_s,Z_s)ds+\Int_t^T g_s(Y_s,Z_s)\cdot d\W_s-\Int_t^T Z_s \cdot dB_s+ K_T-K_t.
\label{eq} 
\end{align} 
We note that the integral with respect to $W$ is a "backward It\^o integral" (see \cite{K90}, pages 111--112) and the integral with respect to $B$ is a standard forward It\^o integral. For any $\P\in\Pc$, $\G-$stopping
time $\tau$, and $\Gc_{\tau}-$measurable random variable $\xi\in\L^2(\P)$, let
$(y^{\P},z^{\P}):=(y^{\P}(\tau,\xi),z^{\P}(\tau,\xi))$ denote the unique solution to the following BDSDE 
\begin{align}
 y^{\P}_t= \xi+\Int_t^{\tau} \widehat{F}_s(y^{\P}_s,z^{\P}_s)ds+\Int_t^{\tau} g_s(y^{\P}_s,z^{\P}_s)\cdot d\W_s -\Int_t^{\tau}z^{\P}_s \cdot dB_s, \ 0\leq t\leq T,\ \P -a.s.
 \label{BDSDE}
\end{align}
Let us point out immediately that wellposedness of a solution is not an immediate consequence of the classical result of Pardoux and Peng \cite{pp1994}. Indeed, in our setting the two martingales $W$ and $B$ are not independent. In our case, the only argument in \cite{pp1994} which does not go through {\it mutatis mutandis} is the one 
 at the end of their proof of their Proposition 1.2 which proves that the solution $(y^\P_t,z^\P_t)$ is actually $\Fc_{t^+}^B\vee\Fc_{t,T}^W-$measurable. However, by Lemma \ref{BDSDEshift} and Step 1 of its proof, in particular \eqref{eq:tructruc}, the required measurability becomes clear.
 
 \vspace{0.5em}
 \noindent We can now give the definition of a solution to a 2BDSDE. 
\begin{Definition}\label{def:def} For $\xi\in \L^{2}$, $(Y,Z,K)\in\D^{2}\times\H^{2}\times\I^2$ is a solution to the 2BDSDE \eqref{eq} if \eqref{eq} is satisfied $\Pc-q.s.$, and if the process $K$ satisfies the following minimality condition
  \begin{align}
  K_t =\underset{\P^{'}\in\Pc(t^+,\P)}{\rm ess \, inf^{\P}}\E_t^{\P^{'}}[K_T], ~\P-a.s.~\text{for all}~ \P\in\Pc, \ t\in[0,T]. \label{eqmin}
  \end{align}
\end{Definition}

\begin{Remark}
We emphasize that we should normally make the dependence of $K$ in the measure $\P$ explicit, since the two stochastic integrals on the right--hand side of \eqref{eq} are, {\it a priori}, only defined $\P-a.s.$ However, we are in a context where we can use the main aggregation result of {\rm \cite{N12}} to always define an universal version of these integrals.
\end{Remark}

\noindent Before closing this subsection, we highlight the fact that Definition \ref{def:def} contains the classical theory of BDSDEs. Indeed, let $f$ be the following linear function of $\gamma$
$$f_t(y,z,\gamma) = \Frac{1}{2}I_d:\gamma - \hat{f}_t(y,z),$$
where $I_d$ is the identity matrix in $\R^d$. Then, we verify immediately that $D_{F_t(w)}=\{I_d\},$ $\widehat{F}_t(y,z) =
\hat{f}_t(y,z)$ and $\Pc=\{\P_0\}$.
In this case, the minimum condition
(\ref{eqmin}) implies
$$ 0=K_0=\E^{\P_0}[K_T]\ \text{and thus}\ K=0,\  \P_0-a.s.,$$
since $K$ is nondecreasing. Hence, the 2BDSDE (\ref{eq}) is equivalent to the following BDSDE: 
\begin{align*} 
Y_t=\xi+\Int_t^T \hat{f}_s(Y_s,Z_s)ds+\Int_t^T g_s(Y_s,Z_s)\cdot d\W_s-\Int_t^T Z_s\cdot  dB_s, 0\leq
t\leq T,~\P_0-a.s. \end{align*}
In addition to Assumption \ref{ass}, we will need to assume the following stronger integrability conditions
\begin{Assumption}\label{ass1}
The processes $\widehat{F}^0$ and $g^0$ satisfy the following integrability conditions for some $\varepsilon>0$
\begin{align*}
\phi^{2,\eps} :=\underset{\P\in\Pc}{\Sup}\,\E^{\P}\left[\Int_0^T |\widehat{F}_s^0|^{2+\varepsilon} ds\right] <+\infty,\ \psi^{2,\eps} :=\underset{\P\in\Pc}{\Sup}\,\E^{\P}
\left[\Int_0^T \|g_s^0\|^{2+\varepsilon} ds\right]<+\infty.
\end{align*}
\end{Assumption}
\noindent Finally, we will see later on in our proof of a priori estimates for the solution of the $2$BDSDE \eqref{eq}, that we actually need to have $L^{2+\varepsilon}-$type estimates for the solutions $(y^\P,z^\P)$ of the corresponding BDSDEs. For this reason, we need to also consider the following, which already appeared as Assumption (H.2) in \cite{pp1994}
\begin{Assumption}\label{assumtion g}
There exist $c>0$ and $0\leq \beta < 1$ such that for all $(t,y,z)\in [0,T]\times\R\times\R^{d}$
$$(g_tg_t^\top) (y,z)\leq c (1+y^2)I_d+ \beta zz^\top.$$
\end{Assumption}
\section{Uniqueness of the solution and estimates}
\label{Uniqueness of the solution and other properties}
\subsection{Representation and uniqueness of the solution}
The aim of this section is to prove the uniqueness of solution for $2$BDSDEs \eqref{eq}, as a direct consequence of a representation theorem, which intuitively originates from the stochastic control interpretation of our problem. As it will become more and more apparent, one of main difficulties with $2$BDSDEs and BDSDEs is the extra backward integral term, which prevents us from obtaining pathwise estimates for the solutions, unlike what happens with BSDEs or 2BSDEs. This introduces additional non trivial difficulties. The same type of problems were already pointed out in \cite{BGM15}, where the authors analyze regression schemes for approximating BDSDEs as well as their convergence, and obtain non-asymptotic error estimates, conditionally to the external noise (that is $W$ in our context). We start with a revisit of the minimality condition \reff{eqmin}.

\begin{Lemma}\label{eq:mincond}
The minimum condition \reff{eqmin} implies that
$$\underset{\P^{\prime}\in\Pc(t^+,\P)}{\inf}\E^{\P^{\prime}}\left[K_T-K_t\right]=0.$$
\end{Lemma}

\begin{proof}
Indeed, fix some $\P\in\Pc$ and some $\P^{\prime}\in\Pc(t^+,\P)$. Taking expectation under $\P$ in \reff{eqmin}, we obtain readily
\begin{equation*}
\E^\P\left[\underset{\P^{\prime}\in\Pc(t^+,\P)}{\rm ess \, inf^{\P}}\E_t^{\P^{\prime}}[K_T-K_t]\right]=0.
\end{equation*}
Then, we know that the family $\Pc(t^+,\P)$ is upward directed (it is indeed clear from the result of \cite{sone:touz:zhan:13}). Therefore, by classical results, there is a sequence $(\P^n)_{n\geq 0}\subset \Pc(t^+,\P)$ such that
$$\underset{\P^{\prime}\in\Pc(t^+,\P)}{\rm ess \, inf^{\P}}\E_t^{\P^{\prime}}[K_T-K_t]=\underset{n\rightarrow+\infty}{\lim}\downarrow \E_t^{\P^{n}}[K_T-K_t].$$
Using this in \eqref{eq:mincond} and then the monotone convergence theorem under the fixed measure $\P$, we obtain
\begin{align*}
0=\E^\P\left[\underset{\P^{\prime}\in\Pc(t^+,\P)}{\rm ess \, inf^{\P}}\E_t^{\P^{\prime}}[K_T-K_t]\right]=\E^\P\left[\underset{n\rightarrow+\infty}{\lim}\downarrow \E_t^{\P^{n}}[K_T-K_t]\right]&=\underset{n\rightarrow+\infty}{\lim}\downarrow\E^\P\left[ \E_t^{\P^{n}}[K_T-K_t]\right]\\
&=\underset{n\rightarrow+\infty}{\lim}\downarrow\E^{\P^{n}}\left[ \E_t^{\P^{n}}[K_T-K_t]\right]\\
&=\underset{n\rightarrow+\infty}{\lim}\downarrow\E^{\P^{n}}\left[K_T-K_t\right]\\
&\geq \underset{\P^{'}\in\Pc(t^+,\P)}{\inf}\E^{\P^{'}}\left[K_T-K_t\right].
\end{align*}
Since $K$ is a non-decreasing process, the result follows.
\ep
\end{proof}

\vspace{0.5em}
\noindent We can now show  as in Theorem 4.4 of \cite{sone:touz:zhan:13} that the solution to the $2$BDSDE (\ref{eq}) can be represented as a supremum of solutions to the BDSDEs (\ref{BDSDE}).
\begin{Theorem}\label{theorep}
Let Assumptions \ref{ass} and \ref{ass1} hold. Assume $\xi\in\L^{2}$ and that $(Y,Z,K)$ is a solution to 2BDSDE \reff{eq}. Then, for any $\P\in\Pc$ and $0\leq t_1 <t_2\leq
T,$ 
\begin{align} 
Y_{t_1} = \underset{\P^{'}\in\Pc(t_1^+,\P)}{\rm ess \, sup^{\P}}
y_{t_1}^{\P^{'}}(t_2,Y_{t_2}),~ \P-a.s. \label{eqrepresentation} \end{align}
%Consequently, the 2BDSDE \reff{eq} has at most one solution in
%$\D^{2}\times\H^{2}$.
\end{Theorem}
\begin{proof}
%Let first assume that $(Y,Z,K)$ is a solution to 2BDSDE (\ref{eq}) and (\ref{eqrepresentation}) is true. Then 
%\begin{align*} 
%Y_{t} = \underset{\P^{'}\in\Pc(t^+,\P)}{\rm ess \, sup^{\P}}
%y_{t}^{\P^{'}}(T,\xi)~,~ t\in[0,T],~ \P-a.s.,\ \text{for all}~
%\P\in\Pc, \end{align*}
%and thus $Y$ is uniquely determined. Since we also have that $d\langle Y,B\rangle_t = Z_t \widehat a_tdt,~\Pc - q.s., Z$ is also uniquely determined. Finally, the processes $K^{\P}$ are then uniquely determined. 
%The rest of the proof is devoted to the proof of the representation (\ref{eqrepresentation}). 
We follow the (by now) classical approach for this problem, first used in \cite{STZ10}, and proceed in 2 steps.

\vspace{0.3em}
(i) Fix $0\leq t_1 <t_2\leq T,$ and $\P\in\Pc $. For any
$\P^{'}\in\Pc(t_1^+,\P)$, note that from (\ref{eq}), we have, $ \P^{'} - a.s.$, for any $t_1\leq t\leq t_2$
\begin{align*}
 Y_t=Y_{t_2}+\Int_t^{t_2}\widehat{F}_s(Y_s,Z_s)ds+\Int_t^{t_2} g_s(Y_s,Z_s)\cdot d\W_s-\Int_t^{t_2} Z_s\cdot dB_s +K_{t_2}-K_t, 
\end{align*}
and that $K$ is nondecreasing, $\P^{'}-a.s.$ Applying the comparison principle for BDSDE (see \cite{SGL05}) under $\P$ , we have 
$Y_{t_1}\geq y_{t_1}^{\P^{'}}(t_2,Y_{t_2}),$ $\P^{'} - a.s$. Since $\P^{'} = \P$ on $\Gc_{t_1}^{+}$, we get $Y_{t_1}\geq y_{t_1}^{\P^{'}}(t_2,Y_{t_2}), \P - a.s.$ and thus 
\begin{align*} 
Y_{t_1} \geq\underset{\P^{\prime}\in\Pc(t_1^+,\P)}{\rm ess \, sup^{\P}}
y_{t_1}^{\P^{\prime}}(t_2,Y_{t_2}),~ \P-a.s. \end{align*} 

(ii) To prove the reverse inequality in representation (\ref{eqrepresentation}), we use standard linearization techniques. Fix
$\P\in\Pc $, for every  $\P^{\prime}\in\Pc(t_1^+,\P)$, denote $ \delta Y := Y-y^{\P^{'}}(t_2,Y_{t_2}) ~\text{and}~ \delta Z := Z-z^{\P^{'}}(t_2,Y_{t_2}). $
By Assumption \ref{ass}(iii), there exist bounded processes $\lambda, \eta, \gamma,\beta$, which are respectively $\R$, $\R^d$, $\R^d$ and $\R-$valued, such that, $\P^{'}-a.s.$ 
\begin{align*} 
\delta Y_t =&\ \Int_t^{t_2}\!(\lambda_s\delta Y_s+\eta_s\cdot \widehat{a}_s^{\frac12}\delta Z_s)ds+\Int_t^{t_2}\left(\!\gamma_s\delta Y_s +\beta_s\delta Z_s\right)\cdot d\W_s-\Int_t^{t_2}\!\delta Z_s\cdot dB_s +K_{t_2}-K_t. \end{align*}
Define
\begin{align}
M_t:=\exp\Big(\Int_0^t\eta_s\cdot \widehat{a}_s^{-1/2}dB_s+\Int_0^t\lambda_sds -\Frac{1}{2}\Int_0^t\|\eta_s\|^2 ds\Big),\ t_1\leq t\leq t_2,\ \P^{'} - a.s. \label{M}
\end{align}
By integration by parts, we have
\begin{align*}
d(M_t\delta Y_t) = M_t(\delta Z_t + \delta Y_t\eta_t\widehat{a}^{-1/2}_t)\cdot dB_t-M_t\beta_t\delta Z_t\cdot d\W_t -M_t dK_t.
\end{align*}
We deduce
$$\E^\P[\delta Y_{t_1}]=\E^{\P^{'}}\left[M_{t_1}^{-1}\Int_{t_1}^{t_2}M_tdK_t\right]\leq\E^{\P^{'}}\left[\underset{t_1\leq t\leq t_2}{\Sup}(M_{t_1}^{-1}M_{t})(K_{t_2}-K_{t_1})\right],$$
%+\Int_{t_1}^{t_2}\begin{align}ta_t
%\widehat{b}_t^{1/2}\delta Z_t(d\W_t+\alpha_t\widehat{b}_tdt))]$$ 
where we used the fact that $K^{\P^{'}}$ is non-decreasing and that since $\delta Y_{t_1}$ is $\mathcal F_{t_1^+}$-measurable, its expectation is the same under $\P$ and $\P^{'}$. By the boundedness of $\lambda,\eta, \gamma, \beta$, for every $p\geq 1$ we have, 
\begin{align}
\E^{\P^{'}}\left[\underset{t_1\leq t\leq t_2}{\Sup}(M_{t_1}^{-1}M_{t})^p+\underset{t_1\leq t\leq t_2}{\Sup}(M_{t_1}M_{t}^{-1})^p\right]\leq C_p ,\ t_1\leq t\leq t_2, \ \P^{'} - a.s.\label{eq4} 
\end{align} 
Then it follows from the H\"{o}lder inequality that
\begin{align*} 
\mathbb E^\P\left[Y_{t_1}-y_{t_1}^{\P^{'}}(t_2,Y_{t_2})\right]&\leq \left(\E^{\P^{'}}\left[\underset{t_1\leq t\leq t_2}{\Sup}(M_{t_1}^{-1}M_{t})^3\right]\right)^{1/3}\left(\E^{\P^{'}}\left[(K_{t_2}-K_{t_1})^{3/2}\right]\right)^{2/3}\\
&\leq C\left(\E^{\P^{'}}\left[K_{t_2}-K_{t_1}\right]\E^{\P^{'}}\left[(K_{t_2}-K_{t_1})^{2}\right]\right)^{1/3}.
\end{align*} 
From the  definition of $K$, we have
\begin{align}
\underset{\P^{\prime}\in\Pc(t_1^+,\P)}{\Sup} \E^{\P^{\prime}}\left[(K_{t_2}-K_{t_1})^{2}\right] \leq C\left(\|Y\|^2_{\D^{2}}+
\|Z\|^2_{\H^{2}}+\left(\phi^{2,\eps}\right)^{\frac 2{2+\eps}}+\left(\psi^{2,\eps}\right)^{\frac 2{2+\eps}} \right)<+ \infty.\label{eq5}
\end{align}
Then, by taking the infimum in $\Pc(t_1^+,\P)$ in the last inequality and using (\ref{eq5}) and the result of Lemma \ref{eq:mincond}, we obtain
$$\underset{\P^{'}\in\Pc(t_1^+,\P)}{\inf}\mathbb E^\P\left[Y_{t_1}-y_{t_1}^{\P^{'}}(t_2,Y_{t_2})\right]\leq 0.$$
But we clearly have
$$0\geq\underset{\P^{'}\in\Pc(t_1^+,\P)}{\inf}\mathbb E^\P\left[Y_{t_1}-y_{t_1}^{\P^{'}}(t_2,Y_{t_2})\right]\geq \mathbb E^\P\left[Y_{t_1}-\underset{\P^{'}\in\Pc(t_1^+,\P)}{\rm ess \, sup^{\P}}y_{t_1}^{\P^{'}}(t_2,Y_{t_2})\right].$$
%Hence $\mathbb E^\P\left[Y_{t_1}-\underset{\P^{'}\in\Pc(t_1^+,\P)}{\rm ess \, sup^{\P}}y_{t_1}^{\P^{'}}(t_2,Y_{t_2})\right]\leq 0.$$
Since the quantity under the expectation is positive $\P-a.s.$ by Step $1$, we deduce that it is actually equal to $0$, $\P-a.s.$, which is the desired result.
\ep
\end{proof}

\vspace{0.5em}
\noindent As an immediate consequence of the representation formula (\ref{eqrepresentation}) together with the comparison principle for BDSDEs, we have the following comparison principle for $2$BDSDEs.
\begin{Theorem}
Let $(Y,Z)$ and $(Y',Z')$ be the solutions of $2$BDSDEs with terminal conditions $\xi$ and $\xi'$ and generators $\widehat{F}$ and $\widehat{F}'$ respectively, and let $(y^{\P},z^{\P})$ and $(y'^{\P},z'^{\P})$ the solutions of the associated BDSDEs. Assume that they both verify Assumptions \ref{ass} and \ref{ass1}, and that we have $\Pc - q.s.$, $\xi\leq\xi' ,~ \widehat{F}(y_t'^{\P},z_t'^{\P})\leq \widehat{F}^{'}(y_t'^{\P},z_t'^{\P}).$ Then $Y \leq Y' ,$  $\Pc - q.s.$
\end{Theorem}
\subsection{A priori estimates}
In this section, we show some a priori estimates which will be not only useful in the sequel, but also ensure the uniqueness of a solution to a 2DBSDE in $\D^2\times\H^2$. 
%{\bf Le th\'eor\`eme est \`a revoir compl\`etement, Au d\'ebut il faut soit rappeller soit prouver les estim\'ees dans $L^{2+\eps}$ pour les BDSDEs, le faire dans $L^2$ ne sert \`a rien je pense. Puis, il faut mettre une r\'ef\'erence apr\`es pour l'estim\'ee sur $Y$, quand on utilise l'in\'egalit\'e de Doob pour les $G$-martingales, sinon c'est incompr\'ehensible !!!}
We start with a reminder of $L^{p}$ estimates for solutions of BDSDEs which were proved in \cite{pp1994} (see Theorem 4.1 p.217).
\begin{Theorem}\label{th:estimees}
Let Assumptions \ref{ass}, \ref{ass1} and \ref{assumtion g} hold and assume that $\xi\in L^{2+\eps}$ for some $\eps>0$. We have
\begin{equation*}
\E^{\P}\left[\underset{0\leq t\leq T}{\Sup}|y_t^{\P}|^{2+\eps}+\left(\Int_0^T \|\widehat{a}_s^{1/2}z_s^{\P}\|^2ds\right)^{\frac{2+\eps}{2}}\right] \leq C \E^{\P}\left[|\xi|^{2+\eps}+\Int_0^T (|\widehat{F}_t^0|^{2+\eps} + \|g_t^0\|^{2+\eps}) dt\right].
\end{equation*} 
\end{Theorem}
\noindent The main result of this section is then
\begin{Theorem}\label{thestiamte}
Let Assumptions \ref{ass}, \ref{ass1} and \ref{assumtion g} hold.

\vspace{0.5em}
\noindent $\rm{(i)}$ Assume $\xi\in\L^{2}\cap L^{2+\eps}$ and that $(Y,Z,K)$ is a solution to the $2$BDSDE \reff{eq}. Then, for any $\eps'\in(0,\eps)$, there exist a constant $C$ such that 
\begin{align*}
\|Y\|^{2+\eps'}_{\D^{2+\eps'}} + \|Z\|^{2+\eps'}_{\H^{2+\eps'}} + \underset{\P\in\Pc}{\Sup}\,\E^{\P}\left[ |K_T|^{2+\eps'}\right] \leq C \left(\|\xi\|^{2+\eps'}_{\L^{2+\eps'}} +\|\xi\|^{\frac{2+\eps'}{2+\eps}}_{L^{2+\eps}}+ (\phi^{2,\eps})^{\frac{2+\eps'}{2+\eps}}+(\psi^{2,\eps})^{\frac{2+\eps'}{2+\eps}}\right). 
\end{align*} 
\vspace{0.5em}
\noindent $\rm{(ii)}$ Assume $\xi^i\in\L^{2}\cap L^{2+\eps}$ and that $(Y^i,Z^i, K^i)$ is a solution to the $2$BDSDE \reff{eq}, $i=1,2$. Denote $\delta\xi:=\xi^1-\xi^2,\,\delta Y:=Y^1-Y^2,\,\delta Z:= Z^1-Z^2,\, \text{and}\ \delta K:=K^{1}-K^{2}$. Then, there exist a constant $C$ such that
\begin{align*} 
\|\delta Y\|_{\D^{2}}& \leq C\|\delta\xi\|_{\L^{2}},\ 
\|\delta Z\|^2_{\H^{2}} + \|\delta K\|^2_{\mathbb I^{2}}
\leq  C \|\delta\xi\|_{\L^{2}}\sum_{i=1}^2\left(\|\xi^i\|^2_{\L^{2}} +\|\xi^i\|^{\frac{2}{2+\eps}}_{L^{2+\eps}}+ (\phi^{2,\eps})^{\frac 2{2+\eps}}+(\psi^{2,\eps})^{\frac 2{2+\eps}}\right).
\end{align*}
\end{Theorem}
\begin{proof}
(i) For every $\P\in\Pc $ and $\P^{\prime}\in\Pc(t^+,\P)$ we have, using Theorem \ref{theorep} and the usual linearization procedure, that for some bounded processes $\alpha$ and $\beta$ and any $\eps'\in(0,\eps)$
\begin{align*}
&\E^{\P}\left[\underset{0\leq t\leq T}{\Sup}\,|Y_t|^{2+\eps'}\right]=  \E^{\P}\left[\underset{0\leq t\leq T}{\Sup}\left( \underset{\P^{'}\in\Pc(t^+,\P)}{\rm ess \, sup^{\P}}|y_t^{\P^{'}}|\right)^{2+\eps'}\right] \\
\leq&\   \E^{\P}\left[\underset{0\leq t\leq T}{\Sup}\left(\E_t^{\Pc,\P}\left[|\xi |+\Int_t^T |\widehat{F}_s^0|ds +C\Int_t^T ( |y_s^{\P^{'}}|+ \|\widehat{a}_s^{1/2}z_s^{\P^{'}}\|)ds+\abs{\Int_t^T g^0_s\cdot d\W_s}\right.\right.\right. \\
&\left.\left.\left. + \abs{\Int_t^T \alpha_s y_s^{\P^{'}}\cdot d\W_s}+\abs{\Int_t^T \beta_s z_s^{\P^{'}}\cdot d\W_s}+\abs{\Int_t^T z_s^{\P^{'}}\cdot dB_s}\right]\right)^{2+\eps'}\right]\\
 \leq&\  C\left(\E^{\P}\left[\underset{0\leq t\leq T}{\Sup} \E_t^{\Pc,\P}[|\xi |]^{2+\eps'}\right]+\E^{\P}\left[\underset{0\leq t\leq T}{\Sup}\left(\E_t^{\Pc,\P}\left[\Int_0^T |\widehat{F}_s^0| ds\right]\right)^{2+\eps'}\right]\right)\\
&+   C\left(\E^{\P}\left[\underset{0\leq t\leq T}{\Sup}\left(\E_t^{\Pc,\P}\left[\Int_0^T |y_s^{\P^{'}}|ds\right]\right)^{2+\eps'}\right]+\E^{\P}\left[\underset{0\leq t\leq T}{\Sup}\left(\E_t^{\Pc,\P}\left[\Int_0^T \|\widehat{a}_s^{1/2}z_s^{\P^{'}}\|ds\right]\right)^{2+\eps'}\right]\right)\\
&+   C\left(\E^{\P}\left[\underset{0\leq t\leq T}{\Sup}\left(\E_t^{\Pc,\P}\left[\abs{\Int_t^T g^0_s\cdot d\W_s}\right]\right)^{2+\eps'}\right]+\E^{\P}\left[\underset{0\leq t\leq T}{\Sup}\left(\E_t^{\Pc,\P}\left[\abs{\Int_t^T z_s^{\P^{'}}\cdot dB_s}\right]\right)^{2+\eps'}\right]\right)\\
&+C\E^{\P}\left[\underset{0\leq t\leq T}{\Sup}\left(\E_t^{\Pc,\P}\left[\abs{\Int_t^T y_s^{\P^{'}}\alpha_s\cdot d\W_s}\right]\right)^{2+\eps'}+\underset{0\leq t\leq T}{\Sup}\left(\E_t^{\Pc,\P}\left[\abs{\Int_t^T z_s^{\P^{'}}\beta_s\cdot d\W_s}\right]\right)^{2+\eps'}\right].
\end{align*}

\vspace{0.5em}
\noindent Now we remind the reader that by (\cite{sone:touz:zhan:11b}, \cite{poss13}), for any random variable $A$, we have
%{\bf R\'ef\'erence au papier de Nizar sur les $G$-expectations et/ou au mien sur les 2BSDE \`a croissance lin\'eaire)}
$$\underset{\P\in\Pc}{\Sup}\ \E^{\P}\left[\underset{0\leq t\leq T}{\Sup}\ \E_t^{\Pc,\P}[|A|]^{2+\eps'}\right]\leq C \underset{\P\in\Pc}{\Sup}\left(\E^{\P}\left[|A|^{2+\epsilon}\right]\right)^{\frac{2+\eps'}{2+\epsilon}}.$$

\vspace{0.5em}
\noindent We therefore deduce with BDG inequalities that (remember that $\alpha $ and $\beta$ are bounded)
\begin{align*}
\E^{\P}\left[\underset{0\leq t\leq T}{\Sup}\,|Y_t|^{2+\eps'}\right]\leq &\ C\left(\No{\xi}^{2+\eps'}_{\L^{2+\eps'}}+(\phi^{2,\eps})^{\frac{2+\eps'}{2+\eps}}+(\psi^{2,\eps})^{\frac{2+\eps'}{2+\eps}}\right)\\
&+C\underset{\P\in\Pc}{\sup}\ \E^{\P}\left[\underset{0\leq t\leq T}{\Sup}|y_t^{\P}|^{2+\eps}+\left(\Int_0^T \|\widehat{a}_s^{1/2}z_s^{\P}\|^2ds\right)^{\frac{2+\eps}{2}}\right] ^{\frac{2+\eps'}{2+\eps}}.
\end{align*}
%%|Y_t|^2 &=& |\xi |^2 + 2\Int_t^T\widehat{F}_s(Y_s,Z_s)Y_s ds +2\Int_t^T g_s(Y_s,Z_s)Y_sd\W_s -2\Int_t^T Y_s dK_s^\P\\
%&-&  2\Int_t^T Y_s Z_s dB_s-\Int_t^T |\widehat{a}^{1/2}_sZ_s|^2 ds+ \Int_t^T
%|g_s(Y_s,Z_s)|^2ds 
%\end{align*}
%Hence, from Assumption \ref{ass} and Burkholder-Davis-Gundy's inequality,
%\begin{align*}
%\E^\P[\underset{0\leq t\leq T}{\Sup} |Y_t|^2] &\leq & C\epsilon^{-1} \E^\P\big[|\xi |^2+ \Int_0^T|\widehat{F}_s^0|ds + \Int_0^T|g_s^0|ds + \Int_0^T |Y_s|^2 ds\big] + (\epsilon -\Frac{1}{2}) \E^\P\big[ \Int_0^T|\widehat{a}_s^{1/2}Z_s|^2ds\big]\\
%&+& C \Big(\E^\P\Big[\sqrt{\Int_0^T |g_s(Y_s,Z_s)Y_s|^2ds}\Big] + \E^\P\Big[\sqrt{\Int_0^T  |Y_s\widehat{a}^{1/2} Z_s|^2ds}\Big]\Big)\\
%&+& \epsilon_1 \E^\P[\underset{0\leq t\leq T}{\Sup} |Y_t|^2]+ \epsilon_1^{-1} \E^\P[|K_T^\P|^2]
%\end{align*}
%for any $\epsilon, \epsilon_1\in(0,1]$. But by the definition of $K^{\P}_T$, it is clear that
%\begin{align}
 %\E^{\P}\Big[|K_T^{\P}|^2\Big]&\leq& C_0\E^{\P}\Big[|\xi|^2+|Y_0|^2+\Int_0^T|\widehat{a}^{1/2}_sZ_s|^2ds
%+\Int_0^T|\widehat{F}^0_s|^2ds+\Int_0^T |g^0_s|^2ds\Big]
%\end{align}
%and, by Burkholder-Davis-Gundy inequality, we get 
%\begin{align*} 
%\E^{\P}[\Sup_{0\leq t\leq T}|Y_t|^2] \leq C_\k
%(||\xi||^2_{\L^{2,\k}} + \underset{\P^{'}\in\Pc^\k(t^+,\P)}{\rm ess \, sup^{\P}}\big(\E^{\P^{'}}\big[\Int_0^T(|\widehat{F}_s^0|^{\k}ds\big]\big)^{2/\k} + \underset{\P^{'}\in\Pc^\k(t^+,\P)}{\rm ess \, sup^{\P}}\big(\E^{\P^{'}}\big[\Int_0^T(|\widehat{F}_s^0|^{\k}ds\big]\big)^{2/\k}) \label {eqesty}
%\end{align*}

\vspace{0.5em}
\noindent Finally, we obtain by Theorem \ref{th:estimees}
\begin{align}
\|Y\|^{2+\eps'}_{\D^{2+\eps'}} \leq C \left(\No{\xi}^{2+\eps'}_{\L^{2+\eps'}}+\No{\xi}_{L^{2+\eps}}^{2+\eps'}+(\phi^{2,\eps})^{\frac{2+\eps'}{2+\eps}}+(\psi^{2,\eps})^{\frac{2+\eps'}{2+\eps}}\right). \label {eqesty}
\end{align}

\vspace{0.5em}
\noindent When it comes to the estimate for $Z$, we apply It\^o's formula to $|Y|^2$ under each $\P\in\Pc$ and from the Lipschitz Assumption \ref{ass}(iii) we have, using BDG inequality and our assumptions on $g$ and $\widehat F$
 
 \begin{align*}
\E^{\P}\left[\left(\Int_0^T \|\widehat{a}^{1/2}_sZ_s\|^2 ds\right)^{\frac{2+\eps'}{2}}\right] \leq&\ C \E^{\P}\left[|\xi|^{2+\eps'}+\left(\Int_0^T |Y_s|(|\widehat{F}^0_s|+|Y_s|+\|\widehat{a}^{1/2}_sZ_s\|)ds\right)^{\frac{2+\eps'}{2}}\right]\\
&+C\E^{\P}\left[\left(\int_0^TY_s^2\No{g_s(Y_s,Z_s)}^2ds\right)^{\frac{2+\eps'}{4}}+\left(\int_0^TY_s^2\No{\widehat a_s^{1/2}Z_s}^2ds\right)^{\frac{2+\eps'}{4}}\right]\\
&+C\E^\P\left[\left(\Int_0^T\| g_s(Y_s,Z_s)\|^2ds\right)^{\frac{2+\eps'}{2}}+\Int_0^T|Y_s|^{\frac{2+\eps'}{2}}dK_s\right]\\
%\leq &\ C \E^{\P}\left[|\xi|^2+\Int_0^T |Y_s|(|\widehat{F}^0_s|+|Y_s|+\|\widehat{a}^{1/2}_sZ_s\|)ds+\Int_0^T|Y_s|dK_s\right]\\
%&+ C \E^{\P}\left[\Int_0^T(\|g_s^0\|^2+|Y_s|^2)ds\right]+\alpha \E^{\P}\left[\Int_0^T\|Z_s\|^2ds\right]\\
\leq &\ C \nu^{-1} \E^{\P}\left[|\xi|^{2+\eps'}+\underset{0\leq s\leq T}{\Sup}|Y_s|^{2+\eps'}+\Int_0^T|\widehat{F}^0_s|^{2+\eps'}ds+\Int_0^T\| g^0_s\|^{2+\eps'}ds\right]\\
&+ \nu \E^{\P}\left[\left(\Int_0^T\|\widehat{a}^{1/2}_sZ_s\|^2ds\right)^{\frac{2+\eps'}{2}}+|K_T|^{2+\eps'}+\alpha^{\frac{2+\eps'}{2}} \left(\Int_0^T\|Z_s\|^2ds\right)^{\frac{2+\eps'}{2}}\right],
\end{align*}
for any $\nu\in(0,1]$. But by the definition of $K_T$, it is clear that
\begin{align}\label{eq:k}
\E^{\P}\left[|K_T|^{2+\eps'}\right]\leq  C_0\E^{\P}\left[|\xi|^{2+\eps'}+\underset{0\leq s\leq T}{\Sup} |Y_s|^{2+\eps'}+\left(\Int_0^T\|\widehat{a}^{1/2}_sZ_s\|^2ds\right)^{\frac{2+\eps'}2}
+\Int_0^T\left(|\widehat{F}^0_s|^{2+\eps'}+ \|g^0_s\|^{2+\eps'}\right)ds\right],
\end{align}
for some constant $C_0$ independent of $\varepsilon$. 

\vspace{0.5em}
\noindent Then, using in particular Assumption \ref{ass}(iv)
\begin{align*}
\E^{\P}\left[\left(\Int_0^T \|\widehat{a}^{1/2}_sZ_s\|^2 ds\right)^{\frac{2+\eps'}{2}}\right]\leq&\ C \nu^{-1} \E^{\P}\left[|\xi|^{2+\eps'}+\underset{0\leq s\leq T}{\Sup}|Y_s|^{2+\eps'}
+\Int_0^T|\widehat{F}^0_s|^{2+\eps'}ds+\Int_0^T \|g^0_s\|^{2+\eps'}ds\right]\\
&+ (\nu +C_0\nu + (1-\lambda)^{\frac{2+\eps'}{2}})\E^{\P}\left[\left(\Int_0^T\|\widehat{a}^{1/2}_sZ_s\|^2ds\right)^{\frac{2+\eps'}{2}}\right]. 
\end{align*}
Choosing $\nu$ small enough, this implies the desired result by (\ref{eqesty}). Finally, the estimate for the $K$ follows directly from \reff{eq:k}.

\vspace{0.5em}
(ii) First of all, we can follow the same arguments as in (i) above to obtain the existence of a constant
$C$, depending only on $T$ and the Lipschitz constant of
$\widehat{F}$ and $g$ such that for all $\P\in\Pc$ 
\begin{align}
\E^{\P}\left[\underset{0\leq t\leq T}{\Sup}\,|\delta Y_t|^2\right]=  \E^{\P}\left[\underset{0\leq t\leq T}{\Sup}\left( \underset{\P^{'}\in\Pc(t^+,\P)}{\rm ess \, sup^{\P}}|\delta y_t^{\P^{'}}|\right)^2\right]\leq C\left(\|\delta\xi\|_{\L^{2}}^2+\|\delta\xi\|_{\L^{2+\eps}}^{\frac 2{2+\eps}}\right)
\label{eqest1} \end{align} 
%Then by definition of our norms, we get from (\ref{eqest1}) and (\ref{eqrepresentation}) that 
%\begin{align}
% \|\delta Y\|_{\D^{2}}\leq C \|\delta\xi\|_{\L^{2}}.
%\end{align} 
Applying It\^o formula to $|\delta Y|^2$, under each $\P\in\Pc$, leads to 
\begin{align*} 
\E^{\P}\left[\Int_0^T \|\widehat{a}^{1/2}_s\delta Z_s\|^2 ds\right]\leq &\
C\E^{\P}\left[|\delta\xi|^2+\Int_0^T |\delta Y_s|(|\delta Y_s|
+\|\widehat{a}^{1/2}_s\delta Z_s\|)ds+\Int_0^T|\delta Y_s|d|\delta K_s|+\Int_0^T|\delta Y_s|^2ds\right]\\
\leq&\ C  \E^{\P}\left[|\delta\xi|^2+\underset{0\leq s\leq T}{\Sup}|\delta Y_s|^2+\underset{0\leq s\leq T}{\Sup}|\delta Y_s|^2[K_T^{1}+K_T^{2}]+ \Frac{1}{2}\Int_0^T\|\widehat{a}^{1/2}_s\delta Z_s\|^2ds\right].
\end{align*}
The estimate for $\delta Z$ is now obvious from the above inequality and the estimates of (i). Finally the estimate for the difference
of the increasing processes is obvious by definition.
\ep
\end{proof}
\section{Existence by a pathwise construction of the solution}
\label{A direct existence argument}
As we have shown in Theorem \ref{theorep}, if a solution to the $2$BDSDE (\ref{eq}) exists, it necessarily can be represented as a supremum of solutions to standard BDSDEs. However, since we are working under a family of non-dominated probability measures, we cannot use the classical technics of BSDEs to construct such a solution. We will therefore follow the original approach of Soner, Touzi and Zhang \cite{sone:touz:zhan:13}, who overcame this problem by constructing the solution pathwise, using the so-salled regular conditional probability distribution.
\subsection{Notations related to shifted spaces}
For any $0\leq t\leq T$, we denote by $\Omega^{B,t}:= \{\omega\in C([t,T],\R^d),\  \omega(t) = 0\}$ the shifted canonical space, $B^t$
the shifted canonical process, $\P_0^{B,t}$ the shifted Wiener measure and $\Fc^{B,t}$ the shifted raw filtration generated by $B^t$. The pathwise density of its quadratic variation is denoted by $\widehat a^t$. We then let $\Omega^t:=\Omega^{B,t}\times\Omega^W$,  and exactly as in Section 2, we can define the set $\overline{\Pc}_S^t$, by restricting the corresponding measures to the shifted space $\Omega^t$.

\vspace{0.3em}
\noindent Next, for any $0\leq s\leq t\leq T$ and $\omega\in\Omega^s$, we define the shifted path $\omega^t:=(\omega^{B,t},\omega^{W})\in\Omega^t$ by
$$\omega^{B,t}_r:=\omega_r^B-\omega_t^B,\ \forall r\in[t,T],$$
and for $\omega^B\in\Omega^{B,s},  \   \tilde{\omega}^B\in\Omega^{B,t}$ we define the concatenated path
by 
\begin{align*}
(\omega^B\otimes_t\tilde{\omega}^B)(r):=\omega^B_r\indi_{[s,t)}(r)+(\omega^B_t+\tilde{\omega}^B_r)\indi_{[t,T]}(r),\
\forall r\in[s,T].
\end{align*}
Similarly, for any $\Fc_T^s-$measurable random variable $\xi$ on $\Omega^s$, and for each
$(\omega^B,\omega^W)\in\Omega^s$, we define the $\Fc_T^t-$measurable random variable $\xi^{t,\omega^B}$ on $\Omega^t$ by
$$\xi^{t,\omega^B}(\tilde{\omega}^B,\omega^W):=\xi(\omega^B\otimes_t\tilde{\omega}^B,\omega^W),\
\forall(\tilde{\omega}^B,\omega^W)\in\Omega^t.$$
%For a $\G$-stopping time $\tau$, we use the following simplification 
%$$\omega\otimes_\tau\tilde{\omega}:= \omega\otimes_{\tau(\omega)}\tilde{\omega}~,~ \xi^{\tau,\omega}:=\xi^{\tau(\omega),\omega}~,
%~  X^{\tau,\omega}:= X^{\tau(\omega),\omega}.$$
The shifted generators that we consider are, for every $(s,(\tilde{\omega}^B,\omega^W))\in[t,T]\times\Omega^t$
\begin{align*}
\widehat{F}_s^{t,\omega^B}((\tilde{\omega}^B,\omega^W),y,z)&:= F_s((\omega^B\otimes_t\tilde{\omega}^B,\omega^W),y,z,\widehat{a}_s^t(\tilde{\omega}^B,\omega^W)), \\
g_s^{t,\omega^B}((\tilde{\omega}^B,\omega^W),y,z)&:= g_s((\omega^B\otimes_t\tilde{\omega}^B,\omega^W),y,z).
\end{align*}
Then note that since $F$ and $g$ are assumed to be uniformly continuous in $\omega$, then so are the maps $(\omega^B,\omega^W)\longmapsto F_s((\omega^B\otimes_t\cdot,\omega^W),\cdot)$ and $(\omega^B,\omega^W)\longmapsto g_s^{t,\omega^B}((\cdot,\omega^W),\cdot)$. Notice that this implies that for any $\P\in\overline{\Pc}_S^t$
$$ \E^{\P}\left[\left(\Int_t^T |\widehat{F}_s^{t,\omega^B}(0,0)|^2 ds\right)\right]+\E^{\P}\left[\left(\Int_t^T \|g_s^{t,\omega^B}(0,0)\|^{2}ds\right)\right] <+\infty,$$
for some $\omega^B\in\Omega^B$ if and only if it holds for all $\omega^B\in\Omega^B$. 

\vspace{0.3em}
\noindent We also extend Definition \ref{def} in the shifted spaces
\begin{Definition}
$\Pc^{t}$ is the subset of $\overline{\Pc}_S^t$, consisting of measures $\P$ such that
$$\underline{a}_{\P}\leq \widehat{a}^t\leq\overline{a}_{\P} ,~dt\times d\P-a.e.~\text{on}~ [t,T]\times\Omega^{t},\  \text{for some}~ \underline{a}_{\P},\overline{a}_{\P}\in\S_d^{>0},$$
$$ \E^{\P}\left[\left(\Int_t^T |\widehat{F}_s^{t,\omega^B}(0,0)|^{2} ds\right)\right] + \E^{\P}\left[\left(\Int_t^T \|g_s^{t,\omega^B}(0,0)\|^{2}ds\right)\right] <+\infty,~ \text{for all}~ \omega^B\in\Omega^B.$$
\end{Definition}

\vspace{0.5em}
\noindent Finally, by Stroock and Varadhan \cite{SV79}, for any $\F^B-$stopping time $\tau$, any probability measure $\P_B$ on $(\Omega^B,\Fc^B)$, and any $\omega^B\in\Omega^B$, there exists a regular conditional probability distribution (r.p.c.d. for short), $\P_{B,\tau(\omega^B)}^{\omega^B}$ with respect to the $\sigma-$field $\Fc^B_\tau$ (since it is countably generated). Such a measure verifies that for every integrable $\Fc^B_T-$measurable random variable $\xi$, we have for $\P^B-a.e.$ $\omega^B$
$$\E^{\P^B}[\left.\xi\right|\Fc^B_\tau](\omega^B) = \E^{\P_{B,\tau(\omega^B)}^{\omega^B}}[\xi].$$
Furthermore, this r.c.p.d. naturally induces a probability measure $\P_{B}^{\tau(\omega^B),\omega^B}$ on $(\Omega^{B,\tau(\omega^B)},\mathcal F^{B,\tau(\omega^B)}_T)$ such that
$$\E^{\P_{B,\tau(\omega^B)}^{\omega^B}}[\xi]=\E^{\P_{B}^{\tau(\omega^B),\omega^B}}[\xi^{\tau(\omega^B),\omega^B}].$$
Notice that if we consider a stochastic kernel $\{\P_B(\omega^W),\ \omega^W\in\Omega^W\}$ on $(\Omega^B,\Fc^B_T)$, then for any $\omega^B\in\Omega^B$, $\{\P_B^{\tau(\omega^B),\omega^B}(\omega^W),\ \omega^W\in\Omega^W\}$ is a stochastic kernel on $(\Omega^{B,\tau(\omega^B)},\Fc^{B,\tau(\omega^B)}_T)$. 

\vspace{0.5em}
\noindent Let us now consider a $\G-$stopping time $\tau$. Then, using the classical results of Stricker and Yor \cite{SY78}, we know that there is a $\P_0^W$ version of $\tau$ (still denoted $\tau$ for simplicity) such that the map $\omega^B\longmapsto \tau(\omega^B,\omega^W)$ defines a $\F^B-$stopping time for $\P_0^W-a.e.$ $\omega^W\in\Omega^W$. Let again $\{\P_B(\omega^W),\ \omega^W\in\Omega^W\}$ be a stochastic kernel on $(\Omega^B,\Fc^B_T)$ and let us define the measure $\P$ on $(\Omega,\Fc)$ by
$$d\P(\omega^B,\omega^W)=d\P_B(\omega^W;\omega^B)d\P_0^W(\omega^W).$$
We claim (and refer the reader to the proof of the more general result in Lemma \ref{BDSDEshift}) that we can write for $\P-a.e.$ $\omega\in\Omega$
\begin{align*}
\E^{\P}[\left.\xi\right|\mathcal G_\tau](\omega^B,\omega^W)&=\E^{\mathbb P_0^W}\left[\left.\int_{\Omega^B}\xi^{\tau(\omega^B,\cdot),\omega^B}(\widetilde\omega^B)d\P_B^{\tau(\omega^B,\cdot),\omega^B}(\cdot;\widetilde\omega^B)\right|\Fc_T^W\right](\omega^W)\\
&=\mathbb E^{\mathbb P_B^{\tau(\omega^B,\omega^W),\omega^B}(\omega^W)}\left[\xi^{\tau(\omega^B,\omega^W),\omega^B}(\cdot,\omega^W)\right].
\end{align*}
Moreover, by Lemma 4.1 in \cite{sone:touz:zhan:13}, we know that for any probability measure $\P\in\overline{\Pc}_S$ on $(\Omega,\Fc_T)$ such that $d\P(\omega^B,\omega^W)=d\P_B(\omega^W;\omega^B)d\P_0^W(\omega^W)$, we have for $\P-a.e.$ $\omega\in\Omega$, for any $\F^B-$stopping time $\tau$ and for $ds\times d\P^{\tau(\omega^B),\omega^B}(\omega^W)-a.e.\ (s,\tilde\omega^B)\in[\tau(\omega^B),T]\times\Omega^{B,\tau(\omega^B)}$
\begin{equation}\label{eq:aa}
\widehat a_s^{\tau(\omega^B),\omega^B}(\tilde \omega^B,\omega^W)=\widehat a^{\tau(\omega^B)}_s(\tilde\omega^B,\omega^W),
\end{equation}
which justifies the definition of the shifted generator $\widehat F^{t,\omega^B}$.

\subsection{Existence when $\xi$ is in ${\rm{UC}}_b(\Omega)$}
When $\xi$ is in $\text{UC}_b(\Omega)$, we know that there exists a
modulus of continuity function $\rho$ for $\xi, F$ and $g$ in
$\omega$. Then, for any $0\leq t\leq s\leq T, (y,z)\in\R\times\R^d$
and $(\omega^B, \omega^{\prime,B})\in\Omega^B\times\Omega^B,$ $\widetilde{\omega}^B\in\Omega^{B,t}$ and $\omega^W\in\Omega^W$,
\begin{align*}
|\xi^{t,\omega^B}(\widetilde{\omega}^B,\omega^W)- \xi^{t,\omega^{\prime,B}}(\widetilde{\omega}^B,\omega^W)|&\leq \rho(\|\omega^B-\omega^{\prime,B}\|_t),\\
 |\widehat{F}_s^{t,\omega^B}((\widetilde{\omega}^B,\omega^W),y,z)- \widehat{F}_s^{t,\omega^{\prime,B}}
((\widetilde{\omega}^B,\omega^W),y,z)|&\leq \rho(\|\omega^B-\omega^{\prime,B}\|_t),\\
| g_s^{t,\omega^B}((\widetilde{\omega}^B,\omega^W),y,z)- g_s^{t,\omega^{\prime,B}}
((\widetilde{\omega}^B,\omega^W),y,z)|&\leq \rho(\|\omega^B-\omega^{\prime,B}\|_t).
\end{align*}
Using this regularity and Assumption \ref{ass}, it is easy to see that we have for all $(t,\omega^B)\in[0,T]\times\Omega^B$.
$$\Lambda_t(\omega^B):= \underset{\P\in\Pc^{t}}{\Sup}\left(\E^{\P}\left[|\xi^{t,\omega^B}|^2+\Int_t^T|\widehat{F}_s^{t,\omega^B}(0,0)|^2 ds
+\Int_t^T\|g_s^{t,\omega^B}(0,0)\|^2 ds\right]\right)^{1/2} < +\infty.$$ 
To prove existence, we define the following value process $V_t$ for every $\omega^B$ 
\begin{align}\label{definitionV}
V_t(\omega^B,\cdot):=\underset{\P\in\Pc^{t}}{\rm ess \, sup}^{\P^0_W}\, \Yc_t^{\P,t,\omega^B}(T,\xi)(\cdot)\,,\, \P^0_W -a.s., 
\end{align} 
where, for any $(t_1,\omega^B)\in[0,T]\times \Omega^B,~ \P\in\Pc^{t_1}, ~ t_2 \in [t_1,T]$, and any $\Fc_{t_2}-$measurable $\eta\in\L^2(\P)$, we
denote $\Yc_{t_1}^{\P,t_{1},\omega^B}(t_2,\eta):= y_{t_1}^{\P,t_{1},\omega^B}$, where
$(y^{\P,t_{1},\omega^B},z^{\P,t_{1},\omega^B})$ is the solution of the following BDSDE on the shifted space $\Omega^{t_1}$ under $\P$, 
\begin{align}
y_s^{\P,t_{1},\omega^B}=&\ \eta^{t_{1},\omega^B}
+\Int_s^{t_2}\widehat{F}_r^{t_{1},\omega^B}(y_r^{\P,t_{1},\omega^B},z_r^{\P,t_{1},\omega^B})dr
-\Int_s^{t_2} z_r^{\P,t_{1},\omega^B}\cdot dB_r^{t_1} +\Int_s^{t_2}
g_r^{t_{1},\omega^B}(y_r^{\P,t_{1},\omega^B},z_r^{\P,t_{1},\omega^B})\cdot d\W_r. \end{align}
The following Lemma allows to give a link between BDSDEs on the shifted spaces. Its technical proof is postponed to the Appendix.

\vspace{0.5em}
\begin{Lemma}
\label{BDSDEshift}
Fix some $\P\in\overline{\mathcal P}_S$ such that $d\P(\omega^B,\omega^W):=d\P_B(\omega^B;\omega^W)d\P_0^W(\omega^W)$. For $\P-a.e.$ $\omega\in\Omega$, the following equality holds
$$y_t^{\P_B^{t,\omega^B}(\cdot)\otimes\P^0_W}(\omega^W)=y_t^{\P}(\omega^B,\omega^W),\ t\in[0,T],$$
where $d(\P_B^{t,\omega^B}(\cdot)\otimes\P^0_W)(\omega^B,\omega^W):=d\P_B^{t,\omega^B}(\omega^B;\omega^W)d\P_0^W(\omega^W)$.
\end{Lemma}

\vspace{0.5em}
\noindent We point out that for classical $2$BSDEs, Soner, Touzi and Zhang have proved in Lemma 4.6 of \cite{sone:touz:zhan:13} a regularity result for the value process, precisely the uniform continuity with respect to the trajectory $\omega^B$ and this is crucial to prove their dynamic programming principle (Proposition 4.7 in \cite{sone:touz:zhan:13}). Since in our context, the value process $V$ defined in \eqref{definitionV} is a random field depending on two source of randomness, we prove the following regularity result which is weaker than Lemma 4.6 of \cite{sone:touz:zhan:13}. Once again, we cannot obtain the same regularity in the context of doubly stochastic 2BSDEs because we cannot have path--wise estimates for their solutions.
\vspace{0.5em}
\begin{Lemma}
\label{mesurabiliteV}We have for every $(\omega^{B,1},\omega^{B,2})\in\Omega^B\times\Omega^B$
$$\E^{\P^0_W}\left[\left(V_t(\omega^{B,1}\cdot)-V_t(\omega^{B,2}\cdot)\right)^2\right]\leq \rho^2\left(\|\omega^{B,1}-\omega^{B,2}\|_t\right).$$
In particular, this implies that the map $\omega^B\longmapsto V_t(\omega^B,\cdot)$ is uniformly continuous in probability $($with respect to $\P^0_W)$, which implies that there is a $\P^0_W-$version, which we still denote $V$ for simplicity, which is jointly measurable in $(\omega^B,\omega^W)$, and more precisely, such that $V_t$ is $\Fc_t-$measurable or even $\Fc^B_t\otimes\Fc^{o,W}_{t,T}-$measurable.
\end{Lemma}
\vspace{0.5em}
\begin{proof}
The estimate is an easy consequence of classical a priori estimates for BDSDEs, using in particular the uniform continuity in $\omega$ of both $F$, $g$ and $\xi$. The reasoning is quite similar to the one we used in the proof of Theorem \ref{thestiamte}, so that we omit it. As for the existence of measurable version, this is a classical result using the fact that the topology of convergence in probability is metrizable (see for instance Dellacherie and Meyer \cite{dm}, chapter IV, Theorem 30, or the proof of Corollary A.3 in \cite{DKN07}).\ep
\end{proof}

\vspace{0.5em}
\noindent We then have the following joint measurability result
\begin{Lemma}\label{mes2}
The map $(t,\omega^B,\omega^W)\longmapsto V_t(\omega^B,\omega^W)$ is $\Bc([0,T])\otimes\Fc_t-$measurable.
\end{Lemma}

\proof
First of all, we claim that the family $\{\Yc_t^{\P,t,\omega^B}(T,\xi),\ \P\in\Pc^t\}$ is upward directed. Indeed, this can be proved exactly as in Step (iii) of the proof of Theorem 4.3 of \cite{STZ10}. As a consequence, we know that there is a sequence $(\P^n)_{n\geq 0}\subset\Pc^t$ such that for $\P_0^W-a.e.$ $\omega^W\in\Omega^W$
$$V_{t}(\omega^B,\omega^W) =\underset{n\geq 0}{\sup}\ \Yc^{\P^n,t\omega^B}_t(\omega^W).$$
Now arguing exactly as in Step (i) of the proof of Theorem 2.1 in \cite{PTZ14}, using in particular the fact that we can always mimic the construction in Section 2.5.2 of \cite{PTZ14} to obtain that the map $(t,\P,\omega^B,\omega^W)\longmapsto \Yc^{\P,t\omega^B}_t(\omega^W)$ is Borel measurable, we deduce that $(t,\omega^B,\omega^W)\longmapsto V_t(\omega^B,\omega^W)$ is $\Bc([0,T])\otimes\Fc_T^{B}\otimes\Fc_{0,T}^{o,W}-$universally measurable. But then it suffices to use the result of Lemma \ref{mesurabiliteV} to conclude.
\ep

\vspace{0.7em}
\noindent Now, we present the main result concerning the dynamic programming principle in our context. We follow the approach of Possama\"i, Tan and Zhou \cite{PTZ14}, where they proved existence result for $2$BSDEs with only measurable parameters. Their proof is based on dynamic programming principle without regularity on the terminal condition and the generator, which is itself strongly inspired by the classical results recalled, for instance, in the papers \cite{elkarouitan1,elkarouitan2}. We therefore omit the proof.

\begin{Theorem}\label{dynam prog}
Under the Assumptions \ref{ass}, \ref{ass1} and for
$\xi\in{\rm UC}_b(\Omega)$, we have for all $0 \leq  t_1\leq t_2\leq T$
 \begin{align} V_{t_1}(\omega^B, \omega^W) =
\underset{\P\in\Pc^{t_1}}{{\rm ess \, sup}^\P}\ \Yc_{t_1}^{\P,t_{1},\omega^B}(t_2,V_{t_2}^{t_{1},\omega^B } (\cdot,\omega^W)) , \; \, \P-a.e.\ \omega\in\Omega.
\end{align}
\end{Theorem}

\noindent Next, we introduce the right limit of the V which is clearly $\Fc_{t^+}-$measurable
\begin{align} \label{defV+}
V_t^{+}:= \underset{r\in\Q\cap(t,T],r\downarrow
t}{\overline{\Lim}} V_r. \end{align}
We have the following regularity result, whose proof is postponed until the appendix.
\begin{Lemma}\label{Lemmaregularity}
Under the Assumptions \ref{ass}, \ref{ass1}, we have
$$V_t^{+}= \underset{r\in\Q\cap(t,T],r\downarrow t}{{\Lim}}V_r, \ \Pc -q.s.$$ and thus $V^{+}$ is c\`adl\`ag
$\Pc -q.s.$
\end{Lemma}
\noindent Thanks to the dynamic programming principle for $V$, as well as the just proved regularity of $V^+$, we can now show that $V^+$ is actually a semi--martingale under any $\P\in\Pc$ and admits a particular decomposition under any $\P\in\mathcal P$.
\begin{Proposition}\label{DecompositionV+}
Under Assumptions \ref{ass}, \ref{ass1}, for any $\P\in\Pc$, denoting by $\G^\P_+$ the usual augmentation of the right-limit of $\G$ under $\P$, there is a $\G^\P_+-$predictable process $\widetilde Z^\P$, which is also $\overline\Fc_t^\P-$mesurable for a.e. $t\in[0,T]$, and a non-decreasing c\`adl\`ag and $\G^\P_+-$predictable process $\widetilde{K}^\P$, which is also $\overline\Fc_t^\P-$mesurable for a.e. $t\in[0,T]$, such that $V^+$ defined by \eqref{defV+} satisfies for all $0\leq t\leq s\leq T$
\begin{align*}
V_s^{+} = \xi+\Int_s^T\widehat{F}_r(V_r^{+},\widetilde{Z}^{\P}_r)ds+\Int_s^T g_r(V_r^{+},\widetilde{Z}^{\P}_r)\cdot d\W_r-\Int_s^T
\widetilde{Z}^{\P}_r \cdot dB_r+\widetilde{K}_T^{\P}-\widetilde{K}_s^{\P}, \ \P-a.s.
\end{align*}
\end{Proposition}
\begin{proof}
We introduce first the following RBDSDE with lower obstacle $V^{+}$ under each $\P\in\Pc^{t}$,
\begin{align*}\left\lbrace
\begin{aligned}
&\widetilde{Y}^{\P}_t
=\xi+\Int_t^T\widehat{F}_s(\widetilde{Y}^{\P}_s,\widetilde{Z}^{\P}_s)ds+\Int_t^T
g_s(\widetilde{Y}^{\P}_s,\widetilde{Z}^{\P}_s)\cdot d\W_s -\Int_t^T \widetilde{Z}^{\P}_s \cdot dB_s+\widetilde{K}_T^{\P}-\widetilde{K}_t^{\P},\\
&\widetilde{Y}^{\P}_t \geq V_t^{+},\ 0\leq t\leq T, \
\P-a.s,\\
& \Int_0^T(\widetilde{Y}^{\P}_{s^-}-V^{+}_{s^-})d\widetilde{K}_{s^-}^{\P}=0 ,\ \P-a.s.
\end{aligned} 
\right.\end{align*}
To the best of our knowledge, there are no results in the literature for the existence and uniqueness of such RBDSDE with c\`adl\`ag obstacle. The proofs of these results are postponed to Section \ref{RBDSDE:section} in the Appendix for completeness. As mentioned in Remark 4.9 in \cite{sone:touz:zhan:13}, and for a fixed $\P\in\Pc^{t}$, we shall use the solution of the above RBDSDEs and the notion of $\widehat{F}-$weak doubly super--martingale whis is introduced in the Appendix to prove the desired result. This notion is a natural extension of nonlinear $f-$super--martingale introduced first by Peng \cite{Peng97} in the context of standard BSDEs. Let us now argue by contradiction and suppose that $\widetilde{Y}^{\P}$ is not equal $\P-a.s.$ to $V^{+}$. Then we can assume without
loss of generality that $\widetilde{Y}^{\P}_0> V_0^{+}, \P-a.s.$ For each $\eps>0$, define the following ${\G}-$stopping
time
$$\tau^{\eps}:= \Inf \{t\geq 0,\ \widetilde{Y}^{\P}_t \leq V_t^{+}+\eps\}.$$
Then $\widetilde{Y}^{\P}$ is strictly above the obstacle before $\tau^{\eps}$, and therefore $\widetilde{K}^{\P}$ is identically equal
to $0$ in $[0,\tau^{\eps}]$. Hence, we have for all $0\leq t\leq s\leq T$
$$\widetilde{Y}^{\P}_s = \widetilde{Y}^{\P}_{\tau^{\eps}}+\Int_s^{\tau^{\eps}}
\widehat{F}_r(\widetilde{Y}^{\P}_r,\widetilde{Z}^{\P}_r)dr
+\Int_s^{\tau^{\eps}}g_r(\widetilde{Y}^{\P}_r,\widetilde{Z}^{\P}_r)\cdot d\W_r-\Int_s^{\tau^{\eps}}\widetilde{Z}^{\P}_r\cdot dB_r,
\ \P-a.s.$$ 
Let us now define the following BDSDE on
$[0,\tau^{\eps}]$
$$ y_s^{+,\P} = V_{\tau^{\eps}}^{+} + \Int_s^{\tau^{\eps}}\widehat{F}_r(y_r^{+,\P},z_r^{+,\P})dr+\Int_s^{\tau^{\eps}}
g_r(y_r^{+,\P},z_r^{+,\P})\cdot d\W_r
-\Int_s^{\tau^{\eps}}z_r^{+,\P}\cdot dB_r , \ \P-a.s.$$ 
By comparison theorem and the standard a priori estimates, we obtain that
$$\E[\widetilde{Y}^{\P}_0] \leq \E[y_0^{+,\P}]+C\E\big[|V_{\tau^{\eps}}^{+}-\widetilde{Y}^{\P}_{\tau^{\eps}}|\big]\leq \E[y_0^{+,\P}]+ C\eps,$$
by definition of $\tau^{\eps}.$

\vspace{0.5em}
\noindent Moreover, we can show similarly to the proof of Lemma \ref{Lemmaregularity} (see also the arguments in Step 1 of the proof of Theorem 4.5 in \cite{sone:touz:zhan:13} pages 328--329) that $V^+$ is a strong $\widehat{F}-$doubly super--martingale under each $\P\in\Pc^{t}$. Thus, we obtain in particular that $y_0^{+,\P} \leq V_0^{+}$ which in
turn implies $$\E[\widetilde{Y}^{\P}_0] \leq \E[V_0^{+}]+ C\eps, $$ hence a contradiction by arbitrariness of $\eps$.
%Therefore, we have obtained the following decomposition
%$$V_s^{+} = \xi+\Int_s^T\widehat{F}_r(V_r^{+},\widetilde{Z}^{\P}_r)dr+\Int_s^T g_r(V_r^{+},\widetilde{Z}^{\P}_r)d\W_r-\Int_s^T
%\widetilde{Z}^{\P}_r dB_r+\widetilde{K}_T^{\P}-\widetilde{K}_s^{\P}, \quad \P-a.s.,\forall \P\in\Pc^t$$
\ep
\end{proof}
%Through Appendix,( On va \'ecrire que $\langle V^{+},B\rangle_t= \Frac{1}{2}V_t^{+}B_t-\Int_0^t V_s^{+}dB_s-\Int_0^t B_s dV_s^{+}$ et \'etendre les r\'esultasts de Karandikar pour le dernier terme) 

\vspace{0.5em}
\noindent We next prove a representation for $V^{+}$ similar to (\ref{eqrepresentation}), which will be useful for us to justify that the value process we have constructed provides indeed a solution to the 2BDSDE \reff{eq}.
\begin{Proposition} Assume that  Assumptions \ref{ass}, \ref{ass1} hold. Then we have
 \begin{align}
V_t^{+}=\underset{\P^{'}\in\Pc(t+,\P)}{{\rm ess \, sup^{\P}}}\Yc_t^{\P^{'}}(T,\xi),\ \P-a.s.,\ \forall \P\in\Pc^t.
\label{repv} \end{align}
\end{Proposition}
\begin{proof}
The proof for the representations is the same as the proof of Lemma 3.5 in \cite{PTZ14}, since a stability result holds in our context, too.
\ep
\end{proof}
\subsection{Existence result in the general case}
We are now in position to state the main result of this section.
\begin{Theorem} \label{Thexistence}
Let $\xi\in\Lc^{2}$ and assume that Assumptions \ref{ass}, \ref{ass1} hold. Then there exists a unique solution $(Y,Z,K)\in\D^{2}\times \H^{2}\times\mathbb I^2$ of the 2BDSDE \reff{eq}.
\end{Theorem}

\begin{proof}
The proof is divided in three steps. In the first one we prove that the value process $V^+$ defined by \eqref{defV+} is the solution of our $2$BDSDE in the case when $\xi$ belongs in ${\rm UC}_b (\Omega)$ and show the aggregation result for the solution. Then, in the second step we verify the minimality condition for the increasing process. Finally, we deal with the general case.

\vspace{0.3em}
{\it{Step 1: Existence and aggregation results for $\xi$ belongs  in  ${\rm UC}_b (\Omega)$.  }} As we have mentioned above, the natural candidate for the $Y$ solution for our $2$BDSDE is given by 
$$Y_t= V_t^+:=\underset{r\in\Q\cap(t,T],r\downarrow t}{{\Lim}}V_r,$$
 where $V$ is the value process defined by \eqref{definitionV}. First, we know that $V^+$ is a c\`adl\`ag process defined path--wise and using the same notations in Proposition \ref{DecompositionV+} our solution $Y$ verifies
$$V_t^+=V_0^+-\Int_0^t\widehat F_s(V_s^+,\widetilde{Z}_s^\P)ds-\Int_0^t g_s(V_s^+,\widetilde{Z}_s^\P)\cdot d\W_s+\Int_0^t\widetilde{Z}_s^{\P} \cdot dB_s-\widetilde K_t^\mathbb P, \text{ }\mathbb P-a.s., \text{ }\forall \mathbb P\in\mathcal P^t.$$
We note that $V^+$ is ($\P-a.s.$) a c\`adl\`ag generalized semi--martingale under any $\P\in\Pc$, (studied by Pardoux and Protter in \cite{PP87} and Pardoux and Peng \cite{pp1994}). By the generalized It\^o's formula of Lemma \ref{proof of lemmaitô}, we have for any $i=1,\dots,d$
\begin{align*}
B_t^iV_t^+=&\ \int_0^t\left(\widehat a_s^{1/2}\widetilde{Z}_s^\P\cdot{\bf 1}_i-\widehat F_s(V_s^+,\widetilde{Z}_s^\P)B^i_s\right)ds+\int_0^t\left(\widetilde{Z}_s^{\P}B^i_s+{\bf 1}_dV_s^+\right)\cdot dB_s\\
&-\int_0^tg_s(V_s^+,\widetilde{Z}_s^\P)B^i_s\cdot d\W_s-\int_0^tB^i_sd\widetilde K_s^\P\\
=&\ \int_0^t\widehat a_s^{1/2}\widetilde{Z}_s^\P\cdot{\bf 1}_ids+\int_0^tB^i_sdV_s^++\int_0^t{\bf 1}_dV_s^+\cdot dB_s.
\end{align*}
Then, we can adapt Karandikar's results obtained for c\`adl\`ag semi--martingale in our context to define universally the two stochastic integrals
$$\int_0^tB^i_sdV_s^+, \text{ and }\int_0^t{\bf 1}_dV_s^+\cdot dB_s.$$
Indeed, $B$ and $V^+$ are both c\`adl\`ag and a backward It\^o integral can always be considered as a forward It\^o integral, provided that time is reversed.

\vspace{0.5em}
\noindent This gives us the existence of a $\G-$predictable process $Z$ such that for any $\P\in\Pc$
$$\widetilde Z^\P_t=Z_t,\ dt\otimes\P-a.e.$$
Furthermore, since for any $\P\in\Pc$, $\widetilde Z_t$ is $\overline\Fc_t^\P-$mesurable for a.e. $t\in[0,T]$, we deduce immediately that $Z_t$ is $\bigcap_{\P\in\Pc}\overline\F_t^\P-$mesurable for a.e. $t\in[0,T]$.
 
\vspace{0.5em} 
\noindent Concerning the fact that we can aggregate the family $(\widetilde K^\mathbb P)_{\mathbb P\in\mathcal P}$, it can be deduced as follows.
We have from \eqref{defV+} that $V^+$ is defined path--wise, and so is the Lebesgue integral $\int_0^t\widehat F_s(V_s^+,Z_s)ds$. By \cite{N12}, the stochastic integrals $\Int_0^tZ_s\cdot dB_s$ and $\Int_0^t g_s(V_s^+,Z_s)\cdot d\W_s$ can also be defined path--wise. We can therefore define path--wise
$$ K_t:=V_0^+-V_t^+-\Int_0^t\widehat F_s(V_s^+, Z_s)ds- \Int_0^t g_s(V_s^+,Z_s)\cdot d\W_s+\Int_0^tZ_s\cdot dB_s,$$
and $K$ is an aggregator for the family $(\widetilde K^\mathbb P)_{\mathbb P\in\mathcal P^t}$. Thus, the triplet $(Y,Z, K)$ satisfies the equation (\ref{eq}) and from the {\it a priori} estimates in Theorem \ref{theorep} we get that $(Y, Z, K)$ belongs to $\D^{2}\times \H^{2}\times \I^{2}$.

\vspace{0.3em}
{\it{Step 2: The minimality condition of $\widetilde K$}.}
Now, we have to check that the minimum condition (\ref{eqmin}) holds. We follow the arguments in the proof of Theorem \ref{theorep}.
For $t\in[0,T],\ \P\in\Pc^t$ and $\P^{\prime}\in\Pc(t+,\P)$, we
denote $\delta Y := V^{+}-y^{\P^{'}}(T,\xi)$ and $\delta Z :=
Z-z^{\P^{'}}(T,\xi)$ and we introduce the process $M$ of
(\ref{M}). We first observe that since ${K}$ is non-decreasing, we have$$\underset{\P^{\prime}\in\Pc(t^+,\P)}{\rm ess \, inf^{\P}}\E_t^{\P^{'}}[{K}_T-{K}_t]\geq 0.$$ 
Then, it suffices to prove that $\E^\P\Big[\underset{\P^{\prime}\in\Pc(t^+,\P)}{\rm ess \, inf^{\P}}\E_t^{\P^{'}}[{K}_T-{K}_t]\Big]\leq 0.$ We know that the family $\Pc(t^+,\P)$ is upward directed. Therefore, by classical results, there is a sequence $(\P^n)_{n\geq 0}\subset \Pc(t^+,\P)$ such that
\begin{align}
\E^\P\left[\underset{\P^{\prime}\in\Pc(t^+,\P)}{\rm ess \, inf^{\P}}\E_t^{\P^{'}}[{K}_T-{K}_t]\right]=\underset{n\rightarrow+\infty}{\lim}\downarrow\E^{\P^{n}}\left[{K}_T-{K}_t\right].\label{essinf}
\end{align}
On the other hand, by (\ref{eq4}), we estimate by the H\"{o}lder inequality that
\begin{align*}
\E^{\P^{n}}\left[{K}_T-{K}_t\right]
&= \E^{\P^{n}}\left[\left(\underset{t\leq s\leq T}{\Inf}(M_{t}^{-1}M_s)\right)^{1/3}\left({K}_T-{K}_t\right)^{1/3}
\left(\underset{t\leq s\leq T}{\Inf}(M_{t}^{-1}M_s)\right)^{-1/3}\left({K}_T-{K}_t\right)^{2/3}\right]\\
%&\leq  \left(\E^{\P^{n}}\left[\left(\underset{t\leq s\leq T}{\Inf}(M_{t}^{-1}M_s)\right)\left(\widetilde{K}_T^{\P^{n}}-\widetilde{K}_t^{\P^{n}}\right)\right]\E^{\P^{n}}\left[\underset{t\leq s\leq T}{\Sup}(M_{t}^{-1}M_s)\right]\E^{\P^{n}}\left[\left(\widetilde{K}_T^{\P^{n}}-\widetilde{K}_t^{\P^{n}}\right)^2\right]\right)^{\frac13}\\
&\leq C \left(\E^{\P^{n}}\left[\left({K}_T^{\P^{n}}\right)^2\right]\E^{\P^{n}}\left[\left(\underset{t\leq s\leq T}{\Inf}(M_{t}^{-1}M_s)\right)\left({K}_T-{K}_t\right)\right]\right)^{1/3}\\
&\leq C \left(\E^{\P^{n}}\left[\left({K}_T\right)^2\right] \E^{\P^{n}}\left[M_{t}^{-1}\Int_t^T M_s d{K}_s\right]\right)^{1/3}\\
&\leq C \left(\E^{\P^{n}}\left[\left({K}_T\right)^2\right]\right)^{1/3}         \left(\E^{\P^{n}}[\delta Y_t]\right)^{1/3},
\end{align*}
where we have used in the last inequality the fact that ${K}$ is non-decreasing and the same arguments as in the proof of Theorem \ref{theorep} (ii). 

\vspace{0.5em}
\noindent Plugging the above in (\ref{essinf}), we obtain
\begin{align*}
\E^\P\left[\underset{\P^{\prime}\in\Pc(t^+,\P)}{\rm ess \, inf^{\P}}\E_t^{\P^{'}}[{K}_T-{K}_t]\right] &\leq  C \underset{n\rightarrow+\infty}{\lim}\downarrow\big(\E^{\P^{n}}\left[\delta Y_t\right]\big)^{1/3}\leq  C \left(\underset{\P^{n}\in\Pc(t^+,\P)}{\rm ess \, inf^{\P}}\E^{\P^{n}}\left[\delta Y_t\right]\right)^{1/3}=0.
\end{align*}
which is the desired result.

\vspace{0.3em}
{\it{Step 3: Existence and aggregation results for $\xi$ belonging to $ \Lc^{2}$. } } For 
$\xi\in\Lc^{2}$, there exists by definition a sequence
$(\xi_n)_{n\geq 0}\subset{UC}_b(\Omega)$ such that 
\begin{align*}
\underset{n\rightarrow
+\infty}{\Lim}\|\xi_n-\xi_m\|_{\L^{2}}=0\quad\text{and}\quad
\underset{n\geq 0}{\Sup}~\|\xi_n\|_{\L^{2}}< +\infty. 
\end{align*} 
Let $(Y^n,Z^n,K^n)\in \D^{2}\times \H^{2}\times\mathbb I^2$ be the solution to $2$BDSDE
(\ref{eq}) with terminal condition $\xi_n$. By the estimates of Theorem
\ref{thestiamte}, we have 
\begin{align*} &
\|Y^n-Y^m\|_{\D^{2}}^2 + \|Z^n-Z^m \|^2_{\H^{2}} +
\underset{\P\in\Pc}{\Sup}\, \E^{\P}\left[\underset{0\leq t\leq T}{\Sup}
|\widetilde K_t^{n}-\widetilde K_t^{m}|^2\right] \\
&\leq C \|\xi_n-\xi_m\|_{\L^{2}}^2 +
\widetilde C\|\xi_n-\xi_m\|_{\L^{2}}\underset{n,m \rightarrow +\infty}{\longrightarrow} 0. \end{align*} 
Extracting a  fast-converging subsequence and using Borel-Cantelli Lemma, we can make sure that, $\P-a.s.$%\begin{align} \|Y^n-Y^m\|_{\D^{2}}^2 +\|Z^n-Z^m\|^2_{\H^{2}} + \underset{\P\in\Pc}{\Sup}\,\E^{\P}[\underset{0\leq t\leq T}{\Sup} | K_t^{n,\P}-K_t^{m,\P}|^2] \leq 2^{-n}\label{subsequence}
%\end{align} 
%for all $m\geq n \geq 0$. This implies by Markov inequality that
%for all $\P$ and all $m\geq n \geq 0$ \begin{align} \P\Big[\underset{0\leq t\leq
%T}{\Sup}\big[|Y^n_t-Y^m_t|^2+| K_t^{n,\P}-K_t^{m,\P}|^2\big] +
%\Int_0^T|Z^n_t-Z^m_t |^2dt > \Frac{1}{n}\Big]\leq Cn2^{-n}.
%\label{markov}\end{align} 
$$\underset{n\rightarrow\infty}{\Lim} \left[\underset{0\leq t\leq
T}{\Sup}\big[|Y^n_t-Y^m_t|^2+| \widetilde K_t^{n}-\widetilde K_t^{m}|^2\big] +
\Int_0^T\|\widehat a_t^{1/2}(Z^n_t-Z^m_t) \|^2dt \right] =0.$$ 
Define then \begin{align*}
Y:=\underset{n\rightarrow\infty}{\overline{\Lim}} Y^n ,\quad
Z:=\underset{n\rightarrow\infty}{\overline{\Lim}} Z^n , \quad
 K:=\underset{n\rightarrow\infty}{\overline{\Lim}} \widetilde K^{n} ,\end{align*}
It is
therefore clear that $(Y,Z, K)\in \D^2\times\H^2\times\I^2$.%Finally, we can proceed exactly as in the regular case ($\xi\in{UC}_b(\Omega)$) to show that the minimum condit
\ep
\end{proof}

\section{Probabilistic interpretation for fully-nonlinear SPDEs}
\label{Probabilistic interpretation for Fully nonlinear SPDEs}
The aim of this section is to give a Feynman--Kac's formula for the solution of the following fully non-linear SPDEs
\begin{align}\label{spde}
\begin{cases}
du(t,x) + \hat{h}(t,x,u(t,x),Du(t,x),D^2u(t,x))dt + g(t,x,u(t,x),Du(t,x))\circ d\W_t = 0,\\
u(T,x)= \phi(x),
\end{cases}
\end{align}
where we consider the case $H_t(\omega,y,z,\gamma)= h(t,B_t(\omega),y,z,\gamma)$, with $h:[0,T]\times\R\times\R^d\times D_h\longrightarrow\R$ (with $D_h$ being a subset of $\mathbb S_d^{>0}$) is a deterministic map. Then, the corresponding conjuguate and bi--conjuguate functions are given by 
\begin{align}
F(t,x,y,z,a) &:= \underset{\gamma\in D_h}{\Sup}\left\{\Frac{1}{2} {\rm Tr}[a\gamma] - h(t,x,y,z,\gamma)\right\}~\text{for}~ a\in \S_d^{>0},\\
\hat{h}(t,x,y,z,\gamma) &:=  \underset{a\in \S_d^{>0}}{\Sup} \left\{\Frac{1}{2} {\rm Tr}[a\gamma] - F(t,x,y,z,a)\right\}~\text{for}~ \gamma\in \R^{d\times d}.
\end{align} 
Notice that $-\infty< \hat{h}\leq h$ and $\hat{h}$ is nondecreasing convex in $\gamma$. Also, $\hat{h}=h$ if and only if $h$ is convex and nondecreasing in $\gamma$, which we will therefore always assume.

\vspace{0.3em}
\noindent For this end, the following markovian $2$BDSDE is considered 
\begin{align}
Y_s^{t,x} =&\ \phi(B_T^{t,x})-\Int_s^T F(s,B_r^{t,x},Y_r^{t,x},Z_r^{t,x},\widehat{a}_r)dr+\Int_t^T g(r,B_r^{t,x},Y_r^{t,x},Z_r^{t,x})\circ d\W_r
 \nonumber\\
 &- \Int_s^T Z_r^{t,x} dB_r^{t,x}+K_T^{t,x}-K_s^{t,x}, ~ t\leq s\leq
 T,~\Pc^t-q.s,
\label{eqmarkovian} \end{align}
where for any $(t,x)\in[0,T]\times\R^d$, $(B_s^t)_{s\in[t,T]}$ is the shifted canonical process on $\Omega^{B,t}$ defined by
$$ B_s^{t,x}:= x+B_s^t \quad \text{for all}~ s\in[t,T].$$
The stochastic integral with respect to $\W$ is the Stratonovich backward integral (see Kunita \cite{K90} page 194). Using the definition of the Stratonovich backward integral, we can show easily that equation (\ref{eqmarkovian}) is equivalent to the following $2$BDSDE
\begin{align}
Y_s^{t,x} =&\ \phi(B_T^{t,x})-\Int_s^Tf(s,B_r^{t,x},Y_r^{t,x},Z_r^{t,x},\widehat{a}_r)dr+\Int_t^T g(r,B_r^{t,x},Y_r^{t,x},Z_r^{t,x}) \cdot d\W_r
 \nonumber\\
 &- \Int_s^T Z_r^{t,x}dB_r^{t,x}+K_T^{t,x}-K_s^{t,x}, ~ t\leq s\leq
 T,~\Pc^{t}-q.s
\label{eqmarkovianmod} \end{align} 
where $$f(s,x,y,z,\widehat{a}_s):= F(s,x,y,z,\widehat{a}_s)+\Frac{1}{2}{\rm Tr}[g(s,x,y,z)D_yg(s,x,y,z)^\top].$$

\vspace{0.5em}
\noindent From now on, we focus our study on providing the probabilistic representation of the classical and stochastic viscosity solutions for the fully nonlinear SPDEs \eqref{spde} via $2$BDSDEs \eqref{eqmarkovian}.
%Let $\tau$ be a $\G^t$-stopping time, $\P\in\Pc^{t}$, and $\eta$ a $\P$-square integrable $\Gc^t_{\tau}$-measurable random variable. We denote by $(y^{\P},z^{\P}):= (y^{\P,t,x}(\tau,\eta),z^{\P,t,x}(\tau,\eta))$ the solution of the following BDSDE under $\P$
%\begin{align*}
%y_s^{\P}=\eta -\Int_s^{\tau} F(r,B_r^{t,x},y_r^{\P},z_r^{\P},\widehat{a}_r)dr - \Int_s^{\tau} g(r,B_r^{t,x},y_r^{\P},z_r^{\P})\cdot d\W_r -\Int_s^{\tau}z_r^{\P}\cdot dB_r, ~ t\leq r\leq \tau,
%\end{align*}
%which is well-posed under our assumptions.
Let us first define the following functional spaces:
\begin{itemize}
\item $\Mc_{0,T}^W$ denotes all the $\F^W-$stopping times $\tau$ such that $0\leq \tau\leq T$.
\item  $L^{p}(\Fc_{\tau,T}^W;\R^d)$, for $p\geq 0$, denotes the space of all $\R^d-$valued $\Fc_{\tau,T}^W-$ measurable r.v. $\xi$ such that $\E[|\xi|^p]< +\infty$.
\item  $C^{\ell,k}([0,T]\times\R^d)$, for $k,\ell\geq 0$, denotes the space of all $\R-$valued functions defined on $[0,T]\times\R^d$, which are $\ell-$times continuously differentiable in $t$, $k-$times continuously differentiable in $x$.
\item  $C_b^{k,m,n}([0,T]\times\R^d\times\R; \R^p)$, for $k, m, n\geq 0$, $p\geq 1$, denotes the space of all $\R^p-$valued functions defined on $[0,T]\times\R^d\times\R$, which are $k-$times continuously differentiable in $t$, $m-$times continuously differentiable in $x$, $n-$times continuously differentiable in $y$ and have uniformly bounded partial derivatives.
\item  $C^{\ell,k}(\Fc_{t,T}^W,[0,T]\times\R^d)$, for $k,\ell\geq 0$, denotes the space of all $C^{\ell,k}([0,T]\times\R^d)-$valued random variables $\varphi$ that are $\Fc_{t,T}^W\otimes\Bc([0,T]\times\R^d)-$measurable.
\item  $C^{\ell,k}(\F^W,[0,T]\times\R^d)$, for $k,\ell\geq 0$, denotes the space of r.v. $\varphi\in C^{\ell,k}(\Fc_{t,T}^W,[0,T]\times\R^d)$ such that for fixed $x\in\R^d$, the mapping $(t,\omega)\longmapsto\varphi(t,x,\omega)$ is $\F^W-$progressively measurable.
\end{itemize}
Furthermore, for $(t,x,y)\in[0,T]\times\R^d\times\R$, we denote $\partial/\partial y= D_y, \partial/\partial t= D_t, D= D_x= (\partial/\partial x_1,\cdots,\partial/\partial x_d),$ and $D^2=D_{xx}= (\partial^2_{x_{i}x_{j}})_{i,j=1}^{d}$. The meaning of $D_{xy}, D_{yy}$,$\dots$, should be clear.

\vspace{0.5em}
\noindent Then, we list the assumptions needed in this section. The following is a slight strengthening of Assumption \ref{ass}, where we assume a bit more regularity.
\begin{Assumption}
\item[\rm{(i)}] $\Pc$ is not empty, the domain $D_{F_t(y,z)}= D_{F_t}$is independent of $(w,y,z)$. Moreover, $F$, $g$ and $D_yg$ are uniformly continuous in $t$, uniformly in $a$ on $D_{F_t}$.
\item[\rm{(ii)}] There exist constants $C > 0, 0\leq\alpha <1$ such that for all $(t,a,x,x^{\prime},z,z^{\prime},y,y^{\prime})t\in [0,T]\times D_{F_t}\times(\R^d)^4\times\R^2$
\begin{align*}
 |F(t,x,y,z,a)- F(t,x^{\prime},y^{\prime},z^{\prime},a)|&\leq  C\big(|x-x^{\prime}|+|y-y^{\prime}|+\|a^{1/2}(z-z^{\prime})\|\big),\\
\|g(t,x,y,z)-g(t,x^{\prime},y^{\prime},z^{\prime})\|^2&\leq  C\left(|x-x^{\prime}|^2+ |y-y^{\prime}|^2\right)+\alpha\|(z-z^{\prime})\|^2.
\end{align*}
\item[\rm{(iii)}] The function $g$ belongs to $C_b^{0,2,3}([0,T]\times\R^d\times\R; \R^l)$.
\item[\rm{(iv)}] There exists a constant $\lambda\in[0,1[$ such that
$$(1-\lambda)\widehat{a}_t \geq \alpha I_d,\ dt\times \Pc-q.e. $$

\label{assspde}
\end{Assumption}

\noindent We next state a strengthened version of Assumption \ref{ass1} in the present Markov framework.
\begin{Assumption}\label{assspde1} 
\begin{itemize}
\item[\rm{(i)}] The function $\phi$ is a uniformly continuous and bounded function on $\R^d$.
\item[\rm{(ii)}] For any $(t,x)\in[0,T]\times\R^d$ and for some $\eps>0$
\begin{align*}
& \underset{\P\in\Pc^{t}}{\Sup}\E^{\P}\left[|\phi(B_T^{t,x})|^{2+\eps}+ \Int_t^T|F(s,B_s^{t,x},0,0,\widehat{a}_s^t)|^{2+\eps}ds+ \Int_t^T \|g(s,B_s^{t,x},0,0)\|^{2+\eps}ds\right] <+\infty.
\end{align*}
\end{itemize}
\end{Assumption}
\vspace{0.2cm}
\noindent Therefore  under Assumptions \ref{assumtion g}, \ref{assspde} and \ref{assspde1} and according to Theorem \ref{Thexistence}, there exists a unique triplet 
$ (Y^{t,x},Z^{t,x},K^{t,x}) $ solution of the $2$BDSDE \eqref{eqmarkovian}. Indeed, it is immediate to check that $f$ and $g$ indeed satisfy Assumption \ref{ass} (recall that $D_yg$ is bounded) and that the terminal condition also verifies all the required regularity and integrability properties.
\subsection{Classical solution of SPDEs}
We can rewrite the SPDE (\ref{spde}) in its so-called integral form, as soon as $\{u(t,x), \ 0\leq t\leq T, x\in\R^d\}\in C^{0,2}(\Fc_{t,T}^W,[0,T]\times\R^d)$ is a classical solution of the following equation  where the stochastic integral is written in the Stratonovich form, namely,
\begin{align}
u(t,x)= \phi(x)+\Int_t^T \hat{h}(t,x,u(t,x),Du(t,x),D^2u(t,x))dt+\Int_t^T g(t,x,u(t,x),Du(t,x))\circ d\W_t.
\label{spdefint}
\end{align}
\begin{Definition}
We define a classical solution of the SPDE \reff{spde} as a $\R-$valued random field $\{u(t,x), \ (t,x)\in[0,T]\times\R^d\}$ such that $u(t,x)$ is $\Fc_{t,T}^W-$measurable for each $(t,x)$, and whose trajectories belong to $C^{0,2}([0,T]\times\R^d)$.
\end{Definition}
\noindent The following is a version of the celebrated Feynman--Kac formula in the present context.
\begin{Theorem}
\label{representationsolutionclassique:theorem}
Let Assumption \ref{assspde} hold true. Suppose further that $H$ is continuous in its domain, $D_F$ is independent of $t$ and is bounded both from above and away from $0$. Let 
$\{u(t,x),\  (t,x)\in[0,T]\times\R^d\}$ be a classical solution of \reff{spde} with $\{(u,Du)(s,B_s^{t,x}), ~s\in[t,T]\} \in\D^{2}\times \H^{2}$. Then
$$Y_s^{t,x}:= u(s,B_s^{t,x}), ~ Z_s^{t,x}:= Du(s,B_s^{t,x}), ~ K_s^{t,x}:= \Int_0^s k_rdr,$$
with
$$k_s:= \hat{h}(s,B_s,Y_s,Z_s,\Gamma_s)-\Frac{1}{2} {\rm Tr}[\widehat{a}_s\Gamma_s] + F(s,B_s,Y_s,Z_s,\widehat{a}_s)~ \text{and}~ \Gamma_s:=D^2u(s,B_s^{t,x}),$$
is the unique solution of the $2$BDSDE \reff{eqmarkovian}. Moreover, $u(t,x)=Y_t^{t,x}$ for all $t\in[0,T]$.
\end{Theorem}
\begin{proof}
It suffices to show that $(Y,Z,K)$ solves the $2$BDSDE (\ref{eqmarkovian}). Let $s=t_0< t_1<t_2<...<t_n=T$, then writing $B$ instead of $B^{t,x}$ and $B_i$ instead of $B_{t_i}$ for notational simplicity, we have
\begin{align*}
&\sum_{i=0}^{n-1}[u(t_i,B_i)-u(t_{i+1},B_{i+1})]
= \sum_{i=0}^{n-1}[u(t_i,B_{i})-u(t_{i},B_{{i+1}})]+\sum_{i=0}^{n-1}[u(t_i,B_{{i+1}})-u(t_{i+1},B_{{i+1}})]\\
&= -\sum_{i=0}^{n-1}\Int_{t_i}^{t_{i+1}} Du(t_i,B_r)dB_r-\sum_{i=0}^{n-1}\Int_{t_i}^{t_{i+1}}\Frac{1}{2} {\rm Tr}[\widehat{a}_rD^2u(t_i,B_r)]dr\\
&\hspace{0.9em}+ \sum_{i=0}^{n-1}\Int_{t_i}^{t_{i+1}} \hat{h}(r,B_{{i+1}},u(r,B_{{i+1}}^{t,x}),Du(r,B_{{i+1}}),D^2u(r,B_{{i+1}}))dr\\
&\hspace{0.9em}+\sum_{i=0}^{n-1}\Int_{t_i}^{t_{i+1}}g(r,B_{{i+1}},u(r,B_{{i+1}}),Du(r,B_{{i+1}}))\circ d\W_s,
\end{align*}
where we have used the It\^o formula and the Equation $(\ref{spdefint})$ satisfied by $u$. Now, the transformation from Stratonovich to It\^o integral yields
\begin{align*}
&\sum_{i=0}^{n-1}[u(t_i,B_{i})-u(t_{i+1},B_{{i+1}})]= -\sum_{i=0}^{n-1}\Int_{t_i}^{t_{i+1}} Du(t_i,B_r)dB_r-\sum_{i=0}^{n-1}\Int_{t_i}^{t_{i+1}}\Frac{1}{2} {\rm Tr}[\widehat{a}_rD^2u(t_i,B_r)]dr\\
&\hspace{0.9em}+ \sum_{i=0}^{n-1}\Int_{t_i}^{t_{i+1}}\hat{h}(r,B_{{i+1}},u(r,B_{{i+1}}),Du(r,B_{{i+1}}),D^2u(r,B_{{i+1}}))dr\\
&\hspace{0.9em} + \sum_{i=0}^{n-1}\Int_{t_i}^{{i+1}}g(r,B_{{i+1}}),u(r,B_{{i+1}}),Du(r,B_{{i+1}})) d\W_r+ \sum_{i=0}^{n-1}\Int_{t_i}^{t_{i+1}} F(r,B_{{i+1}},u(r,B_{{i+1}}),Du(r,B_{{i+1}}),\widehat{a}_r)dr\\
&\hspace{0.9em}- \sum_{i=0}^{n-1}\Int_{t_i}^{t_{i+1}} F(r,B_{{i+1}},u(r,B_{{i+1}}),Du(r,B_{{i+1}}),\widehat{a}_r)dr\\
&\hspace{0.9em}-  \Frac{1}{2}\sum_{i=0}^{n-1}\Int_{t_i}^{t_{i+1}}{\rm Tr}[g(r,B_{{i+1}},u(r,B_{{i+1}}),Du(r,B_{{i+1}}))Dg(r,B_{{i+1}},u(r,B_{{i+1}})),Du(r,B_{{i+1}})]dr.
\end{align*}
It then suffices to let the mesh size go to zero to obtain that the processes $(Y,Z,K)$ we have defined do satisfy Equation \reff{eqmarkovian}. It now remains to prove the minimum condition
\begin{align}
\underset{\P^{'}\in\Pc(t+,\P)}{\rm ess \, inf^{\P}}\E_t^{\P^{'}}\left[\Int_t^T k_sds\right]=0 ~ \text{for all}~ t\in[0,T],~ \P\in\Pc,
\label{mincond}
\end{align}
by which we can conclude that $(Y,Z,K)$ is a solution of the $2$BDSDE (\ref{eqmarkovian}), provided that (\ref{mincond}) holds. However, the proof of that (technical) point can actually be carried out exactly as in \cite[Theorem 5.3]{STZ10} or \cite[Theorem 5.3]{kazi2015second}. Indeed, the main point is that one has to be able to construct appropriate strong solutions to some SDEs on $\Omega^B$, similar to the ones in Example 4.5 of \cite{sone:touz:zhan:11a}. In our framework, this construction can be carried about for every fixed $\omega^W$, and it then suffices to use once more the results of Stricker and Yor \cite{SY78}.
\ep
\end{proof}
\subsection{Stochastic viscosity solution for SPDE}
\label{Stochastic viscosity solution for SPDE:section}
The aim of this section is to give a probabilistic representation for the stochastic viscosity solutions of the following fully non-linear SPDEs via solutions of  2BDSDEs \eqref{eqmarkovian}. We restrict our study  to the following class of SPDEs where the coefficient $g$ does not depends on the gradient of the solution, 
\begin{align}\label{spde:viscosité}
\begin{cases}
du(t,x) + \hat{h}(t,x,u(t,x),Du(t,x),D^2u(t,x))dt + g(t,x,u(t,x))\circ d\W_t = 0,\\
u(T,x)= \phi(x),
\end{cases}
\end{align}
As mentioned in the introduction, Lions and Souganidis have introduced a notion of stochastic viscosity solution for fully nonlinear SPDEs  in   \cite{lion:soug:98, lion:soug:00, lion:soug:01} motivated by applications in path--wise stochastic control problems and the associated stochastic HJB equations.  Buckdahn and Ma \cite{buck:ma:10a,buck:ma:10b} have introduced the rigorous notion of stochastic viscosity solution for semi--linear SPDEs and have then given the probabilistic interpretation of such equation via BDSDEs, where the intensity of the noise $g$ in the SPDEs \eqref{spde:viscosité} does not depend on the gradient of the solution. As mentioned in the introduction, they used the so--called Doss--Sussmann transformation and stochastic diffeomorphism flow technics to convert the semi--linear SPDEs to PDEs with random coefficients. This transformation permits to remove the stochastic integral term from the SPDEs and then gives a rigorous definition of so--called stochastic viscosity solution for SPDEs. We have to mention that it is difficult to define viscosity solution for SPDEs due to the fact that there are no maximum principle for solutions of SPDEs, because of the presence of the stochastic integral term in the equation. Since we are following a similar approach, this explains why we also assume that the non--linearity $g$ is not impacted by the gradient term. 

\vspace{0.5em}
\noindent We will use the shifted probability spaces defined in Section \ref{A direct existence argument}. We now introduce the random function $u:[0,T]\times \Omega^W\times \R^d\longrightarrow \R$ given by
\begin{align}
u(t,x):=Y_t^{t,x}= \underset{\P\in\Pc^{t}}{\Sup}\ y_t^{\P,t,x}, ~\text{for}~ (t,x)\in[0,T]\times \R^d,
\label{probarep}
\end{align}
where for any $(t,x,\P)\in[0,T]\times\R^d\times\Pc^t$, $(y^{\P,t,x},z^{\P,t,x})$ is the unique solution of the BDSDE
\begin{align*}
y_s^{\P,t,x} =&\ \phi(B_T^{t,x})-\Int_s^T f(r,B_r^{t,x},y_r^{\P,t,x},z_r^{\P,t,x},\widehat{a}_r)dr+\Int_t^T g(r,B_r^{t,x},y_r^{\P,t,x},z_r^{\P,t,x}) \cdot d\W_r
 \nonumber\\
 &- \Int_s^T z_r^{\P,t,x}\cdot dB_r^{t,x}, ~ t\leq s\leq
 T,~\P-a.s.
\end{align*}
By the Blumenthal $0-1$ law, it follows that $u(t,x)$ is deterministic with respect to $B$, but still an $\F^W-$adapted process.
%As we have explained in the existence of the solution for $2$BDSDE in the abstract setting, we suppose that our candidate $u$ defined by \eqref{probarep} belongs to $C(\Fc_{t,T}^W,[0,T]\times\R^d)$.\\

\begin{Theorem}
Let Assumptions \ref{assumtion g}, \ref{assspde} and \ref{assspde1} hold true. Then $u$ belongs to $C(\Fc_{t,T}^W,[0,T]\times\R^d)$.
\end{Theorem}
\begin{proof}
Let us start with the uniform continuity in $x$. For any $(t,x), (t,x^{\prime})\in[0,T]\times \R^d$, we have for any $\P\in\Pc^t$
\begin{align*}
\E^{\P_0^W}[|u(t,x)-u(t,x^{\prime})|^2]&= \E^{\P_0^W}\left[|\underset{\P\in\Pc^{t}}{\Sup} y_t^{\P,t,x}-\underset{\P\in\Pc^{t}}{\Sup} y_t^{\P,t,x^{\prime}}|^2\right]\leq  \underset{\P\in\Pc^{t}}{\sup} \E^{\P_0^W}\left[|y_t^{\P,t,x}-y_t^{\P,t,x^{\prime}}|^2\right].
\end{align*}
But by the classical {\it a priori} estimates for BDSDEs, and using in particular the fact that $\phi$ is Lipschitz continuous, and $f$ and $g$ are Lipschitz continuous in $x$, we have for any $p\in[1,2]$
\begin{align*}
\E^{\P_0^W}\left[|y_t^{\P,t,x}-y_t^{\P,t,x^{\prime}}|^p\right]&\leq C\mathbb E^{\P_0^W}\left[|x-x'|^p+\int_t^T|B_s^{t,x}-B_s^{t,x'}|^pds\right]\leq C\No{x-x'}^p.
\end{align*}
By Kolmogorov--Chentsov's Theorem, this implies immediately that a $\P_0^W-$version of $u$ is $(1/2-\eta)-$H\"older continuous in $x$, for any $\eta\in(0,1/2)$.
 
\vspace{0.3em}
\noindent Next, for any $(t,t^{\prime},x)\in[0,T]^2\times \R^d$, we have the following classical dynamic programming result
$$Y_{t^\prime}^{t,x}=u(t^\prime,B_{t^\prime}^{t,x}),$$
which implies that for any $\P\in\Pc^t$ and any $\eps'\in(0,\eps)$, using the estimates of Theorem \ref{thestiamte}(i)
\begin{align*}
&\E^{\P_0^W}[|u(t^{\prime},x)-u(t,x)|^{2+\eps'}]\\
&= \E^{\P_0^W}\left[\abs{u(t^{\prime},x)-u(t^{\prime},B_{t^{\prime}}^{t,x})+Y_{t^{\prime}}^{t,x}-Y_t^{t,x}}^{2+\eps'}\right]\\
&\leq  C \E^{\P_0^W}[|u(t^{\prime},x)-u(t^{\prime},B_{t^{\prime}}^{t,x})|^{2+\eps'}]+C\abs{\E^{\P_0^W}\left[Y_{t^{\prime}}^{t,x}-Y_t^{t,x}\right]}^{2+\eps'}\\
&\leq   C\E^{\P_0^W}|x-B_{t^{\prime}}^{t,x}|^{2+\eps'}]+C\abs{\E^{\P_0^W}\left[\int_t^{t^\prime}{f}(r,B_r^{t,x},Y_r^{t,x},Z_r^{t,x},\widehat{a}_r)dr+\int_t^{t^\prime}dK_r^{t,x}\right]}^{2+\eps'}\\
&\leq C|t-t^\prime|^{1+\frac{\eps'}{2}}+C\underset{\P\in\Pc^{t}}{\Sup}\E^{\P_0^W}\left[\left(\int_t^{t^\prime}|{f}(r,B_r^{t,x},Y_r^{t,x},Z_r^{t,x},\widehat{a}_r)|dr\right)^{2+\eps'}+\left(\int_t^{t^\prime}dK_r^{t,x}\right)^{2+\eps'}\right]\\
&\leq C|t-t^\prime|^{1+\frac{\eps'}{2}}\left(1+\No{Y^{t,x}}^{1+\frac{\eps'}{2}}_{\mathbb D^{2+\eps'}}+\No{Z^{t,x}}^{1+\frac{\eps'}{2}}_{\mathbb H^{2+\eps'}}+\underset{\P\in\Pc^{t}}{\Sup}\E^{\P_0^W}\left[\left(K_T^{t,x}+\int_0^T|B_s^{t,x}|ds\right)^{2+\eps'}\right]\right)\\
&\leq C|t-t'|^{1+\frac{\eps'}{2}}.
\end{align*}
By Kolmogorov--Chentsov Theorem, this implies immediately that a $\P_0^W-$version of $u$ is $\eta-$H\"older continuous in $t$, for any $\eta\in(0,\eps/(2(2+\eps)))$.
\ep
\end{proof}
\subsubsection{Stochastic flow and definitions}
We follow Buckdahn and Ma \cite{buck:ma:10a}. The definition of our stochastic viscosity solution will depend on the following stochastic flow $\eta\in C(\F^W,[0,T]\times\R^d\times\R)$, defined as the unique solution of the (SDE)
\begin{align}\label{floweta}
\eta(t,x,y) = y + \Int_t^T g(s,x,\eta(s,x,y))\circ d\W_s, ~ 0\leq t\leq T.
\end{align}
Under Assumption \ref{assspde}, for fixed $x$ the random field $\eta(.,x,.)$ is continuously differentiable in the variable $y$, and the mapping $y\longmapsto\eta(t,x,y,\omega)$ defines a diffeomorphism for all $(t,x)$, $\P-a.s.$ 

\vspace{0.5em}
%\noindent {\color{red} Est-ce que $\eta$ et ses d\'eriv\'ees sont born\'ees ?? On a aussi besoin que $D_y\eta$ soit born\'e loin de $0$ pour que $\tilde f$ soit bien d\'efinie. J'ai tr\`es peur de la Proposition 5.3, parce que si on n'a pas mieux que \c{c}a, je suis \`a peu pr\`es s\^ur que $\tilde f$ ne sera pas à croissance quadratique, il va y avoir des termes qui vont faire intervenir des $\exp(CB_t)$, devant le $z$ et devant le $y$. Et on en sait pas faire \c{c}a...}
%
\vspace{0.5em}
\noindent We denote by $\mathcal{E}(t,x,y)$ the $y-$inverse of $\eta(t,x,y)$, so $\mathcal{E}(t,x,y)$ is the solution of the following first-order SPDE
\begin{align}\label{flowinverse}
\mathcal{E}(t,x,y)= y -\Int_t^T D_y\mathcal{E}(s,x,y)g(s,x,y)\circ d\W_s, ~ \forall (t,x,y),\ \P-a.s.
\end{align}
We note that $\mathcal{E}(t,x,\eta(t,x,y))= \mathcal{E}(T,x,\eta(T,x,y))=y , ~ \forall (t,x,y)$. We now define the notion of stochastic viscosity solution for SPDE (\ref{spde}).
\begin{Definition}\label{Definitionviscosity}
$(i)$ A random field $u\in C(\F^W,[0,T]\times\R^d)$ is called a stochastic viscosity subsolution $($resp. supersolution$)$ of SPDE \eqref{spde:viscosité}, if 
$$u(T,x)\leq (\text{resp.} \geq)\phi(x), ~ \forall x\in\R^d,$$ 
and if for any $\tau\in\Mc_{0,T}^W,$ $\zeta\in L^0(\Fc_{\tau,T}^W;\R^d)$, and any random field $\varphi\in C^{1,2}(\Fc_{\tau,T}^W,[0,T]\times\R^d)$ satisfying
$$u(t,x)-\eta(t,x,\varphi(t,x))\leq (\text{resp.} \geq)~ 0= u(\tau,\zeta)-\eta(\tau,\zeta,\varphi(\tau,\zeta),$$
for all $(t,x)$ in a neighborhood of $(\tau,\zeta),\ \P_W^0-a.e.$ on the set $\{0< \tau < T\}$,  it holds that 
$$ - \hat{h}(\tau,\zeta,\psi(\tau,\zeta),D\psi(\tau,\zeta),D^2\psi(\tau,\zeta))\leq (\text{resp.} \geq) D_y\eta(\tau,\zeta,\varphi(\tau,\zeta))D_t\varphi(\tau,\zeta),$$
$\P-a.e.$ on $\{0< \tau < T\}$, where $\psi(t,x):=\eta(t,x,\varphi(t,x)).$

\vspace{0.3em}
\noindent
$(ii)$ A random field $u\in C(\F^W,[0,T]\times\R^d)$ is called a stochastic viscosity solution of SPDE \eqref{spde:viscosité}, if it is both a stochastic viscosity subsolution and a supersolution.
\end{Definition}
\begin{Definition}
\label{wise-viscosity:definition}
A random field $u\in C(\F^W,[0,T]\times\R^d)$ is called a $\omega-$wise viscosity  $($sub--, super--$)$ solution, if for $\P-a.e.\ \omega\in\Omega,\ u(\omega,\cdot)$ is a $($deterministic$)$ viscosity $($sub--, super--$)$ solution of the SPDE \eqref{spde:viscosité}.
\end{Definition}
\begin{Remark}
If we assume that $\varphi\in C^{1,2}(\F^W,[0,T]\times\R^d)$, and that $g\in C^{0,0,3}([0,T]\times\R^d\times\R; \R^l)$, then a straightforward computation using the It\^o--Ventzell formula shows that the random field $\psi(t,x)=\eta(t,x,\varphi(t,x))$ satisfies
\begin{align}
d\psi(t,x)=D_y\eta(t,x,\varphi(t,x))D_t\varphi(t,x)dt +  g(t,x,\psi(t,x))\circ  d\W_t , \ t\in[0,T]. 
\end{align}
Since $g(\tau,\zeta,\psi(\tau,\zeta))=g(\tau,\zeta,u(\tau,\zeta))$ by defintion, it seems natural to compare $$\hat{h}(\tau,\zeta,\psi(\tau,\zeta),D\psi(\tau,\zeta),D^2\psi(\tau,\zeta)), \ \text{with}\ D_y\eta(\tau,\zeta,\varphi(\tau,\zeta))D_t\varphi(\tau,\zeta),$$
to characterize a viscosity solution of SPDE $(\hat{h},g)$.

\vspace{0.3em}
\noindent If the function $g\equiv 0$ in SPDE \eqref{spde}, the flow $\eta$ becomes $\eta(t,x,y)=y$, $\forall (t,x,y)$ and $\psi(t,x)=\varphi(t,x)$. Thus the definition of a stochastic viscosity solution becomes the same as that of a deterministic viscosity solution $($see, e.g. Crandall, Ishii and Lions {\rm\cite{CIL92}}$)$.
\end{Remark}
\noindent One of  the main results of our paper is the following probabilistic representation of stochastic viscosity solution for fully nonlinear SPDEs, which is, to the best of our knowledge, the first result of this kind for such a class of SPDEs. The proof will be obtained in the subsequent subsections.
\begin{Theorem}
\label{representationsolutionviscosité:theorem}
Let Assumptions \ref{assspde}, \ref{assspde1} and \ref{extraasumption} hold true and  $( Y_s^{t,x},  Z_s^{t,x}, K_s^{t,x}) $ be the unique solution of the $2$BDSDE \reff{eqmarkovian}. Then, 
$$u(t,x)=Y_t^{t,x} = \underset{\P\in\Pc^{t}}{\Sup} y_t^{\P,t,x}(T,\phi(B_T^{t,x})),  \ \P_W^0-a.s,\ \text{for all $(t,x)\in[0,T]\times \R^d$},$$ 
 is a stochastic viscosity solution of SPDE \eqref{spde:viscosité}. Moreover,  
$$
u(t,x)=\eta(t,x,v(t,x)),~ v(t,x)=\mathcal{E}(t,x, u(t,x)),  \ \P_W^0-a.s.,  \ \mbox{and}  ~v(t,x)=U_t^{t,x},
$$
where $(U^{t,x},V^{t,x},\tilde{K}^{t,x})$ is a solution of the following $2$BSDE, for all $t\leq s\leq T$,
\begin{align}
U_s^{t,x}= \phi(B_T^{t,x})-\Int_s^T\tilde{f}(r,B_r^{t,x},Y_r^{t,x},Z_r^{t,x},\widehat{a}_r)dr - \Int_s^T V_r^{t,x}\cdot dB_r + \tilde{K}_T^{t,x}-\tilde{K}_s^{t,x}, 
\label{2BSDE}
\end{align}
with $\tilde{f}:[0,T]\times\R^d\times\R\times\R^d\times D_f\longrightarrow\R$ defined by
\begin{align}
\label{Doss-Sussmann-coefficient}
\tilde{f}(t,x,y,z,a):=&\  \Frac{1}{D_y\eta(t,x,y)}\left(f(t,x,\eta(t,x,y),D_y\eta(t,x,y)z+D_x\eta(t,x,y),a)-\Frac{1}{2}{\rm Tr}[a D_{xx}\eta(t,x,y)]\right.\nonumber\\
&\left.- z^\top a D_{xy}\eta(t,x,y)-\Frac{1}{2}{\rm Tr}\left[D_{yy}\eta(t,x,y)a^{1/2}zz^\top a^{1/2}\right]\right).
\end{align}
\end{Theorem}

\subsubsection{Doss--Sussmann transformation}
\label{Doss-Sussmann-section}
In this subsection, we use the so--called Doss--Sussmann transformation  to convert the  fully nonlinear SPDEs \eqref{spde:viscosité} to PDEs with random coefficients. This transformation permits to remove the martingale term from the SPDEs. To begin with, let us note that, under Assumption \ref{assspde} (iii), the random field $\eta\in C^{0,2,2}(\F^W,[0,T]\times\R^d\times\R)$, thus so is $\Ec$. Now for any random field $\psi:[0,T]\times \R^d\times \Omega\longrightarrow \R$, consider the transformation introduced in Definition \ref{Definitionviscosity}
$$\varphi(t,x)=\Ec(t,x,\psi(t,x)), \ (t,x)\in [0,T]\times \R^d,$$
or equivalently, $\psi(t,x)=\eta(t,x,\varphi(t,x)).$ One can easily check that $\psi\in C^{0,p}(\F^W, [0,T]\times \R^d)$ if and only if $\varphi\in C^{0,p}(\F^W, [0,T]\times \R^d)$, for $p= 0,1,2.$ Moreover, if $\varphi\in C^{0,2}(\Fc^W, [0,T]\times \R^d)$, then $$D_x\psi= D_x\eta + D_y\eta D_x\varphi,$$

\vspace{-2.3em}

\begin{align}\label{deriveeseconde}
D_{xx}\psi= D_{xx}\eta + 2(D_{xy}\eta)(D_x\varphi)^{\top}+ (D_{yy}\eta)(D_x\varphi)(D_x\varphi)^{\top}+(D_{y}\eta)(D_{xx}\varphi).
\end{align}
Furthermore, since $\Ec(t,x,\eta(t,x,y))\equiv y, \ \forall (t,x,y),\ \P-a.s.$, differentiating the equation up to the second order we have (suppressing variables fro simplicity), for all $(t,x,y)$ and $\P-a.s.$,
\begin{align} 
\begin{split}
& D_x\Ec + D_y\Ec D_x\eta = 0,\ D_y\Ec D_y\eta = 1,\\
& D_{xx}\Ec + 2(D_{xy}\Ec)(D_x\eta)^{\top}+ (D_{yy}\Ec)(D_x\eta)(D_x\eta)^{\top}+(D_{y}\Ec)(D_{xx}\eta)=0,\\
&(D_{xy}\Ec)(D_y\eta)+ (D_{yy}\Ec)(D_x\eta)(D_y\eta)+(D_{y}\Ec)(D_{xy}\eta)=0,\\
&(D_{yy}\Ec)(D_y\eta)^2+(D_{y}\Ec)(D_{yy}\eta)=0.
\end{split}
\end{align}

\noindent The following additional assumption  is needed to study the growth of the random fields  $ \eta$  and $\Ec$ (see \cite{buck:ma:10a}, p.188--189).
\begin{Assumption}\label{extraasumption}
For any $\varepsilon >0$, there exists a function $G^{\varepsilon}\in C^{1,2,2,2}([0,T]\times \R^n\times\R^d\times \R)$, such that 
$$\Frac{\partial G^{\varepsilon}}{\partial t}(t,w,x,y)=\varepsilon,~ \Frac{\partial G^{\varepsilon}}{\partial w^i}= g^i(t,x,G^{\varepsilon}(t,w,x,y)),~ i=1,\cdots,n, \text{ and}~ G^{\varepsilon}(0,0,x,y)=y.$$
\end{Assumption}

\vspace{0.5em}
\begin{Proposition}\label{prop:prop}
Let $\eta$ be the unique solution to SDE \eqref{floweta} and $\Ec$ be the $y-$inverse of $\eta$ $($the solution to \eqref{flowinverse}$)$. Then, under Assumption \ref{extraasumption}, there exists a constant $C>0$, depending only on the bound of $g$ and its partial derivatives, such that for $\zeta=\eta, \Ec$, it holds for all $(t,x,y)$ and $\P_0^W-a.s.$ that
\begin{align*}
&|\zeta(t,x,y)\leq |y|+C |W_t|,\\
&|D_{x}\zeta|+|D_{y}\zeta|+|D_{xx}\zeta|+|D_{xy}\zeta|+|D_{yy}\zeta|\leq C \exp{\left(C|W_t|\right)},
\end{align*}
where all the derivatives are evaluated at $(t,x,y)$.
\end{Proposition}
\noindent The proof of this proposition is done in \cite{buck:ma:10a}, p.189--191, so we omit it. Now, we will use the Doss transformation to transform SPDE (\ref{spde}) to PDE with random coefficients and we obtain the following proposition where the proof follows the lines of the proof of Proposition 3.1. in \cite{buck:ma:10a} (p. 187-188).
% {\color{red} Je ne comprends pas ce que vous dites. Le r\'esultat de Buckdahn et Ma marche pour les EDSRs. Dans le cas second ordre, il faut prendre les fenchel des g\'en\'erateurs, non ?}
\begin{Proposition}
Let Assumptions \ref{assspde}, \ref{assspde1} and \ref{extraasumption} hold true. A random field $u$ is a stochastic viscosity sub- $($resp. super-$)$ solution to SPDE \eqref{spde:viscosité} if and only if $v(.,.)= \mathcal{E}(.,.,u(.,.))$ is a stochastic viscosity solution to the following PDE with random coefficients
\begin{align}\label{pde:viscosité}
\begin{cases}
dv(t,x) + \tilde{h}(t,x,v(t,x),Dv(t,x),D^2v(t,x))dt = 0,\\
v(T,x)= \phi(x),
\end{cases}
\end{align} 
with
\begin{align*} 
\tilde{h}(t,x,y,z,\gamma):=& \sup_{a\in\mathbb S_d^{>0}}\left\{\frac12{\rm Tr}\left[a\gamma\right]-\tilde f(t,x,y,z,a)\right\}+ \Frac{1}{2}{\rm Tr}[g(s,x,\eta(s,x,y))D_yg(s,x,\eta(s,x,y))^\top].
\end{align*}
Consequently, $u$  is a stochastic viscosity solution of SPDE \eqref{spde:viscosité} if and only if $v(.,.)= \mathcal{E}(.,.,u(.,.))$ is a stochastic viscosity solution to the PDE with random coefficients \eqref{pde:viscosité}.

\end{Proposition}
\begin{proof} Let $u\in C(\F^W,[0,T]\times\R^d)$ be a stochastic viscosity subsolution of SPDE \eqref{spde:viscosité} and let $v$ be defined by $v(t,x)= \mathcal{E}(t,x,u(t,x))$. In order to show that $v$ is  a stochastic viscosity subsolution to the PDE \eqref{pde:viscosité}, we let $\tau\in\Mc_{0,T}^W,$ $\zeta\in L^0(\Fc_{\tau,T}^W;\R^d)$ be arbitrary given, and let $\varphi\in C^{1,2}(\Fc_{\tau,T}^W,[0,T]\times\R^d)$ be such that
$$v(t,x)-\varphi(t,x)\leq  0= v(\tau,\zeta)-\varphi(\tau,\zeta),$$
for all $(t,x)$ in a neighborhood of $(\tau,\zeta),\ \P_W^0-a.e.$ on the set $\{0< \tau < T\}$.\\
Now, we define $\psi(t,x)=\eta(t,x,\varphi(t,x)),\, \forall (t,x)\,\, \P_W^0-a.e.$ Since $y\longmapsto \eta(t,x)$ is strictly increasing, we have
\begin{align}
u(t,x)-\psi(t,x)&=\eta(t,x,v(t,x))-\eta(t,x,\varphi(t,x))\nonumber\\
&\leq 0=\eta(\tau,\zeta,v(\tau,\zeta))-\eta(\tau,\zeta,\varphi(\tau,\zeta))= u(\tau,\zeta)-\psi(\tau,\zeta),
\end{align}
for all $(t,x)$ in a neighborhood of $(\tau,\zeta),\ \P_W^0-a.e.$ on the set $\{0< \tau < T\}$. Then, since $u$ is a stochastic viscosity subsolution of SPDE \eqref{spde:viscosité}, we have $\P_W^0-a.e.$ on $\{0< \tau < T\}$,
\begin{equation}\label{subsolution}
 - \hat{h}(\tau,\zeta,\psi(\tau,\zeta),D\psi(\tau,\zeta),D^2\psi(\tau,\zeta))\leq  D_y\eta(\tau,\zeta,\varphi(\tau,\zeta))D_t\varphi(\tau,\zeta).
\end{equation}
On the other hand, we recall the expression of $\hat{h}$,
\begin{align*}
\hat{h}(t,x,y,z,\gamma) &:=  \underset{a\in \S_d^{>0}}{\Sup} \left\{\Frac{1}{2} {\rm Tr}[a\gamma] - F(t,x,y,z,a)\right\}\\
&= \underset{a\in \S_d^{>0}}{\Sup} \left\{\Frac{1}{2} {\rm Tr}[a\gamma] - f(s,x,y,z,\widehat{a}_s)+\Frac{1}{2}{\rm Tr}[g(s,x,y)D_yg(s,x,y)^\top].\right\}
\end{align*}
Thus, using \eqref{deriveeseconde}, we have
\begin{align*}
{\rm Tr}[aD^2\psi(\tau,\zeta)]=&\ {\rm Tr}[aD_{xx}\eta(\tau,\zeta,\varphi(\tau,\zeta))]+{\rm Tr}[(D_x\varphi(\tau,\zeta))^{\top}a(D_{xy}\eta(\tau,\zeta,\varphi(\tau,\zeta))]\\
&+ {\rm Tr}[a(D_{yy}\eta(\tau,\zeta,\varphi(\tau,\zeta)))(D_x\varphi(\tau,\zeta))(D_x\varphi(\tau,\zeta))^{\top}]+ {\rm Tr}[a (D_{y}\eta(\tau,\zeta,\varphi(\tau,\zeta)))(D_{xx}\varphi(\tau,\zeta))],
\end{align*}
Finally, plugging the above calculations in \eqref{subsolution} and appealing to \eqref{Doss-Sussmann-coefficient}, we conclude that
\begin{equation}
 - \tilde{h}(\tau,\zeta,\varphi(\tau,\zeta),D\varphi(\tau,\zeta),D^2\varphi(\tau,\zeta))\leq  D_t\varphi(\tau,\zeta),
\end{equation} 
which is the desired result. The reciprocal part of the proposition can be proved in a similar way.\\
\ep 
\end{proof}
\noindent 
We apply now Doss--Sussmann transformation for the $2$BDSDE (\ref{eqmarkovian}) so that the Stratonovich backward integral vanishes. Thus, the $2$BDSDE will become a $2$BSDE (a standard one) with a new generator $\tilde{f}$ \eqref{Doss-Sussmann-coefficient}, which is quadratic in $z$. A similar class of 2BSDEs has been studied by Possama{\"{\i}} and Zhou \cite{PZ13} and Lin \cite{L14} in the case of a bounded final condition $\phi(B_T^{t,x})$ and a generator $F$ satisfying (see Assumption 2.1. (iv) p. 3776 in \cite{PZ13})
$$|F_t(x,y,z,a)|\leq \alpha+\beta|y|+\frac{\gamma}{2}|a^{1/2}z|^2,$$
for some positive constants $\alpha$, $\beta$ and $\gamma$. In our case, for fixed $\omega^W$, thanks to Proposition \ref{prop:prop}, we know that the generator $\tilde f$ satisfies a similar assumption, so that we can apply the results of \cite{PZ13} or \cite{L14}.

\vspace{0.5em}
\noindent Let us then define the following three processes
\begin{align}
U_s^{t,x} &:= \mathcal{E}(t,B_s^{t,x},Y_s^{t,x}),\nonumber \\
V_s^{t,x} &:= D_y\mathcal{E}(s,B_s^{t,x},Y_t^{t,x})Z_s^{t,x}+D_x\mathcal{E}(s,B_s^{t,x},Y_s^{t,x}),\nonumber \\
\tilde{K}_s^{t,x} &:= \Int_0^s D_y\mathcal{E}(r,B_r^{t,x},Y_r^{t,x})dK_r^{t,x}.
\label{dosstransf}
\end{align}
\begin{Theorem} Let Assumptions \ref{assspde},  \ref{assspde1} and \ref{extraasumption} hold true. Then
$(U^{t,x},V^{t,x},\tilde{K}^{t,x})$ is  the unique solution of the $2$BSDE \eqref{2BSDE}. 
\end{Theorem}
\begin{proof}
It is easily checked that the mapping $(B,Y,Z,K)\longmapsto (B,U,V,\tilde{K})$ is one--to--one, and admits as an inverse
\begin{align}
Y_t=\eta(t,B_t,U_t),\ Z_t=D_y\eta(t,B_t,U_t)V_t+D_x\eta(t,B_t,U_t),\ K_t = \Int_0^t D_y\eta(s,B_s,U_s)d\tilde{K}_s.
\label{inversetransf}
\end{align}
Consequently, the uniqueness of (\ref{2BSDE}) follows from that of $2$BDSDE (\ref{eqmarkovian}), thanks to (\ref{dosstransf}) and (\ref{inversetransf}). Thus we need only show that $(U,V,\tilde{K})$ is a solution of the $2$BSDE (\ref{2BSDE}). Applying the generalized It\^o--Ventzell formula (see Lemma \ref{ItoVentzellLemma} below) to $\mathcal{E}(t,B_t,Y_t)$, one derives that for any $(t,x)\in [0,T]\times\R^d$
\begin{align*}
U_t = \mathcal{E}(t,B_t,Y_t)=&\ \phi(B_T) - \Int_t^T D_y\mathcal{E}(s,B_s,Y_s)f(s,B_s,Y_s,Z_s,\widehat{a}_s)ds\\
&-\Int_t^T D_x\mathcal{E}(s,B_s,Y_s)\cdot dB_s-\Int_t^T D_y\mathcal{E}(s,B_s,Y_s)Z_s\cdot dB_s + \Int_t^T D_y\mathcal{E}(s,B_s,Y_s)dK_s \\
&-\Frac{1}{2}\Int_t^T {\rm Tr}[D_{xx}\mathcal{E}(s,B_s,Y_s)\widehat{a}_s]ds-\Frac{1}{2}\Int_t^T {\rm Tr}[D_{yy}\mathcal{E}(s,B_s,Y_s)\widehat{a}_s^{1/2}Z_sZ_s^\top\widehat a_s^{1/2}]ds\\
&-\Int_t^T {\rm Tr}[D_{xy}\mathcal{E}(s,B_s,Y_s)Z_s^\top\widehat{a}_s]ds\\
=&\ \phi(B_T)-\Int_t^T\mathcal{H}(s,B_s,Y_s,Z_s,\widehat{a}_s)ds- \Int_t^T V_s\cdot dB_s + \tilde{K}_T-\tilde{K}_t,
\end{align*}
where 
\begin{align*} 
\mathcal{H}(s,x,y,z,a):= &\ (D_y\mathcal{E})f(s,x,y,z,a)+\Frac{1}{2} {\rm Tr}[(D_{xx}\mathcal{E)}a]+\Frac{1}{2} {\rm Tr}[(D_{yy}\mathcal{E}){a}^{1/2}zz^\top a^{1/2}]+{\rm Tr}[(D_{xy}\mathcal{E})z^\top a].
\end{align*}
Next, we can show that
\begin{align} 
\mathcal{H}(s,B_s,Y_s,Z_s,\widehat{a}_s)=\tilde{f}(s,B_s,U_s,V_s,\widehat{a}_s),\  \forall s\in[0,T],~ \P-a.s.
\end{align}
similarly as done in Buckdahn and Ma $\cite{buck:ma:10a}$ (proof of Theorem 5.1. page 198--199).

\vspace{0.3em}
\noindent The process $\tilde{K}$ is a non-decreasing process thanks to the fact that $y\longmapsto \eta(t,x,y)$ is strictly increasing and the non-decreasing of $K$. Now, it remains to prove the minimum condition (\ref{eqmin}) for the process $\tilde{K}$. Notice that by Proposition \ref{prop:prop}, and since $\widehat a$ is bounded under any of the measure we consider, it is clear that $D_y\mathcal{E}(r,B_r^{t,x},Y_r^{t,x})$ has moments of any order under any $\P$. Therefore, we can argue exactly as in Step (ii) of the proof of Theorem \ref{theorep} to show that $\tilde K$ inherits the required minimality condition directly from $K$.
\ep
\end{proof}

\vspace{0.5em}
\noindent We are now ready for the proof of our main theorem

\vspace{0.5em}
\proof[Proof of Theorem \ref{representationsolutionviscosité:theorem}]
First, we introduce the random field $v(t,x)=U_t^{t,x}$, where $U$ is the solution of $2$BSDE (\ref{2BSDE}).
% {\color{red} On ne sait pas si elle existe...}. 
 Then by (\ref{dosstransf}) and (\ref{inversetransf}) we know that, for $(t,x)\in[0,T]\times\R^d$
\begin{align}
u(t,x)=\eta(t,x,v(t,x)),\, v(t,x)=\mathcal{E}(t,x,u(t,x)).
\label{transf}
\end{align}
 Thanks to Proposition 5.1, we know that we only need to prove that the random field $v$ defined in (\ref{transf}) is a viscosity solution of the PDE with random coefficients \eqref{pde:viscosité}. The idea is then to follow the proof in \cite[Theorem 7.3]{PZ13}, which itself follows \cite[Theorem 5.11]{STZ10}, to prove that the solution of $2$BSDE (\ref{2BSDE}) $v(t,x)=U_t^{t,x}$ is an $\omega-$wise viscosity solution of the PDE \eqref{pde:viscosité} (recall Definition \ref{wise-viscosity:definition}), which then ends the proof. It suffices to notice that the fact that $f$ satisfies Assumptions \ref{assspde} and Assumption \ref{assspde} implies that, for fixed $\omega^W$, $\tilde f$ satisfies Assumption 7.1 of \cite{PZ13}.
\ep

%\begin{Remark} {\color{blue}On a plus besoin???}
%As mentioned in the beginning of Section \ref{Stochastic viscosity solution for SPDE:section}, we suppose that our candidate $u$ defined by \eqref{probarep} belongs to $C(\Fc_{t,T}^W,[0,T]\times\R^d)$. Thanks to the Relation \eqref{transf}, this condition is equivalent to $v(\cdot,\cdot,\omega)$ belongs to $C([0,T]\times\R^d)$, $\P-a.s.$ for all $\omega\in\Omega$.  This latter can be proved by using regularity results for deterministic fully nonlinear PDEs $($\cite{CC95},\cite{L96}$)$, but it is still a work in progress to make it rigorous.
%\end{Remark}

\begin{appendix}
\section{Technical results}\label{Appendix Chapter2}
%\subsection{$L^p$ estimates for backward doubly SDEs}
\subsection{It\^o and It\^o--Ventzell formulae}
\label{proof of lemmaitô} 
 The following It\^o's formula is a mix between the classical forward and backward It\^o's formulas and is similar to Lemma 1.3 in \cite{pp1994}. We give it here for ease of reference and completeness. The proof being standard, we omit it.
 \begin{Lemma}\label{lemma:ito}
 Let $X^1$ and $X^2$ be defined, for $i=1,2$, by
 $$X^i_t=X_0^i+\int_0^t\alpha_s^ids+\int_0^t\beta_s^i\cdot dB_s+\int_0^t\gamma_s^i\cdot d\W_s+K^i_t,\ 0\leq t\leq T,\ \P-a.s.,$$
 for some c\`adl\`ag bounded variation and $\G-$progressively measurable processes $K^i$, such that one of them is continuous. We then have
\begin{align*}
X^1_t X^2_t=&\ X^1_0X^2_0 +\int_0^t\left(\widehat a_s^{1/2}\beta_s^1\cdot\widehat a_s^{1/2}\beta_s^2  -\gamma_s^1\cdot\gamma_s^2 +\alpha_s^1X_s^2+\alpha_s^2X_s^1\right)ds+\int_0^t\left(X_s^2\beta_s^1+X_s^1\beta_s^2\right)\cdot dB_s\\
 &+\int_0^t\left(X_s^1\gamma_s^2+X_s^2\gamma_s^1\right)\cdot d\W_s+\int_0^tX^1_{s^-}dK^2_s+\int_0^tX^2_{s^-}dK_s^1,\ t\in[0,T],\ \P-a.s.
 \end{align*}
 \end{Lemma}
\noindent We now give a generalized version of It\^o--Ventzell formula that combines the generalized It\^o formula of Pardoux and Peng \cite{pp1994} and the It\^o--Ventzell formula of Ocone and Pardoux \cite{OP89}.
\begin{Lemma}\label{ItoVentzellLemma}$($Generalized It\^o--Ventzell formula$)$

\vspace{0.5em}
\noindent Suppose that $F\in C^{0,2}(\F,[0,T]\times\R^k)$ is a semimartingale with spatial parameter $x\in\R^k$:
\begin{align*}
\begin{split}
F(t,x)= F(0,t)& + \Int_0^t G(s,x)ds +\Int_0^t H(s,x)\cdot  dB_s+ \Int_0^t K(s,x) \cdot d\W_s, \quad t\in[0,T], 
\end{split}
\end{align*}
where $G\in C^{0,2}(\F^B,[0,T]\times\R^k)$, $H\in C^{0,2}(\F^B,[0,T]\times\R^k;\R^d)$ and $K\in C^{0,2}(\F^W,[0,T]\times\R^k;\R^l)$. Let $\phi\in C(\F,[0,T];\R^k)$ be a process of the form 
$$\phi_t= \phi_0 + A_t + \Int_0^t\gamma_s \cdot dB_s +\Int_0^t \delta_s \cdot d\W_s, \quad t\in[0,T],$$
where $\gamma\in \H^2_{k\times d}$, $\delta\in \H^2_{k\times l}$ and $A$ is a continuous $\F$-adapted process with paths of locally bounded variation. Then, $\P$-almost surely, it holds for all $0\leq t\leq T$ that
\begin{align} \label{ItoVentzell}
\nonumber F(t,\phi_t)= &\ F(0,x) + \Int_0^t G(s,\phi_s)ds +\Int_0^t H(s,\phi_s) \cdot dB_s+ \Int_0^t K(s,\phi_s) \cdot d\W_s\\
\nonumber &+ \Int_0^t D_x F(s,\phi_s)dA_s + \Int_0^t D_x F(s,\phi_s)\gamma_s \cdot dB_s + \Int_0^t D_x F(s,\phi_s)\delta_s \cdot d\W_s\\
\nonumber & +\Frac{1}{2}  \Int_0^t {\rm Tr}(D_{xx} F(s,\phi_s)\gamma_s \gamma_s^\top) ds - \Frac{1}{2}  \Int_0^t {\rm Tr}[D_{xx} F(s,\phi_s)\delta_s \delta_s^\top] ds\\
& +\Int_0^t {\rm Tr}(D_{x} H(s,\phi_s)\gamma_s^\top) ds - \Int_0^t {\rm Tr}[D_{x} F(s,\phi_s) \delta_s^\top] ds.
\end{align} 
\end{Lemma}
\subsection{Proof of Lemma \ref{BDSDEshift}}
\label{proof of lemmashift} 
We divide the proof in two steps.

\vspace{0.5em}
{\it Step 1:} We start by showing the result in the case where $F$ and $g$ do not depend on $(y,z)$. In this case, we can solve directly the BDSDEs to find that $\text{for }\P-a.e. \ (\omega^B,\omega^W)\in\Omega$
\begin{equation}\label{eq:shift1}
y_t^{\P}(\omega^B,\omega^W)=\E^{\P}\left[\left.\xi+\int_t^T\widehat F_sds+\int_t^Tg_s\cdot d\W_s\right|\mathcal G_t\right](\omega^B,\omega^W).
\end{equation}
Then, since $\xi$ is actually $\Fc_T^B-$measurable, we deduce immediately, using the definition of the r.c.p.d. that for $\P_0^W-$a.e. $\omega^W\in\Omega^W$
$$\E^{\P}\left[\left.\xi\right|\mathcal G_t\right](\omega^B,\omega^W)=\E^{\P_B(\omega^W)}\left[\left.\xi\right|\mathcal F_t^B\right](\omega^B)=\E^{\P_B^{t,\omega^B}(\omega^W)}\left[\xi^{t,\omega^B}\right].$$
Next, we know from the results of Stricker and Yor \cite{SY78} that we can define a measurable map from $(\Omega^W\times[0,T],\Fc_T\otimes\Bc([0,T]))$ to $(\R,\Bc(\R))$ which coincides $\P^0_W\otimes dt-$a.e. with the conditional expectation of $g_s$, under $\P^B(\cdot;\omega^W)$ (remember that this is a stochastic kernel, and thus measurable), with respect to the $\sigma-$algebra $\Fc_t^B$. For notational simplicity, we still denote this map as
$$(\omega^W,s)\longmapsto \E^{\P_B(\cdot;\omega^W)}\left[\left.g_s(\cdot,\omega^W)\right|\Fc_t^B\right].$$
In other words, the above map does indeed define a stochastic process. That being said, we claim that for $\P-$a.e. $\omega\in\Omega$
\begin{align}\label{eq:shift2}
\nonumber\E^{\P}\left[\left.\int_t^Tg_s\cdot d\W_s\right|\mathcal G_t\right](\omega^B,\omega^W) &=\left(\int_t^T\E^{\P_B(\cdot;\cdot)}\left[\left.g_s\right|\Fc_t^B\right](\omega^B,\cdot)\cdot d\W_s\right)(\omega^W)\\
&=\left(\int_t^T\E^{\P_B^{t,\omega^B}(\cdot)}\left[g_s^{t,\omega^B}\right](\cdot)\cdot d\W_s\right)(\omega^W).
\end{align}
To prove the claim, let us first show it in the case where $g$ is a simple process with the following decomposition
$$g_t(\omega^B,\omega^W)=\sum_{i=0}^{n-1}g_{t_i}(\omega^B,\omega^W){\bf 1}_{(t_i,t_{i+1}]}(t).$$
Then, we have by definition of backward stochastic integrals, for $\P$-a.e. $(\omega^B,\omega^W)\in\Omega$
\begin{align*}
\E^{\P}\left[\left.\int_t^Tg_s\cdot d\W_s\right|\mathcal G_t\right](\omega^B,\omega^W)&=\sum_{i=0}^{n-1}\E^{\P}\left[\left.g_{t_{i+1}}\cdot \left(W_{t_{i+1}\wedge t}-W_{t_i\wedge t}\right)\right|\mathcal G_t\right]\left(\omega^B,\omega^W\right)\\
&=\sum_{i=0}^{n-1}\E^{\P}\left[\left.g_{t_{i+1}}\right|\mathcal G_t\right]\left(\omega^B,\omega^W\right)\cdot \left(W_{t_{i+1}\wedge t}-W_{t_i\wedge t}\right)(\omega^W).
\end{align*}
Notice next that for $\P-a.e.$ $(\omega^B,\omega^W)\in\Omega$
$$\E^{\P}\left[\left.g_{t_{i+1}}\right|\mathcal G_t\right]\left(\omega^B,\omega^W\right)=\E^{\P_B(\cdot;\omega^W)}\left[\left.g_{t_{i+1}}\left(\cdot,\omega^W\right)\right|\Fc_t^B\right]\left(\omega^B\right).$$
Indeed, for any $X$ which is $\mathcal G_t-$measurable, we have
\begin{align*}
&\int_{\Omega}\E^{\P_B(\cdot;\omega^W)}\left[\left.g_{t_{i+1}}\left(\cdot,\omega^W\right)\right|\Fc_t^B\right]\left(\omega^B\right)X(\omega^B,\omega^W)d\P_B(\omega^B;\omega^W)d\P^0_W(\omega^W)\\
&=\int_{\Omega^W}\left(\int_{\Omega^B}\E^{\P_B(\cdot;\omega^W)}\left[\left.g_{t_{i+1}}\left(\cdot,\omega^W\right)\right|\Fc_t^B\right]\left(\omega^B\right)X(\omega^B,\omega^W)d\P_B(\omega^B;\omega^W)\right)d\P^0_W(\omega^W)\\
&=\int_{\Omega^W}\left(\int_{\Omega^B}g_{t_{i+1}}\left(\omega^B,\omega^W\right)X(\omega^B,\omega^W)d\P_B(\omega^B;\omega^W)\right)d\P^0_W(\omega^W)\\
&=\int_{\Omega}g_{t_{i+1}}\left(\omega^B,\omega^W\right)X(\omega^B,\omega^W)d\P(\omega^B,\omega^W),
\end{align*}
where we have used the fact that since for every $\omega^W\in\Omega^W$, $\omega^B\longmapsto X(\omega^B,\omega^W)$ is $\Fc_t^B-$measurable, we have by definition of the conditional expectation that
\begin{align*}
&\int_{\Omega^B}\E^{\P_B(\cdot;\omega^W)}\left[\left.g_{t_{i+1}}\left(\cdot,\omega^W\right)\right|\Fc_t^B\right]\left(\omega^B\right)X(\omega^B,\omega^W)d\P_B(\omega^B;\omega^W)\\
&=\int_{\Omega^B}g_{t_{i+1}}\left(\omega^B,\omega^W\right)X(\omega^B,\omega^W)d\P_B(\omega^B;\omega^W).
\end{align*}
Hence, we deduce finally that
\begin{align*}
\E^{\P}\left[\left.\int_t^Tg_s\cdot d\W_s\right|\mathcal G_t\right](\omega^B,\omega^W)&=\sum_{i=0}^{n-1}\E^{\P_B(\cdot;\omega^W)}\left[\left.g_{t_{i+1}}\left(\cdot,\omega^W\right)\right|\Fc_t^B\right]\left(\omega^B\right)\cdot \left(W_{t_{i+1}\wedge t}-W_{t_i\wedge t}\right)(\omega^W)\\
&=\left(\int_t^T\E^{\P_B(\cdot;\cdot)}\left[\left.g_s\right|\Fc_t^B\right](\omega^B,\cdot)\cdot d\W_s\right)(\omega^W).
\end{align*}
By a simple density argument, we deduce that the same holds for general processes $g$. Next, notice that by definition of r.p.c.d., we have for $\P_0^W-$a.e. $\omega^W\in\Omega^W$
$$\E^{\P_B(\cdot;\omega^W)}\left[\left.g_s\right|\Fc_t^B\right](\omega^B,\omega^W)=\E^{\P_B^{t,\omega^B}(\cdot;\omega^W)}\left[g_s^{t,\omega^B}\right](\omega^W),\ \text{for }\P_B(\cdot;\omega^W)-a.e.\ \omega^B\in\Omega^B.$$
By definition of $\P$, we are exactly saying that the above holds for $\P-$a.e. $\omega\in\Omega$. This finally proves \reff{eq:shift2}. 

\vspace{0.5em}
\noindent Using similar argument, we show that we also have for $\P-$a.e. $\omega\in\Omega$
$$\E^{\P}\left[\left.\int_t^T\widehat F_sds\right|\mathcal G_t\right](\omega^B,\omega^W)=\int_t^T\E^{\P_B^{t,\omega^B}(\cdot;\omega^W)}\left[\widehat F_s^{t,\omega^B}\right](\omega^W)ds.$$
To sum up, we have obtained that for $\P-$a.e. $\omega\in\Omega$
\begin{align*}
y_t^{\P}(\omega^B,\omega^W)=&\ \E^{\P_B^{t,\omega^B}(\cdot;\omega^W)}\left[\xi^{t,\omega^B}\right]+\left(\int_t^T\E^{\P_B(\cdot;\cdot)}\left[\left.g_s\right|\Fc_t^B\right](\omega^B,\cdot)\cdot d\W_s\right)(\omega^W)\\
&+\int_t^T\E^{\P_B^{t,\omega^B}(\cdot;\omega^W)}\left[\widehat F_s^{t,\omega^B}\right](\omega^W)ds.
\end{align*}
But, we also have (remember that by the Blumenthal $0-1$ law $y_t^{\P^{t,\omega^B}_B(\cdot)\otimes\P^0_W,t,\omega^B}$ only depends on $\omega^W$) for any $\omega^B\in\Omega^B$ and for $\P_0^W-$a.e. $\omega^W\in\Omega^W$
\begin{align*}
y_t^{\P^{t,\omega^B}_B(\cdot)\otimes\P^0_W,t,\omega^B}(\omega^W)&=\E^{\P_B^{t,\omega^B}(\cdot)\otimes\P^0_W}\left[\left.\xi^{t,\omega^B}+\int_t^T\widehat F^{t,\omega^B}_sds+\int_t^Tg_s^{t,\omega^B}\cdot d\W_s\right|\mathcal G^t_t\right](\omega^W)\\
&=\E^{\P_B^{t,\omega^B}(\cdot)\otimes\P^0_W}\left[\left.\xi^{t,\omega^B}+\int_t^T\widehat F^{t,\omega^B}_sds+\int_t^Tg_s^{t,\omega^B}\cdot d\W_s\right|\mathcal F^W_{t,T}\right](\omega^W).
\end{align*}
Using the same arguments as above, we obtain
\begin{align}\label{eq:tructruc}
\nonumber y_t^{\P^{t,\omega^B}_B(\cdot)\otimes\P^0_W,t,\omega^B}(\omega^W)=&\ \E^{\P_B^{t,\omega^B}(\cdot,\omega^W)}\left[\xi^{t,\omega^B}\right]+\int_t^T\E^{\P_B^{t,\omega^B}(\cdot,\omega^W)}\left[\widehat F^{t,\omega^B}_s\right](\omega^W)ds\\
&+\left(\int_t^T\E^{\P_B^{t,\omega^B}(\cdot,\cdot)}\left[g_s^{t,\omega^B}\right]\cdot d\W_s\right)(\omega^W),
\end{align}
which proves the desired result.

\vspace{0.3em}
{\it Step 2: } Since we are in a Lipschitz setting, solutions to BDSDEs can be constructed via Picard iterations. Hence, using Step 1, the results holds at each step of the iteration and therefore also when passing to the limit. We emphasize that this step crucially relies on \reff{eq:aa}.
\ep
\subsection{Proof of Lemma \ref{Lemmaregularity}}
For each $\P\in\Pc$, let $(\overline{\Yc}^\P(T,\xi), \overline{\Zc}^\P(T,\xi))$ be the solution of the BDSDE with generators $\widehat{F}$ and $g$, and terminal condition $\xi$ at time $T$. We define $\widetilde{V}^\P :=  V- \overline{\Yc}^\P(T,\xi)$. Then, $\widetilde{V}^\P\geq 0, \ \P - a.s.$
For any $0\leq t_1\leq t_2\leq T$, let $(y^{\P,t_2},z^{\P,t_2}):= (\Yc^\P(t_2,V_{t_2}), \Zc^\P(t_2,V_{t_2}))$. Note that $$\Yc_{t_1}^\P(t_2,V_{t_2})(\omega)= \Yc_{t_1}^{\P,t_1,\omega}(t_2,V_{t_2}^{t_1,\omega}), \ \P - a.s.$$ 
Then by the dynamic programming principle (Theorem \ref{dynam prog}) we get $$V_{t_1}\geq y^{\P,t_2}_{t_1}, \ \P - a.s.$$
Denote $ \widetilde{y}^{\P,t_2}_t:= y^{\P,t_2}_{t}- \overline{\Yc}^\P_t  ,~ \widetilde{z}^{\P,t_2}_t:= \widehat{a}_t^{-1/2}(z^{\P,t_2}_{t}- \overline{\Zc}^\P_t).$
Then $(\widetilde{y}^{\P,t_2},\widetilde{z}^{\P,t_2})$ is solution of the following BDSDE on $[0,t_2]$
$$\widetilde{y}^{\P,t_2}_t= \widetilde{V}^\P_t+\Int_t^{t_2} f_s^\P(\widetilde{y}^{\P,t_2}_s,\widetilde{z}^{\P,t_2}_s)ds + \Int_t^{t_2} \widehat{g}_s^\P(\widetilde{y}^{\P,t_2}_s,\widetilde{z}^{\P,t_2}_s)\cdot d\W_s -  \Int_t^{t_2} \widetilde{z}^{\P,t_2}_s\cdot \widehat a^{1/2}_s dB_s,$$
where 
\begin{align*}
f_t^\P(\omega, y,z) &:= \widehat{F}_t(\omega,y+ \overline{\Yc}^\P_t(\omega),\widehat{a}^{1/2}(\omega)z+ \overline{\Zc}^\P_t(\omega))- \widehat{F}_t(\omega,\overline{\Yc}^\P_t(\omega),\overline{\Zc}^\P_t(\omega))\\
\widehat{g}_t^\P(\omega, y,z) &:= g_t(\omega,y+ \overline{\Yc}^\P_t(\omega),\widehat{a}^{1/2}(\omega)z+ \overline{\Zc}^\P_t(\omega))- g_t(\omega,\overline{\Yc}^\P_t(\omega),\overline{\Zc}^\P_t(\omega)).
\end{align*}
Then $\widetilde{V}^\P_{t_1}\geq \widetilde{y}^{\P,t_2}_{t_1}$. Therefore, $\widetilde{V}^\P$ is a positive weak doubly $f^\P-$super--martingale under $\P$ by Definition \ref{doublymartingale} (given below in the Appendix).\\
Now, we assume that the coefficient ${g}$ does not depend in $ (y,z)$,   then obviously we have  that   $\bar{V}^\P_{t_1} \geq \bar{y}^{\P,t_2}_{t_1}$ where
$$\bar{y}_t :=  \widetilde{y}_t +   \Int_0^{t} \widehat{g}_s \cdot d\W_s \text{ and }\bar{V}_t :=  \widetilde{V}_t +   \Int_0^{t} \widehat{g}_s \cdot d\W_s.$$  
Thanks to this change of variable, we have that $(\bar{y}^{\P,t_2},\bar{z}^{\P,t_2})$ solves the following standard BSDE on $[0,t_2]$
$$\bar{y}^{\P,t_2}_t= \bar{V}^\P_t+\Int_t^{t_2} \bar{f}_s^\P(\widetilde{y}^{\P,t_2}_s,\widetilde{z}^{\P,t_2}_s)ds -  \Int_t^{t_2} \bar{z}^{\P,t_2}_s\cdot \widehat a^{1/2}_s dB_s,$$
where 
$$\bar{f}_t^\P(\omega, y,z) := f_t(\omega,y+  \Int_0^{t} \widehat{g}_s \cdot d\W_s ,\widehat{a}^{1/2}(\omega)z ).$$
Now applying the down--crossing inequality for $f$-martingale Theorem 6 in \cite{CP2000} combined  with the result concerning the classical down--crossing inequality for non necessarily positive  super--martingales in \cite{Doob84} (chapter III,  p. 446), we deduce that for $\P- a.e.$ $\omega$, the limit $\underset{r\in\Q\cap(t,T],r\downarrow t}{\Lim} \bar{V}^\P _r$, and consequently the limit $\underset{r\in\Q\cap(t,T],r\downarrow t}{\Lim} \widetilde{V}^\P _r$ exists for all $t\in [0,T]$. Note that $y^\P$ is continuous, $\P-a.s.$,  and obviously $\bar{y}^\P$ is continuous, $\P-a.s$. Therefore, we get that the $\overline{\Lim}$ in the definition of $V^+$ is in fact a true limit, which implies that
$$V_t^{+}= \underset{r\in\Q\cap(t,T],r\downarrow t}{{\Lim}}V_r, \ \Pc -q.s.,$$ and thus
 $V^{+}$ is c\`adl\`ag
$\Pc -q.s.$
Finally,  we can prove the general case when $g$ depend in $ (y,z)$ using classically the Banach  fixed point theorem.
\ep
\section{Doubly $f-$supersolution and martingales}
In this section, we extend some of the results of Peng \cite{peng:99} concerning $f-$super--solutions of BSDEs to the case of BDSDEs. In the following, we fix a probability measure $\P\in\Pc$ and work implicitly with $\overline{\F^B}^\P$ and $\F^W$.
We introduce the following spaces for a fixed probability $\P$.
\begin{itemize}
\item[-] $L^{2}(\P)$ denotes the space of all $\Fc_T-$measurable scalar r.v. $\xi$ with
$\|\xi\|^2_{L^{2}}:= \E^{\P}[|\xi|^2] < +\infty .$
\item[-] $\D^{2}(\P)$ denotes the space of $\R-$valued processes $Y$, s.t. $Y_t$ is $\Fc_{t}$ measurable for every $t\in[0,T]$, with
$\text{ c\`adl\`ag paths, and } \ \|Y\|^2_{\D^{2}(\P)}:= \E^{\P}\left[\underset{0\leq t\leq T}{\Sup}\,|Y_t|^2 \right]< +\infty .$
\item[-] $\H^{2}(\P)$ denotes the space of all $\R^d-$valued processes $Z$  s.t. $Z_t$ is $\Fc_{t}$ measurable for a.e. $t\in[0,T]$, with
$$\|Z\|^2_{\H^{2}(\P)}:=  \E^{\P}\left[\left(\Int_0^T \|\widehat{a}^{1/2}_tZ_t\|^2 dt\right)\right]< +\infty .$$
\end{itemize}

Let us be given the following objects
\begin{description}
 \item[(i)] a terminal condition $\xi$ which is $\Fc_T-$measurable and in $L^2(\P)$.
 \item[(ii)] two maps $f:\Omega\times\R\times\R^d\rightarrow\R ,~ g:\Omega\times\R\times\R^d\rightarrow\R^l$ verifying
 
 \vspace{0.5em}
$\bullet$ $\E\left[\Int_0^T |f(t,0,0)|^2dt\right] < +\infty$, and $\E\left[\Int_0^T\|g(t,0,0)\|^2dt\right] < +\infty.$

\vspace{0.5em}
$\bullet$ There exist $(\mu,\alpha)\in\R_+^*\times (0,1)$ s.t. for any $(\omega,t,y_1,y_2,z_1,z_2)\in\Omega\times[0,T]\times\R^2\times(\R^d)^2$
\begin{align*}
|f(t,\omega, y_1,z_1)-f(t,\omega,y_2,z_2)| &\leq \mu\big(|y_1-y_2|+\|z_1-z_2\|\big),\\
\|g(t,\omega,y_1,z_1)-g(t,\omega,y_2,z_2)\|^2 &\leq c|y_1-y_2|^2+\alpha\|z_1-z_2\|^2.
\end{align*}
\item[(iii)] a real--valued c\`adl\`ag, progressively measurable process $\{V_t, 0\leq t\leq T\}$ with
$$\E\left[\underset{0\leq t\leq T}{\Sup}|V_t|^2\right] < +\infty.$$
\end{description}
We want to study the following problem: to find a pair of processes $(y,z)\in\D^2(\P) \times\H^2(\P)$  satisfying 
\begin{align}
y_t =\xi_T +\Int_t^T f_s(y_s,z_s)ds  +\Int_t^T g_s(y_s,z_s)\cdot d\W_s +V_T-V_t- \Int_t^T z_s\cdot dB_s ~,~\P-a.s.\label{BDSDE f-super}
\end{align}
We have the following existence and uniqueness theorem
\begin{Proposition}
Under the above hypothesis there exists a unique pair of processes $(y,z)\in\D^2(\P) \times\H^2(\P)$ solution of BDSDE \eqref{BDSDE f-super}. 
\end{Proposition}
\begin{proof}
In the case where $V\equiv0$, the proof can be found in \cite{pp1994}. Otherwise, we can make the change of variable $\overline{y}_t:= y_t+V_t$ and treat the equivalent BDSDE
\begin{align}
\overline{y}_t =\xi_T +V_T+\Int_t^T f_s(\overline{y}_s-V_s,z_s)ds  +\Int_t^T g_s(\overline{y}_s-V_s,z_s)\cdot d\W_s - \Int_t^T z_s\cdot dB.
\end{align}
\ep
\end{proof}
\noindent We also have a comparison theorem in this context
\begin{Proposition}
Let $\xi_1$ and $\xi_2\in L^2(\P), V^i, i=1,2$ be two adapted c\`adl\`ag processes and $f_s^i(y,z), g_s^i(y,z)$ four functions verifying the above assumption. Let $(y^i,z^i)\in\D^2(\P) \times\H^2(\P)$, $i=1,2$ be the solution of the following BDSDEs
\begin{align*}
y_t^i =\xi_T^i +\Int_t^T f_s^i(y_s^i,z_s^i)ds  +\Int_t^T g_s(y_s^i,z_s^i)\cdot d\W_s +V_T^i-V_t^i- \Int_t^T z_s^i\cdot dB_s ,~\P-a.s,
\end{align*}
respectively. If we have $\P -a.s.$ that $\xi_1\geq\xi_2 , V^1-V^2$ is non decreasing, and $f_s^1(y_s^1,z_s^1)\geq f_s^2(y_s^1,z_s^1)$ then it holds that for all $t\in [0,T]$ $$y_t^1\geq y_t^2, \ \P -a.s.$$
\end{Proposition}
%\begin{Remark}
%If we replace the deterministic time $T$ by a bounded stopping time $\tau$, then all the above is still valid.
%\end{Remark}
\noindent For a given $\G-$stopping time, we now consider the following BDSDE
\begin{align}
y_t =\xi_T +\Int_{t\wedge\tau}^{\tau} f_s(y_s,z_s)ds  +\Int_{t\wedge\tau}^{\tau} g_s(y_s,z_s)\cdot d\W_s +V_{\tau}-V_{t\wedge\tau}- \Int_{t\wedge\tau}^{\tau} z_s\cdot dB_s ~,~\P-a.s.\label{tau BDSDE f-super}
\end{align}
where $\xi\in L^2(\P)$ and $V\in\I^2(\P)$.
\begin{Definition}
If $y$ is a solution of BDSDE of form \reff{tau BDSDE f-super}, the we call $y$ a doubly $f-$super--solution on $[0,\tau]$. If $V\equiv 0$ in $[0,\tau]$, then we call $y$ a doubly $f-$solution.
\end{Definition}
\noindent We now introduce the notion of doubly $f-$(super)martingales.
\begin{Definition}\label{doublymartingale}
\item[$(i)$] A doubly $f-$martingale on $[0,T]$ is a doubly $f-$solution on $[0,T]$.
\item[$(ii)$] A process $(Y_t)$ is a doubly $f-$super--martingale in the strong $($resp. weak$)$ sense if for all stopping time $\tau\leq t$ $($resp. all $t\leq T)$, we have $\E^\P[|Y_{\tau}|^2]< +\infty$ $($resp. $\E^\P[|Y_{t}|^2]< +\infty)$ and if the doubly $f$-solution $(y_s)$ on $[0,\tau]$ $($resp. $[0,t])$ with terminal condition $Y_{\tau}$ $($resp. $Y_{t})$ verifies $y_{\sigma}\leq Y_{\sigma}$ for every stopping time $\sigma\leq \tau$ $($resp. $y_s\leq Y_s$ for every $s\leq t)$.
\end{Definition}
\section{Reflected backward doubly stochastic differential equations}
\label{RBDSDE:section}
In this section, we want to study the problem of a reflected backward doubly stochastic differential equation (RBDSDE for short) with one c\`adl\`ag barrier. This is an extension of the work of Hamad\`ene and Ouknine \cite{HO11} for the standard reflected BSDEs to our case. So in addition to the terminal condition and generators that we used in the previous section, we need
\begin{description}
\item[(iv)] a barrier $\{S_t, 0\leq t\leq T\}$, which is a real-valued c\`adl\`ag $\Fc_t-$measurable process satisfying $S_T\leq\xi$ and
$$\E\left[\underset{0\leq t\leq T}{\Sup}(S_t^+)^2\right] < +\infty.$$
\end{description}
Now we present the definition of the solution of RBDSDEs with one lower barrier.
\begin{Definition}\label{def:rbsde}
We call $(Y,Z,K)$ a solution of the backward doubly stochastic differential equation with one reflecting lower barrier $S(.)$,
terminal condition $\xi$ and coefficients $f$ and $g$, if the following holds:
\begin{itemize}
\item[\rm{(i)}]  $Y\in\D^2(\P),~ Z\in\H^2(\P)$.
 \item[\rm{(ii)}] $Y_t  = \xi +\Int_t^T f(s,Y_s,Z_s)ds +\Int_t^T g(s,Y_s,Z_s)\cdot d\W_s -\Int_t^T Z_s\cdot dB_s +K_T-K_t ,~ 0\leq t\leq T$.
 \item[\rm{(iii)}] $Y_t\geq S_t ~ ,~ 0\leq t\leq T,~ a.s.$
 \item[\rm{(iv)}] If $K^c$ $($resp. $K^d)$ is the continuous $($resp. purely discontinuous$)$ part of $K$, then
 $$\Int_0^T (Y_{s}-S_{s})dK^c_s = 0 ,~ a.s.\ \text{and }\forall t\leq T,\ \Delta K_t^d = (S_{t^-}-Y_t)^+\mathbf{1}_{[Y_{t^-}=S_{t^-}]}.$$
\end{itemize}
\label{RBDSDE}
\end{Definition}
%The state-process $Y(.)$ is forced to remain above the barrier $S(.)$, thanks to the cumulation action of the reflection process $K(.)$,
%which acts only when necessary to prevent $Y(.)$ from crossing the barrier, and in this sense, its action can be considered minimal.
\begin{Remark}
The condition \rm{(iv)} implies in particular that $\Int_0^T (Y_{s^-}-S_{s^-})dK_s = 0.$ Actually
\begin{align*}
\Int_0^T (Y_{s^-}-S_{s^-})dK_s &= \Int_0^T (Y_{s^-}-S_{s^-})dK_s^c+\Int_0^T (Y_{s^-}-S_{s^-})dK_s^d\\
&= \Int_0^T (Y_{s^-}-S_{s})dK_s^c + \sum_{s\leq T}(Y_{s^-}-S_{s^-})\Delta K_s^d=0.
\end{align*}
The last term of the second equality is null since $K^d$ jumps only when $Y_{s^-}=S_{s^-}$. \ep
\end{Remark}
\noindent The main objective of this section is to prove the following theorem.
\begin{Theorem}
 Under the above hypotheses, the RBDSDE in Definition \ref{def:rbsde} has a unique solution $(Y,Z,K)$.
\label{solRBDSDE}
\end{Theorem}
\noindent Before we start proving this theorem, let us establish the same result in the case where $f$ and $g$ do not depend on $y$ and $z$. More precisely, given $f$ and
$g$ such that $$\E\left[\Int_0^T |f(s)|^2ds\right]+\E\left[\Int_0^T \|g(s)\|^2ds\right] < +\infty$$
 and $\xi$ as above, consider the reflected BDSDE
\begin{align}
Y_t  = \xi +\Int_t^T f(s)ds +\Int_t^T g(s)\cdot d\W_s -\Int_t^T Z_s\cdot dB_s +K_T-K_t.
\label{RBDSDE1}
\end{align}
\begin{Proposition}
 There exists a unique triplet $(Y,Z,K)$ verifies conditions of Definition \ref{RBDSDE} and satisfies \reff{RBDSDE1}.
\end{Proposition}
\begin{proof}
\textbf{ a) Existence:} The method combines penalization and the Snell envelope method. For each $n\in\N^*$, we set $$f_n(s,y) = f(s)+ n(y_s-S_s)^-,$$
and consider the BDSDE
\begin{align}
Y_t^n  = \xi^n +\Int_t^T f_n(s,Y_s^n)ds +\Int_t^T g(s)\cdot d\W_s -\Int_t^T Z_s^n\cdot dB_s.
\label{RBDSDE2}
\end{align}
It is well known (see Pardoux and Peng \cite{pp1994}) that BDSDE (\ref{RBDSDE2}) has a unique solution $(Y^n,Z^n)\in\D^2(\P)\times\H^2(\P)$
such that for each $n\in\N,$
$$\E\left[\underset{0\leq t\leq T}{\Sup}|Y_t^n|^2 + \Int_0^T \|Z_s^n\|^2ds\right] <+\infty.$$
From now on the proof will be divided into three steps.

\vspace{0.3em}
\noindent \textit{{\bf{Step 1}}}: For all $n\geq 0$ and $(s,y)\in[0,T]\times\R$,
$$f_n(s,y,z)\leq f_{n+1}(s,y,z),$$ which provide by the comparison theorem, $Y_t^n\leq Y_t^{n+1},~ t\in[0,T]~ a.s.$ For each $n\in\N,$  denoting
$$\bar{Y}_t^n := Y_t^n + \Int_0^t g(s) d\W_s,\ \bar{\xi}:= \xi + \Int_0^T g(s)\cdot d\W_s , \ \bar{S}_t := S_t + \Int_0^t g(s)\cdot d\W_s,$$
we have
\begin{align}
\bar{Y}_t^n= \bar{\xi}+\Int_t^T f(s)ds +n\Int_t^T (\bar{Y}_s^n - \bar{S}_s)^- ds -\Int_t^T Z_s^n\cdot dB_s.
\end{align}
The process $\bar{Y}_t^n$ satisfies
\begin{align}
\forall t\leq T,\ \bar{Y}_t^n= \underset{\tau \geq t}{\rm ess \, sup}\ \E\left[ \left.\Int_t^\tau f(s)ds + (\bar{Y}_\tau^n\wedge \bar{S}_\tau){\bf{1}}_{\{\tau < T\}}+ \bar{\xi}{\bf{1}}_{\{\tau =T\}}\right|\Gc_t\right].
\end{align}
In fact, for any $n\in\N$ and $t\leq T$ we have
\begin{align}
\bar{Y}_t^n= \bar{\xi}+\Int_t^T f(s)ds + n \Int_t^T (\bar{Y}_s^n-\bar{S}_s )^-ds -\Int_t^T \bar{Z}_s^n\cdot dB_s.
\label{penBDSDE}
\end{align}
Therefore for any $\G-$stopping time $\tau\geq t$ we have
\begin{align} \label{penaSnellenv}
\bar{Y}_t^n &=\E\left[\left.\bar{Y}_{\tau}^n+\Int_t^{\tau} f(s)ds + n \Int_t^{\tau} (\bar{Y}_s^n-\bar{S}_s )^-ds\right|\Gc_t\right]\nonumber\\
&\geq  \E\left[\left.(\bar{S}_{\tau}\wedge \bar{Y}_{\tau}^n){\bf{1}}_{[\tau < T]}+ \bar{\xi}{\bf{1}}_{\{\tau = T\}}+\Int_t^{\tau} f(s)ds\right|\Gc_t\right],
\end{align}
since $\bar{Y}_{\tau}^n\geq (\bar{S}_{\tau}\wedge \bar{Y}_{\tau}^n){\bf{1}}_{[\tau < T]}+ \bar{\xi}{\bf{1}}_{\{\tau = T\}}$. On the other hand, let $\tau_t^{*}$ be the stopping time defined as follows:
$$ \tau_t^{*} = \Inf \{s\geq t, \bar{K}_s^n-\bar{K}_t^n > 0\}\wedge T,$$
where $\bar{K}_t^n=n \Int_0^{t} (\bar{Y}_s^n-\bar{S}_s )^-ds$. Let us show that ${\bf{1}}_{[\tau_t^{*} < T]}\bar{Y}_{\tau_t^{*}}^n)= (\bar{S}_{\tau_t^{*}}\wedge \bar{Y}_{\tau_t^{*}}^n){\bf{1}}_{[\tau_t^{*} < T]}$.

\vspace{0.5em}
\noindent Let $\omega$ be fixed such that $\tau_t^{*}(\omega) < T$. Then there exists a sequence $(t_k)_{k\geq 0}$ of real numbers which decreases to $\tau_t^{*}(\omega)$ such that $\bar{Y}_{t_k}^n(\omega)\leq \bar{S}_{t_k}(\omega)$. As $\bar{Y}^n$ and $\bar{S}$ are RCLL processes then taking the limit as $k\rightarrow \infty$ we obtain 
$\bar{Y}_{\tau_t^{*}}^n\leq \bar{S}_{\tau_t^{*}}$ which implies ${\bf{1}}_{[\tau_t^{*} < T]}\bar{Y}_{\tau_t^{*}}^n)= (\bar{S}_{\tau_t^{*}}\wedge \bar{Y}_{\tau_t^{*}}^n){\bf{1}}_{[\tau_t^{*} < T]}$. Now from (\ref{penBDSDE}), we deduce that:
\begin{align*}
\bar{Y}_t^n &= \bar{Y}_{\tau_t^{*}}^n +\Int_t^{\tau_t^{*}} f(s)ds  -\Int_t^{\tau_t^{*}} \bar{Z}_s^n \cdot dB_s\\
&= (\bar{S}_{\tau_t^{*}}\wedge \bar{Y}_{\tau_t^{*}}^n){\bf{1}}_{[\tau_t^{*} < T]}+\bar{\xi}{\bf{1}}_{\{\tau_t^{*} = T\}}+ \Int_t^{\tau_t^{*}} f(s)ds  -\Int_t^{\tau_t^{*}} \bar{Z}_s^n\cdot dB_s.
\end{align*}
Taking the conditional expectation and using inequality (\ref{penaSnellenv}) we obtain: $\forall n\geq 0$, and $t\geq T$
\begin{align}
 \bar{Y}_t^n= \underset{\tau \geq t}{\rm ess \, sup}\ \E\left[\left. \Int_t^\tau f(s)ds + (\bar{Y}_\tau^n\wedge \bar{S}_\tau){\bf{1}}_{\{\tau < T\}}+ \bar{\xi}{\bf{1}}_{\{\tau =T\}}\right|\Gc_t\right].
\end{align}

\vspace{0.3cm}
\noindent \textit{{\bf{Step 2}}}: There exists a RCLL $(Y_t)_{t\leq T}$ of $\D^2(\P)$ such that $\P-a.s.$
\begin{itemize}
\item[(i)] $Y=\underset{n \rightarrow \infty}{\Lim} Y^n$ in $\H^2(\P)$, $S\leq Y$.

\vspace{0.5em}
\item[(ii)] for any $t\leq T,$
\begin{align}\label{Snellenv}
Y_t= \underset{\tau \geq t}{\rm ess \, sup}\ \E\left[\left. \Int_t^\tau f(s)ds + \bar{S}_\tau{\bf{1}}_{\{\tau < T\}}+ \bar{\xi}{\bf{1}}_{\{\tau =T\}}\right|\Gc_t\right]- \Int_0^t g(s)\cdot d\W_s.
\end{align}
\end{itemize}
Actually for $t\leq T$ let us set 
$$\tilde{Y}_t:= \underset{\tau \geq t}{\rm ess \, sup}\ \E\left[\left. \Int_t^\tau f(s)ds + \bar{S}_\tau{\bf{1}}_{\{\tau < T\}}+ \bar{\xi}{\bf{1}}_{\{\tau =T\}}\right|\Gc_t\right].$$
since $\bar{S}\in \D^2(\P)$, $f\in \H^2(\P)$ and $\bar{\xi}$ is square integrable, the process $\tilde{Y}$ belongs to $\D^2(\P)$. On the other hand for any $n\geq 0$ and $t\leq T$ we have $\bar{Y}_t^n\leq \tilde{Y}_t$. Thus there exist a $\G-$progressively measurable process $\bar{Y}$ such that $\P-a.s.$, for any $t\leq T, \ \bar{Y}_t^n\ \nearrow \bar{Y}_t\leq \tilde{Y}_t $ and we have $Y_t^n\nearrow Y_t= \bar{Y}_t-\Int_0^t g(s)\cdot d\W_s$, then $Y=\underset{n \rightarrow \infty}{\Lim} Y^n$ in $\H^2(\P).$ 

\vspace{0.5em}
\noindent Besides, the process $\bar{Y}_\cdot^n+\Int_0^\cdot f(s)ds$ is a c\`adl\`ag super--martingale as the Snell envelope of 
$$\left(\int_0^{\cdot} f(s)ds+\bar{S}_{\cdot}\wedge \bar{Y}_{\cdot}^n\right){\bf{1}}_{[\cdot< T]}+ \bar{\xi}{\bf{1}}_{\{\cdot = T\}},$$ and it converges increasingly to $\bar{Y}_\cdot+\int_0^\cdot f(s)ds$. It follows that the latter process is a c\`adl\`ag super--martingale. Hence, the process $Y$ is also $\G-$progressively measurable, c\`adl\`ag, and belongs to $\D^2(\P)$. Even more than that, $Y_t$ is $\Fc_t$-measurable for every $t\in[0,T]$ as the limit of  $Y_t^n$, which has this property.

\vspace{0.3em}
\noindent Next let us prove that $Y\geq S$. We have
$$\E[Y_0^n]=\E\left[\xi+\Int_0^Tf(s)ds\right]+\E\left[\Int_0^T n(Y_s^n-S_s)^-ds\right].$$
Dividing the two sides by $n$ and taking the limit as $n\rightarrow\infty$, we obtain 
$$\E\left[\Int_0^T (Y_s-S_s)^-ds\right]=0.$$ Since the processes $Y$ and $S$ are c\`adl\`ag, then, $\P-a.s.$, $Y_t\geq S_t, $ for $t< T$. But $Y_T=\xi\geq S_T$, therefore $Y\geq S$.

\vspace{0.3em}
\noindent Finally let us show that $Y$ satisfies (\ref{Snellenv}). But this is a direct consequence of the continuity of the Snell envelope through sequences of increasing c\`adl\`ag processes. In fact on the one hand, the sequence of increasing c\`adl\`ag processes $((\bar{S}_{t}\wedge \bar{Y}_{t}^n){\bf{1}}_{[t < T]}+ \bar{\xi}{\bf{1}}_{\{t = T\}})_{t\leq T})_{t\leq T}$ converges increasingly to the c\`adl\`ag process $(\bar{S}_{t}{\bf{1}}_{[t < T]}+ \bar{\xi}{\bf{1}}_{\{t = T\}})_{t\leq T}){[t\leq T}$ since $\bar{Y}_{t}\geq \bar{S}_{t}$. Therefore, 
$$\Int_0^tf(s)ds+\bar{Y}_{t}^n\longrightarrow \underset{\tau \geq t}{\rm ess \, sup}\ \E\left[\left. \Int_0^\tau f(s)ds + \bar{S}_\tau{\bf{1}}_{\{\tau < T\}}+ \bar{\xi}{\bf{1}}_{\{\tau =T\}}\right|\Gc_t\right]=\Int_0^tf(s)ds+\bar{Y}_{t},$$
which implies that 
\begin{align*}
Y_t &=\bar{Y}_t-\Int_0^t g(s)\cdot d\W_s= \underset{\tau \geq t}{\rm ess \, sup}\ \E\left[\left. \Int_t^\tau f(s)ds + \bar{S}_\tau{\bf{1}}_{\{\tau < T\}}+ \bar{\xi}{\bf{1}}_{\{\tau =T\}}\right|\Gc_t\right]- \Int_0^t g(s)\cdot d\W_s.
\end{align*}

\vspace{0.3em}
\noindent 
\textit{{\bf{Step 3}}}: We know from (\ref{Snellenv}) that the process $\int_0^\cdot f(s)ds+\bar{Y}_{\cdot}^n$ is a Snell envelope. Then, there exist a process $K\in\I^2(\P)$ and a $\G$-martingale such that
$$ \Int_0^tf(s)ds+Y_t+ \Int_0^t g(s)\cdot d\W_s = M_t-K_t ,\ 0\leq t\leq T.$$
Additionally $K=K^c+K^d$ where $K^c$ is continuous, non-decreasing and $K^d$ non-decreasing purely discontinuous predictable such that for any $t\leq T, \Delta_t K^d= (S_{t^-}-Y_t){\bf{1}}_{\{Y_{t^-}=S_{t^-}\}}$.
Now the martingale $M$ belongs to $\D^2(\P)$, so that the It\^o's martingale representation theorem implies the existence of a $\G$-predictable process $Z\in\H^2(\P)$ such that
$$M_t=M_0+\Int_0^t Z_s\cdot dB_s, \quad 0\leq t\leq T,\ \P-a.s.$$
Hence $$Y_t=Y_0- \Int_0^tf(s)ds-\Int_0^t g(s)\cdot d\W_s+\Int_0^t Z_s\cdot dB_s-K_t ,\ 0\leq t\leq T.$$
The proof of  $\Int_0^T (Y_{s}-S_{s})dK^c_s = 0$ is the same as in \cite{HO11}, so we omit it.

\vspace{0.3em}
\noindent It remains to show that ${Z_t}$ and ${K_t}$ are in fact $\Fc_t-$measurable. For $K_t$, it is obvious since it is the limit of $K_t^n= \Int_0^t n(Y_s^n-S_s)^- ds$ which is $\Fc_t-$measurable for each $t\leq T$. Now 
$$ \Int_t^T Z_s\cdot dB_s = \xi +\Int_t^T f(s,Y_s,Z_s)ds +\Int_t^T g(s,Y_s,Z_s)\cdot d\W_s - Y_t+K_T-K_t,$$
and the right side is $\Fc_T^B\vee\Fc_{t,T}^{W}$-measurable. Hence from the It\^o's martingale representation theorem $(Z_s)_{t\leq s\leq T}$ is $\Fc_s^B\vee\Fc_{t,T}^{W}$ adapted. Consequently $Z_s$ is $\Fc_s^B\vee\Fc_{t,T}^{W}$-measurable for any $t<s$, so it is $\Fc_s^B\vee\Fc_{s,T}^{W}$-measurable.

\vspace{0.3em}
\vspace{0.3em}
\noindent \textbf{b) Uniqueness:} Under Lipschitz continuous conditions, the proof of uniqueness is standard in BSDE theory (see e.g. proof of Proposition 2.1. in \cite{AM10}).
\ep
\end{proof}

\vspace{0.5em}
\noindent The existence of solution of RBDSDE in Theorem \ref{solRBDSDE} is obtained via a standard fixed Banach point theorem for reflected BSDEs (see for instance El Karoui, Hamad\`ene and Matoussi \cite{ElkMH08}).

\end{appendix}
%\newpage
%*****************************************
%**************bibliography***************
%*****************************************
\bibliographystyle{acm}
\bibliography{Biblio-2BDSDE}

\begin{thebibliography}{10}

\bibitem{AM10}
{\sc Aman, A., and Mrhardy, N.}
\newblock Obstacle problem for {SPDE} with onlinear {N}eumann boundary
  condition via reflected generalized backward doubly {SDE}s.
\newblock {\em Statistics \& Probability Letters 83}, 3 (2013), 863--874.

\bibitem{ALP95}
{\sc Avellaneda, M., Levy, A., and Paras, A.}
\newblock Pricing and hedging derivative securities in markets with uncertain
  volatility.
\newblock {\em Applied Mathematical Finance\/} (1995).

\bibitem{BGM15}
{\sc Bachouch, A., Gobet, E., and Matoussi, A.}
\newblock Empirical regression method for backward doubly stochastic
  differential equations.
\newblock {\em SIAM/ASA Journal on Uncertainty Quantification 4}, 1 (2016),
  358--379.

\bibitem{matouetal13}
{\sc Bachouch, A., Lasmar, A.~B., Matoussi, A., and Mnif, M.}
\newblock Numerical scheme for semilinear {SPDE}s via backward doubly {SDE}s.
\newblock {\em Stochastic Partial Differential Equations: Analysis and
  Computation 1\/}, 1--43.

\bibitem{BM01}
{\sc Bally, V., and Matoussi, A.}
\newblock Weak solutions for {SPDE}s and backward doubly stochastic
  differential equations.
\newblock {\em Journal of Theoretical Probability 14}, 1 (2001), 125--164.

\bibitem{BS78}
{\sc Bertsekas, D., and Shreve, S.}
\newblock {\em Stochastic optimal control: the discrete-time case}.
\newblock Academic Press, New York, 1978.

\bibitem{B81}
{\sc Bichteler, K.}
\newblock Stochastic integration and {$L^{p}-$}theory of semimartingales.
\newblock {\em Annals of Probability 9}, 1 (1981), 49--89.

\bibitem{buck2011pathwise}
{\sc Buckdahn, R., Bulla, I., and Ma, J.}
\newblock Pathwise {T}aylor expansions for {I}t{\=o} random fields.
\newblock {\em Mathematical Control and Related Fields 1}, 4 (2011), 437--468.

\bibitem{buck:ma:10a}
{\sc Buckdahn, R., and Ma, J.}
\newblock Stochastic viscosity solutions for nonlinear stochastic partial
  differential equations. {I}.
\newblock {\em Stochastic Processes and their Applications 93}, 2 (2001),
  181--204.

\bibitem{buck:ma:10b}
{\sc Buckdahn, R., and Ma, J.}
\newblock Stochastic viscosity solutions for nonlinear stochastic partial
  differential equations. {II}.
\newblock {\em Stochastic Processes and their Applications 93}, 2 (2001),
  205--228.

\bibitem{buckdahn2002pathwise}
{\sc Buckdahn, R., and Ma, J.}
\newblock Pathwise stochastic {T}aylor expansions and stochastic viscosity
  solutions for fully nonlinear stochastic {PDE}s.
\newblock {\em The Annals of Probability 30}, 3 (2002), 1131--1171.

\bibitem{buckdahn2007pathwise}
{\sc Buckdahn, R., and Ma, J.}
\newblock Pathwise stochastic control problems and stochastic {HJB} equations.
\newblock {\em SIAM Journal on Control and Optimization 45}, 6 (2007),
  2224--2256.

\bibitem{buckdahn2015pathwise}
{\sc Buckdahn, R., Ma, J., and Zhang, J.}
\newblock Pathwise {T}aylor expansions for random fields on multiple
  dimensional paths.
\newblock {\em Stochastic Processes and their Applications 125}, 7 (2015),
  2820--2855.

\bibitem{Buck:Ma:Zhan:15}
{\sc Buckdahn, R., Ma, J., and Zhang, J.}
\newblock Pathwise viscosity solutions of stochastic {PDE}s and forward
  path--dependent {PDE}s.
\newblock {\em arXiv preprint arXiv:1501.06978\/} (2015).

\bibitem{Caru:Friz:Ober:11}
{\sc Caruana, M., Friz, P., and Oberhauser, H.}
\newblock A (rough) pathwise approach to a class of non--linear stochastic
  partial differential equations.
\newblock {\em Annales de l'institut Henri Poincar{\'e}, Analyse Non
  Lin{\'e}aire $(${\rm C}$)$ 28}, 1 (2011), 27--46.

\bibitem{CP2000}
{\sc Chen, Z., and Peng, S.}
\newblock A general downcrossing inequality for {$g-$}martingales.
\newblock {\em Statistics \& Probability Letters 46}, 2 (2000), 169--175.

\bibitem{CSTV07}
{\sc Cheridito, P., Soner, H., Touzi, N., and Victoir, N.}
\newblock Second--order backward stochastic differential equations and fully
  nonlinear parabolic {PDE}s.
\newblock {\em Communications on Pure and Applied Mathematics 60}, 7 (2007),
  1081--1110.

\bibitem{CIL92}
{\sc Crandall, M., Ishii, H., and Lions, P.-L.}
\newblock User's guide to viscosity solutions of second order partial
  differential equations.
\newblock {\em Bulletin of the American Mathematical Society 27}, 1 (1992),
  1--67.

\bibitem{DKN07}
{\sc Dalang, R., Khoshnevisan, D., and Nualart, E.}
\newblock Hitting probabilities for systems of non--linear stochastic heat
  equations with additive noise.
\newblock {\em ALEA. Latin American Journal of Probability and Mathematical
  Statistics 3\/} (2007), 231--271.

\bibitem{dawson1972stochastic}
{\sc Dawson, D.}
\newblock Stochastic evolution equations.
\newblock {\em Mathematical Biosciences 15}, 3 (1972), 287--316.

\bibitem{dm}
{\sc Dellacherie, C., and Meyer, P.}
\newblock {\em Probabilit\'{e}s et Potentiel, Chapitres XII \`a XVI, Th\'eorie
  du potentiel}.
\newblock Hermann, Paris, 1980.

\bibitem{deni:mart:06}
{\sc Denis, L., and Martini, C.}
\newblock A theoretical framework for the pricing of contingent claims in the
  presence of model uncertainty.
\newblock {\em The Annals of Applied Probability 16}, 2 (2006), 827--852.

\bibitem{diehl:friz:12}
{\sc Diehl, J., and Friz, P.}
\newblock Backward stochastic differential equations with rough drivers.
\newblock {\em The Annals of Probability 40}, 4 (2012), 1715--1758.

\bibitem{Doob84}
{\sc Doob, J.~L.}
\newblock {\em Classical potential theory and its probabilistic counterpart}.
\newblock Classics in Mathematics. Springer-Verlag, Berlin, 2001.
\newblock Reprint of the 1984 edition.

\bibitem{ElkMH08}
{\sc {\uppercase{e}l}~Karoui, N., Hamad\`ene, S., and Matoussi, A.}
\newblock Backward stochastic differential equations and applications.
\newblock {\em Chapter 8 in the book Indifference Pricing: Theory and
  Applications, Springer-Verlag\/} (2008), 267--320.

\bibitem{elkarouitan1}
{\sc El~Karoui, N., and Tan, X.}
\newblock Capacities, measurable selection and dynamic programming part {I}:
  abstract framework.
\newblock {\em arXiv preprint arXiv:1310.3363\/} (2013).

\bibitem{elkarouitan2}
{\sc El~Karoui, N., and Tan, X.}
\newblock Capacities, measurable selection and dynamic programming part {II}:
  application in stochastic control problems.
\newblock {\em arXiv preprint arXiv:1310.3364\/} (2013).

\bibitem{F84}
{\sc Fremlin, D.~H.}
\newblock {\em Consequences of {M}artin's axiom}, vol.~84 of {\em Cambridge
  Tracts in Mathematics}.
\newblock Cambridge University Press, Cambridge, 1984.

\bibitem{friz2016eikonal}
{\sc Friz, P., Gassiat, P., Lions, P.-L., and Souganidis, P.}
\newblock Eikonal equations and pathwise solutions to fully non-linear {SPDE}s.
\newblock {\em arXiv preprint arXiv:1602.04746\/} (2016).

\bibitem{GGK15}
{\sc Gerencs{\'e}r, M., Gy{\"o}ngy, I., and Krylov, N.}
\newblock On the solvability of degenerate stochastic partial differential
  equations in {S}obolev spaces.
\newblock {\em Stoch. Partial Differ. Equ. Anal. Comput. 3}, 1 (2015), 52--83.

\bibitem{Gubi:Tind:Torr:14}
{\sc Gubinelli, M., Tindel, S., and Torrecilla, I.}
\newblock Controlled viscosity solutions of fully nonlinear rough {PDE}s.
\newblock {\em arXiv preprint arXiv:1403.2832\/} (2014).

\bibitem{GK10}
{\sc Gy{\"o}ngy, I., and Krylov, N.}
\newblock Accelerated finite difference schemes for linear stochastic partial
  differential equations in the whole space.
\newblock {\em SIAM J. Math. Anal. 42}, 5 (2010), 2275--2296.

\bibitem{GK11}
{\sc Gy{\"o}ngy, I., and Krylov, N.}
\newblock Accelerated numerical schemes for {PDE}s and {SPDE}s.
\newblock In {\em Stochastic analysis 2010}. Springer, Heidelberg, 2011,
  pp.~131--168.

\bibitem{HO11}
{\sc Hamad{\`e}ne, S., and Ouknine, Y.}
\newblock Reflected backward {\sc{sde}}s with general jumps.
\newblock {\em Theory of Probability \& Its Applications 60}, 2 (2015),
  357--376.

\bibitem{ichikawa1978linear}
{\sc Ichikawa, A.}
\newblock Linear stochastic evolution equations in {H}ilbert space.
\newblock {\em Journal of Differential Equations 28}, 2 (1978), 266--277.

\bibitem{K95}
{\sc Karandikar, R.}
\newblock On pathwise stochastic integration.
\newblock {\em Stochastic Processes and Their Applications, 57:11-18\/} (1995).

\bibitem{kazi2015second}
{\sc Kazi-Tani, N., Possama{\"\i}, D., and Zhou, C.}
\newblock Second order {BSDE}s with jumps: existence and probabilistic
  representation for fully--nonlinear {PIDE}s.
\newblock {\em Electronic Journal of Probability 20\/} (2015).

\bibitem{krylov1977cauchy}
{\sc Krylov, N., and Rozovski{\u\i}, B.}
\newblock On the {C}auchy problem for linear stochastic partial differential
  equations.
\newblock {\em Izvestiya: Mathematics 11}, 6 (1977), 1267--1284.

\bibitem{krylov1981stochastic}
{\sc Krylov, N., and Rozovski{\u\i}, B.}
\newblock Stochastic evolution equations.
\newblock {\em Journal of Soviet Mathematics 16}, 4 (1981), 1233--1277.

\bibitem{K90}
{\sc Kunita, H.}
\newblock {\em Stochastic flows and stochastic differential equations}, vol.~24
  of {\em Cambridge Studies in Advanced Mathematics}.
\newblock Cambridge University Press, Cambridge, 1990.

\bibitem{L14}
{\sc Lin, Y.}
\newblock A new existence result for second--order {BSDE}s with quadratic
  growth and their applications.
\newblock {\em Stochastics: An International Journal of Probability and
  Stochastic Processes 88}, 1 (2016), 128--146.

\bibitem{lion:soug:98}
{\sc Lions, P.-L., and Souganidis, P.~E.}
\newblock Fully nonlinear viscosity stochastic partial differential equations:
  non-smooth equations and applications.
\newblock {\em C.R. Acad. Sci. Paris 327}, 1 (1998), 735--741.

\bibitem{lion:soug:00}
{\sc Lions, P.-L., and Souganidis, P.~E.}
\newblock \'{E}quations aux d\'eriv\'ees partielles stochastiques
  nonlin\'eaires et solutions de viscosit\'e.
\newblock {\em S\'eminaire \'equations aux d\'eriv\'ees partielles,
  1998--1999}, 1 (2000), 1--13.

\bibitem{lion:soug:01}
{\sc Lions, P.-L., and Souganidis, P.~E.}
\newblock Viscosity solutions of fully nonlinear stochastic partial
  differential equations.
\newblock {\em S\=urikaisekikenky\=usho K\=oky\=uroku}, 1287 (2002), 58--65.
\newblock Viscosity solutions of differential equations and related topics
  (Japanese) (Kyoto, 2001).

\bibitem{L95}
{\sc Lyons, T.~J.}
\newblock Uncertain volatility and the risk--free synthesis of derivatives.
\newblock {\em Applied Mathematical Finance 2\/} (1995), 117--133.

\bibitem{MS02}
{\sc Matoussi, A., and Sheutzow, M.}
\newblock Semilinear stochastic {PDE}'s with nonlinear noise and backward
  doubly {SDE}'s.
\newblock {\em Journal of Theoretical Probability 15\/} (2002), 1--39.

\bibitem{N12}
{\sc Nutz, M.}
\newblock Pathwise construction of stochastic integrals.
\newblock {\em Electronic Communications in Probability 17}, 24 (2012), 1--7.

\bibitem{Nutz12}
{\sc Nutz, M.}
\newblock A quasi--sure approach to the control of non--{M}arkovian stochastic
  differential equations.
\newblock {\em Electronic Journal of Probability 17}, 23 (2012), 1--23.

\bibitem{OP89}
{\sc Ocone, D., and Pardoux, E.}
\newblock A generalized it\^{o}--ventzell formula. application to a class of
  anticipating stochastic differential equations.
\newblock {\em Annales de l'institut Henri Poincar{\'e}, Probabilit{\'e}s et
  Statistiques $(${\rm B}$)$ 25}, 1 (1989), 39--71.

\bibitem{pardouxt1980stochastic}
{\sc Pardoux, {\'E}.}
\newblock Stochastic partial differential equations and filtering of diffusion
  processes.
\newblock {\em Stochastics 3}, 1-4 (1980), 127--167.

\bibitem{PP90}
{\sc Pardoux, {\'E}., and Peng, S.}
\newblock Adapted solution of a backward stochastic differential equation.
\newblock {\em Systems \& Control Letters 14}, 1 (1990), 55--61.

\bibitem{pp1994}
{\sc Pardoux, {\'E}., and Peng, S.}
\newblock Backward doubly {\sc{sde}}'s and systems of quasilinear {\sc{spde}}s.
\newblock {\em Probab. Theory and Related Field 98\/} (1994), 209--227.

\bibitem{PP87}
{\sc Pardoux, {\'E}., and Protter, P.}
\newblock A two--sided stochastic integral and its calculus.
\newblock {\em Probab. Theory and Related Field 76}, 1 (1987), 15--49.

\bibitem{Peng97}
{\sc Peng, S.}
\newblock Backward {SDE} and related {$g-$}expectation.
\newblock In {\em Backward stochastic differential equations}, N.~El~Karoui and
  L.~Mazliak, Eds., vol.~364 of {\em Pitman research notes in mathematics}.
  Longman, Harlow, 1997, pp.~141--159.

\bibitem{peng:99}
{\sc Peng, S.}
\newblock Monotonic limit theorem of {BSDE} and nonlinear decomposition theorem
  of {D}oob--{M}eyer's type.
\newblock {\em Probability Theory and Related Fields 113}, 4 (1999), 473--499.

\bibitem{poss13}
{\sc Possama{\"{\i}}, D.}
\newblock Second order backward stochastic differential equations under a
  monotonicity condition.
\newblock {\em Stochastic Processes and their Applications 123}, 5 (2013),
  1521--1545.

\bibitem{possamai2015weak}
{\sc Possama{\"\i}, D., and Tan, X.}
\newblock Weak approximation of second--order {BSDE}s.
\newblock {\em The Annals of Applied Probability 25}, 5 (2015), 2535--2562.

\bibitem{PTZ14}
{\sc Possama{\"\i}, D., Tan, X., and Zhou, C.}
\newblock Stochastic control for a class of non--linear stochastic kernels and
  applications.
\newblock {\em arXiv preprint arXiv:1510.08439\/} (2015).

\bibitem{PZ13}
{\sc Possama{\"{\i}}, D., and Zhou, C.}
\newblock Second order backward stochastic differential equations with
  quadratic growth.
\newblock {\em Stochastic Process. Appl. 123}, 10 (2013), 3770--3799.

\bibitem{ren2015convergence}
{\sc Ren, Z., and Tan, X.}
\newblock On the convergence of monotone schemes for path--dependent {PDE}.
\newblock {\em arXiv preprint arXiv:1504.01872\/} (2015).

\bibitem{SGL05}
{\sc Shi, Y., Gu, Y., and Liu, K.}
\newblock Comparison theorems of backward doubly stochastic differential
  equations and applications.
\newblock {\em Stochastic Analysis and its Applications 23}, 1 (2005), 97--110.

\bibitem{sone:touz:zhan:11b}
{\sc Soner, H., Touzi, N., and Zhang, J.}
\newblock Martingale representation theorem for the {$G-$}expectation.
\newblock {\em Stochastic Processes and their Applications 121}, 2 (2011),
  265--287.

\bibitem{sone:touz:zhan:11a}
{\sc Soner, H., Touzi, N., and Zhang, J.}
\newblock Quasi--sure stochastic analysis through aggregation.
\newblock {\em Electronic Journal of Probability 16}, 67 (2011), 1844--1879.

\bibitem{sone:touz:zhan:13}
{\sc Soner, H., Touzi, N., and Zhang, J.}
\newblock Dual formulation of second order target problems.
\newblock {\em The Annals of Applied Probability 23}, 1 (2013), 308--347.

\bibitem{STZ10}
{\sc Soner, H.~M., Touzi, N., and Zhang, J.}
\newblock Wellposedness of second order backward {SDE}s.
\newblock {\em Probability Theory and Related Fields 153}, 1-2 (2012),
  149--190.

\bibitem{SY78}
{\sc Stricker, C., and Yor, M.}
\newblock Calcul stochastique d\'ependant d'un param\`etre.
\newblock {\em Zeitschrift f\"ur Wahrscheinlichkeitstheorie und Verwandte
  Gebiete 45}, 2 (1978), 109--133.

\bibitem{SV79}
{\sc Stroock, D., and Varadhan, S.}
\newblock {\em Multidimensional diffusion processes}.
\newblock Springer, 1979.

\end{thebibliography}

\end{document}